\documentclass[a4paper, final,leqno,onefignum,onetabnum,10pt]{amsart}
\usepackage{verbatim}
\usepackage{graphicx}
\usepackage{enumerate}
\usepackage{xspace}
\usepackage[ruled]{algorithm}
\usepackage{subfigure}
\usepackage{booktabs}
\usepackage{tabularx}
\usepackage[square,sort,compress]{natbib}%numbers,for just numbered
\usepackage{times}
\usepackage{amsmath,amssymb,amsxtra}
\usepackage{subfigure}
\usepackage{geometry}                % See geometry.jpg to learn the layout options. There are lots.

\usepackage{commath}
\usepackage{tikz}
\numberwithin{equation}{section}

\providecommand{\highlight}[1]{{\color{red}#1}}
\usepackage{ifthen}
         \provideboolean{showchanges}
         \setboolean{showchanges}{true}
         \provideboolean{newchanges}
         \setboolean{newchanges}{true}
         \providecommand{\changes}[1]{
           \ifthenelse{\boolean{showchanges}}{{\highlight{#1}}}{#1}
         }
 \providecommand{\newchanges}[1]{
           \ifthenelse{\boolean{newchanges}}{{\highlight{#1}}}{#1}
         }
         \providecommand{\changefromto}[3][replace with]{
           \ifthenelse{\boolean{showchanges}}
           {{\sout{#2}\margnote{#1}}{\highlight{#3}}}
           {#3\xspace}
         }
         \providecommand{\ChangePar}[2]{
           \ifthenelse{\boolean{showchanges}}
           {{\par$\mapsfrom$ \textcolor{red!20}{#1}}{\par$\mapsto$ \textcolor{blue}{#2}}}
           {\par #2}
         }
         \providecommand{\InsertPar}[1]{
           \ifthenelse{\boolean{showchanges}}
           {{\par$\mapsto$ \textcolor{blue}{#1}}}
           {\par #1}
         }

\provideboolean{showdelete}
\setboolean{showdelete}{false}%%false
\newcommand{\delete}[1]{
  \ifthenelse{\boolean{showdelete}} {{\color{red}{#1}}}{}
}

\swapnumbers
\numberwithin{equation}{section}
\newtheorem{The}[subsection]{Theorem}

\newtheorem{Lem}[subsection]{Lemma}

\newtheorem{Prop}[subsection]{Proposition}

\theoremstyle{note}

\newtheorem{Rem}[subsection]{Remark}

\newtheorem{Defn}[subsection]{Definition}

\newtheorem{Pbm}[subsection]{Problem}

\numberwithin{equation}{section}

\newcommand{\Program}[1]{\textsf{#1}\xspace}
\newcommand{\ALBERTA}{\Program{ALBERTA}}
\newcommand{\PARAVIEW}{\Program{PARAVIEW}}

\renewcommand{\vec}[1]{\ensuremath{\boldsymbol{#1}}}

\newcommand{\myall}{\ensuremath{\quad \mbox{ for all }}}
\newcommand{\pdt}{\ensuremath{\partial_t}}

\newcommand{\pdtau}{\ensuremath{\bar{\partial}_\tau}}

\newcommand{\Reals}{\ensuremath{{\mathbb{R}}}}

\newcommand{\diff}{\ensuremath{{\operatorname{d}}}}
\newcommand{\dd}{\mathop{}\, \diff}
\renewcommand{\O}{\ensuremath{{\Omega}}}

\newcommand{\normal}{\ensuremath{{\vec{\nu}}}}

\newcommand{\G}{\ensuremath{{\Gamma}}}

\newcommand{\Gc}{\ensuremath{{\G_h}}}

\newcommand{\Lp}[1]{\ensuremath{{L}^{#1}}}

\newcommand{\Hil}[1]{\ensuremath{{H}^{#1}}}
\newcommand{\dual}[1]{\ensuremath{\left({#1}\right)^\prime}}

\newcommand{\Sc}{\ensuremath{{\mathbb{S}_h}}}

\newcommand{\Vc}[1]{\ensuremath{{\mathbb{V}_h^{#1}}}}
\newcommand{\Kc}[1]{\ensuremath{{\mathbb{K}_h^{#1}}}}
\newcommand{\dualp}[4]{\ensuremath{{}_{#1}\left\langle#2,#3\right\rangle}_{#4}}
\newcommand{\ltwon}[2]{\ensuremath{\left\|#1\right\|}_{\Lp{2}\left(#2\right)}}

\newcommand{\Lpn}[3]{\ensuremath{\left\|#2\right\|}_{\Lp{#1}(#3)}}

\newcommand{\lv}{\ensuremath{\left\vert}}
\newcommand{\rv}{\ensuremath{\right\vert}}
\usepackage{upgreek}
\newcommand{\lap}{\ensuremath{\Updelta}}

\newcommand{\T}{\ensuremath{{\mathcal{T}}}}

\newcommand{\Clement}{\ensuremath{{\mathcal{I}^h}}}
\newcommand{\Lagrange}{\ensuremath{{\Lambda^h}}}
\definecolor{MyGreen}{rgb} {0.05,0.4,0.05}
\definecolor{RedViolet}{rgb} {0.1,0.1,0.75}

          \providecommand{\highlight}[1]{{\color{blue}#1}}
         \provideboolean{shownotes}
          \setboolean{shownotes}{true}%%false
          \newcommand{\standout}[1]{\colorbox{red}{\textcolor{white}{#1}}}
          \newcounter{margnote}[page]
          \newcommand{\margnotemark}{{\standout{\footnotesize\upshape\texttt{\arabic{margnote}}}}}
          \newcommand{\margnote}[2][]{
            \ifthenelse{
              \boolean{shownotes}
            }{\stepcounter{margnote}\margnotemark\marginpar{
                \texttt{
                  \begin{minipage}{2cm}
                    \raggedright\tiny
                    \margnotemark{#1}:
                    #2
                  \end{minipage}
            }}}{}
          }
          \provideboolean{showchanges}
          \setboolean{showchanges}{false}
          \providecommand{\chcolor}{\color{blue}}
          \providecommand{\changes}[1]{
            \ifthenelse{\boolean{showchanges}}
                     {{\chcolor #1}}
                     {#1
          }
          }%end\providecommand
          \providecommand{\changes}[1]{
            \ifthenelse{\boolean{showchanges}}{{\highlight{#1}}}{#1}
          }
          \providecommand{\changefromto}[3][replace with]{
            \ifthenelse{\boolean{showchanges}}
            {{\sout{#2}\margnote{#1}}{\highlight{#3}}}
            {#3\xspace}
          }
          \providecommand{\ChangePar}[2]{
            \ifthenelse{\boolean{showchanges}}
            {{\par$\mapsfrom$ \textcolor{red!20}{#1}}{\par$\mapsto$ \textcolor{blue}{#2}}}
            {\par #2}
          }
          \providecommand{\InsertPar}[1]{
            \ifthenelse{\boolean{showchanges}}
            {{\par$\mapsto$ \textcolor{blue}{#1}}}
            {\par #1}
         }
\usepackage[british]{babel}
\newcommand \beq{\begin{equation}}
\newcommand \eeq{\end{equation}}

\def\eps{{\varepsilon}}
\def\Du{{\mathbb{D}u}}

\def\bpmat{\begin{pmatrix}}
\def\epmat{\end{pmatrix}}
\def\mathref#1{\ifmmode\mathrm{(\ref{#1})}\else{\rm(\ref{#1})}\fi}
\def\nref#1{\ifmmode\mathrm{\ref{#1}}\else{\rm\ref{#1}}\fi}

\newcommand\ubar{\bar{u}}
\newcommand\vbar{\bar{v}}
\def\tu{{\tilde{u}}}
\def\tw{{\tilde{w}}}
\def\tv{{\tilde{v}}}
\def\tz{{\tilde{z}}}
\def\hu{{\hat{u}}}
\def\hw{{\hat{w}}}
\def\hv{{\hat{v}}}
\def\hz{{\hat{z}}}
\newcommand{\delO}{\ensuremath{\delta_\Omega}}
\newcommand{\delG}{\ensuremath{\delta_\Gamma}}
\newcommand{\delk}{\ensuremath{\delta_k}}
\newcommand{\sgn}{\mathrm{sgn}}

\begin{document}
\bibliographystyle{abbrvnat}
\title{Coupled bulk-surface free boundary problems arising from a mathematical model of receptor-ligand dynamics}%
\author{%
  Charles M.\ Elliott}
  \address[C.~M.~Elliott]{Mathematics Institute, Zeeman Building, University of Warwick, Coventry, UK, CV4 7AL.}
  \email[C.~M.~Elliott]{C.M.Elliott@warwick.ac.uk}
  \author{Thomas Ranner}
    \address[T.~Ranner]{School of Computing, E.C. Stoner Building, University of Leeds, Leeds, UK. LS2 9JT.}
  \email[T.~Ranner]{T.Ranner@leeds.ac.uk}
  \author{Chandrasekhar Venkataraman}
    \address[C.~Venkataraman]{Mathematical Institute, North Haugh, University of St Andrews, Fife, UK. KY16 9SS.}
  \email[C.~Venkataraman]{cv28@st-andrews.ac.uk}

\maketitle

%%%%%%%%%%%%%%%%%%%%%%%%%%%%%%%%%%%%%%%%%%%%%%%%%%%%%%%%
%%%%%%%%%%%%%%%%%%%%%%%%%%%%%%%%%%%%%%%%%%%%%%%%%%%%%%%%%

\begin{abstract}
We consider a coupled bulk-surface system of partial differential equations with nonlinear coupling modelling receptor-ligand dynamics. The model arises 
as a simplification of a mathematical model for the reaction between cell surface resident receptors and ligands present in the extra-cellular medium.
We prove the existence and uniqueness of solutions. We also consider a number of biologically relevant asymptotic limits of the model. We prove 
convergence to limiting problems which take the form of free boundary problems posed on the cell surface. We also report on numerical simulations 
illustrating convergence to one of the limiting problems as well as the spatio-temporal distributions of the receptors and ligands in a realistic geometry.
\end{abstract}

\pagestyle{myheadings}
\thispagestyle{plain}
\markboth{C.M. Elliott, T. Ranner and C. Venkataraman}{Coupled bulk-surface free boundary problems arising from a mathematical model of receptor-ligand dynamics}

%\section*{Reports}
\section{Introduction}\label{sec:intro}
We start by outlining the mathematical model for receptor-ligand dynamics whose analysis and asymptotic limits will be the main 
focus of this work.
Let $\Gamma$ be a smooth, compact closed $n$-dimensional hypersurface contained in  {the interior of a simply 
connected} domain $D \subset \Reals^{n+1}$, $n=1,2$.
The surface $\Gamma$ separates the domain $D$ into an interior domain $I$ and an exterior domain $\Omega$. We will denote 
by $\partial_0 \Omega$ the outer boundary of $\Omega$, i.e.\ the boundary $\partial D$. {The vectors $\normal$ and 
$\normal_\Omega$ denote the outward pointing unit normals to $\O$ on $\G$ and $\partial_0\O$ respectively.
 Fig.~\ref{fig:domain} shows a cartoon sketch of the setup.
We  assume that the outer boundary $\partial_0 \Omega$  is Lipschitz.
We consider the following problem: Find $u \colon \bar{\Omega} \times [0,T) \to \mathbb{R}^+$ and $w \colon \Gamma \times [0,T) \to \mathbb{R}^+$ such that
\begin{subequations}
  \label{eq:eps-problem}
  \begin{align}
    \delO \partial_{{t}} {u} - \Delta {u} & = 0 && \mbox{ in } \Omega\times(0,T)  \\
    \nabla {u} \cdot \vec\nu & = - \frac{1}{\delk} {u} {w} && \mbox{ on } \Gamma\times(0,T)  \\
    {u} = {u}_D & \mbox{ or } \nabla {u} \cdot \vec{\nu}_{\Omega} = 0 && \mbox{ on } \partial_0 \Omega\times(0,T)  \\
    \partial_{{t}} {w} - \delG \Delta_{\Gamma} {w} & =  \nabla {u} \cdot \vec\nu && \mbox{ on } \Gamma\times(0,T)  \\
    u(\cdot,0)&=u^0(\cdot) && \mbox{ in } \O \\
    w(\cdot,0)&=w^0(\cdot) && \mbox{ on } \G ,
  \end{align}
\end{subequations}
}
where $\delO, \delG,\delk >0$ are given model parameters and the initial data are bounded, non-negative 
functions, i.e., $u^0\in\Lp{\infty}(\O)$, $w^0\in\Lp{\infty}(\G)$ and $u^0,w^0\geq 0$. In the above $\lap_\G$ denotes 
the Laplace-Beltrami operator on the surface $\G$ and $\lap$ the usual Cartesian Laplacian in $\Reals^{n+1}$. 

We will use either Dirichlet or Neumann boundary conditions on $\partial_0 \Omega$.
For the Dirichlet case, we assume that the Dirichlet boundary data $u_D$ is a positive scalar constant. Our analysis remains valid 
if we consider bounded positive functions
for the Dirichlet boundary data, we restrict the discussion to positive scalar boundary data for the sake of simplicity. The restriction 
to non-negative solutions is made since we are interested in biological problems where $u$ and $w$ represent chemical concentrations and hence are non-negative.

\begin{figure}
  \centering
  \begin{center}
    \includegraphics{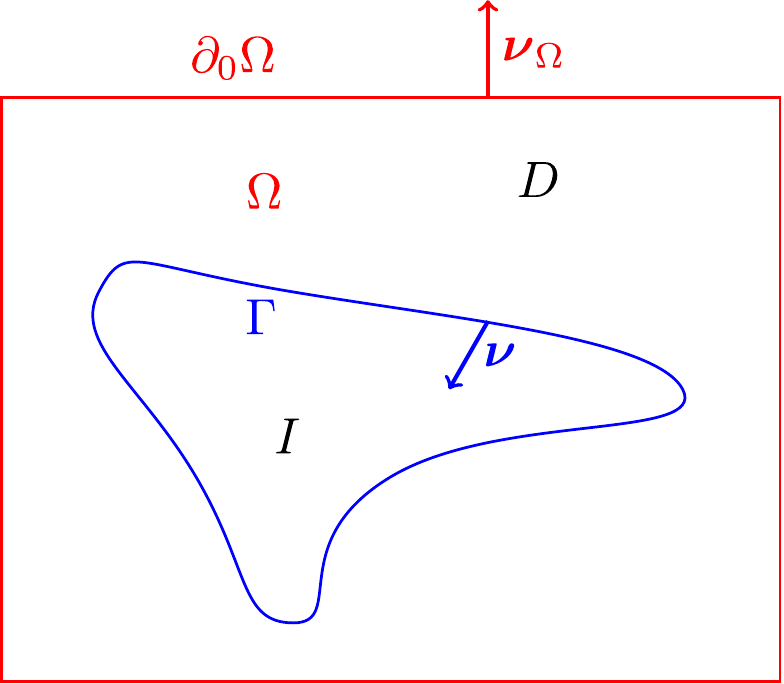}
  \end{center}
  \caption{A sketch of the cell membrane $\G$ and the extra-cellular medium $\O$.}
  \label{fig:domain}
\end{figure}

Problem (\ref{eq:eps-problem})  may be regarded as a basic model for receptor-ligand dynamics in cell biology,
modelling the dynamics of mobile cell surface receptors reacting with a mobile bulk ligand, which is a reduction of the 
model (\ref{eq:k-problem}) presented in \S \ref{sec:motivation-limit-problems}. 
Receptor-ligand interactions and the associated cascades of activation of signalling molecules, so called signalling cascades,  
are the primary mechanism by which cells sense and respond to their environment. Such processes therefore constitute a 
fundamental part of many basic phenomena in cell biology such as proliferation, motility, the maintenance of structure or form, 
adhesion, cellular signalling, etc. \citep{bongrand1999ligand,hynes1992integrins,locksley2001tnf}. Due to the complexity of the biochemistry 
involved in signalling networks, an integrated approach combining theoretical and computational  mathematical studies with experimental and modelling efforts appears necessary. 
Motivated by this need, in this work we focus on understanding a mathematical reduction of  theoretical models for receptor-ligand dynamics in cell 
biology consisting of a coupled system of bulk-surface partial differential equations (PDEs). 

A number of recent theoretical  and computational studies of receptor-ligand interactions, \citep[e.g.,][]{marciniak2008derivation, garcia2013mathematical}, 
employ models which are similar in structure to those considered in this work. Models with similar features arise in the modelling of signalling networks 
coupling the dynamics of ligands  within the cell (e.g., G-proteins) with those on the cell surface {\citep{levine2005membrane, jilkine2007mathematical,mori2008wave,ratz2012turing,ratz2013symmetry,MadChuVen14,bao2014well,morgan2015global}}. The ability of cells to 
create their own chemotactic gradients, i.e., to influence the bulk ligand field,  has been conjectured to play a crucial role in collective directed migration 
for example during neural crest formation {\citep{mclennan2012multiscale,mclennan2015vegf,mclennan2015neural}} and hence 
understanding such models is of much biological importance.

Through proving well-posedness results, this work gives a mathematically sound foundation for the use and simulation of coupled bulk-surface models for receptor-ligand dynamics. 
Moreover, we justify the consideration of various small parameter asymptotic limits of such models, through non-dimensionalisation using experimentally measured parameter values. 
We provide a rigorous derivation of the limiting problems and discuss their well-posedness. We also discuss the numerical solution of the original and limiting problems 
illustrating the asymptotic convergence together with robust and efficient methods for their approximation.  This work suggests that models for receptor-ligand dynamics 
featuring fast reaction kinetics can be derived using classical elements of free boundary methodology as components of the modelling.

Whilst our focus is on receptor-ligand dynamics, problems of a similar structure arise in fields such as ecology where one considers populations consisting of two or more 
competing species \citep{holmes1994partial}. Such a scenario can be modelled by so-called spatial segregation models and the corresponding asymptotic limits have been the subject 
of much mathematical study, \citep[e.g.,][]{conti2005asymptotic,crooks2004spatial,dancer1999spatial}.
Further details on the cell-biological motivation for studying (\ref{eq:eps-problem}), together with the limits $\delO,\delG,\delk\to0$, is given in \S \ref{sec:motivation-limit-problems}.

The main focus of this work is to show the system of partial differential equations (\ref{eq:eps-problem}) is well posed and so is meaningful from the 
mathematical perspective and,
furthermore, to obtain reduced models as limits of this system  as we send the parameters $\delO,\delG$ and $\delk$ to zero.  Specifically, we establish 
existence and 
uniqueness of a solution to (\ref{eq:eps-problem}) and show that in the limits $\delk\to0$, $\delO,\delG>0$ fixed, $\delG=\delk\to0$, $\delO>0$ fixed, 
$\delO=\delG=\delk\to0$, this solution to (\ref{eq:eps-problem}) converges to a solution of suitably defined limit problems. Furthermore, in the latter two 
cases, $\delG=\delk\to0$ and $\delO=\delG=\delk\to0$, the uniqueness of the solution to the limit problems, respectively constrained parabolic and elliptic  
problems with dynamic boundary conditions, is also shown. We then show that the limit problems with dynamic boundary conditions may be reformulated as 
variational inequalities and briefly explore some connections with classical free boundary problems. These reduced models in the form of free boundary 
problems may be considered as models in their own right and offer simplifications with respect to numerical computation.

That the fast reaction limit ($\delk\to0$) leads to interesting free boundary problems is  because of the complementarity nature of the resulting limit
\begin{align*}
 {u}\ge 0, && w\ge 0, && {u} {w} =0 && \mbox{ on } \Gamma.
\end{align*}
Such limits have been considered for  coupled systems of parabolic equations (posed in the same domain) in a number of previous 
works  \citep[e.g.,][]{evans1980ctc,bothe2001instantaneous,bothe2012instantaneous} with the limiting problem corresponding to a 
Stefan problem  \citep{hilhorst1996fast,hilhorst2001competition,hilhorst2003vanishing}. 
Here in this paper the main complication in the analysis  is that the species reside in different domains and the coupling is on the boundary of the bulk domain 
which results in added technical complications in passing to the limit.

For the limit problems $\delG=\delk=0$ and $\delG=\delk=\delO=0$ with dynamic boundary conditions, we obtain Stefan and Hele-Shaw type problems 
on the  hypersurface $\Gamma$  with a differential operator, which may be interpreted as a non-local fractional differential operator, obtained by using the 
Dirichlet to Neumann map for the bulk parabolic and elliptic operators. This leads to an interesting variational inequality reformulation  in the  case of the limit bulk 
elliptic equation consisting of  a boundary obstacle problem  that is satisfied by the integral in time of the solution. The approach follows that employed for the 
reformulation of the one-phase Stefan problem and the Hele-Shaw problem  for which the transformed variable (integral in time of the solution) satisfies a 
parabolic \citep{duvaut1973resolution} or elliptic \citep{ell80, elliott1981variational} variational inequality respectively.  

Problems related to those considered in this work have been the focus of recent studies.
For example, { \citet{morgan2015global} consider coupled bulk-surface systems of parabolic equations with nonlinear coupling in which the surface 
resident species are defined on the boundary of the bulk domain.  They derive sufficient conditions  on the coupling to ensure global existence of classical solutions 
extending the results of \citet{pierre2010global} , from the planar case to the coupled bulk-surface case.}
 \citet{2013arXiv1302.5026S} consider the well posedness of singular heat equation with dynamic boundary conditions of reactive-diffusive type (i.e., including 
 the Laplace-Beltrami of the trace of the solution on the boundary).
\citet{bao2014well} consider a reaction-diffusion equation in a bulk domain coupled to a reaction-diffusion equation posed on the boundary.
They prove existence and uniqueness of a weak solution to the problem and establish exponential convergence to equilibrium. 
\citet{vazquez2008heat,vazquez2009laplace,vazquez2011heat} study the well posedness of the Laplace and heat equations with dynamic boundary conditions of 
reactive- and reactive-diffusive type. The heat equation with nonlinear dynamic Neumann boundary conditions which arises in problems of boundary 
heat control is considered by  \citet{athanasopoulos2010continuity}. 
The authors  prove continuity of the solution and furthermore, they extend their results to the case where the heat operator in the interior is replaced 
with a fractional diffusion operator.  Existence  and uniqueness of weak solutions to Hele-Shaw problems which are  Stefan-type free boundary problems with  
vanishing specific heat    are considered by  \citet{crowley1979weak}. Elliptic equations with non-smooth nonlinear dynamic boundary conditions have been studied in 
a number of applications. \citet{aitchison1984model} propose a simplified model for an electropaint process that consists of an elliptic equation with nonlinear 
dynamic boundary conditions involving the normal derivative. The authors  formally derive the steady state stationary problem which consists of a Signorini 
problem similar to the elliptic variational inequality we derive in \S \ref{sec:var-ineq}. 
 This problem is studied by \citet{caffarelli1985nonlinear} where the authors prove that the steady state solution ($t\to\infty$) of  an implicit time discretisation 
 solves the Signorini problem proposed as the formal limit by \citet{aitchison1984model}. A similar problem, which models percolation in gently sloping beaches, 
 that  consists of an elliptic equation variational inequality with dynamic boundary conditions  involving the normal derivative is proposed and 
 analysed by \citet{aitchison1983percolation,elliott1985analysis,colli1987nonlinear}. \citet{perthame2014hele} 
 derive Hele-Shaw type free boundary problems as limits of models for tumour growth.  Finally we mention the work of \citet{nochetto2014convergence} who 
 consider the numerical approximation of obstacle problems, in particular, they prove optimal convergence rates for the thin obstacle (Signorini) problem and 
 prove quasi-optimal convergence rates for the approximation of the obstacle problem for the fractional Laplacian.
 
Our main results are stated in Theorems \ref{thm:eps-problem}, \ref{thm:parab-pbm}, \ref{thm:PdBC-pbm} and \ref{thm:EdBC-pbm}.
\begin{itemize}
\item In Theorem \ref{thm:eps-problem} we establish the existence of a unique, bounded solution to (\ref{eq:eps-problem}).
\item In Theorem \ref{thm:parab-pbm} we present a rigorous derivation that in the limit $\delk\to0$, $\delO,\delG>0$ fixed, 
the solution to (\ref{eq:eps-problem}) converges to a solution of a system of constrained coupled bulk-surface parabolic equations (c.f., (\ref{eq:parabolic-problem})).
\item In Theorem \ref{thm:PdBC-pbm} we present a rigorous derivation that in the limit $\delG=\delk\to0$, with $\delO>0$ fixed, 
the solution to (\ref{eq:eps-problem}) converges to the unique solution of constrained parabolic problem with dynamic boundary condition (c.f., (\ref{eq:PdBC-limit-problem})).
\item In Theorem \ref{thm:EdBC-pbm} we present a rigorous derivation that in the limit $\delO=\delG=\delk\to0$, the solution to (\ref{eq:eps-problem}) 
converges to the unique solution of constrained elliptic problem with dynamic boundary condition (c.f., (\ref{eq:EdBC-limit-problem})).
\end{itemize}

We conclude the paper by providing some numerical experiments employing a coupled bulk-surface finite element method where we support 
numerically the theoretical convergence results to a limiting problem  and investigate the resulting free boundary problem on a surface.

%%%%%%%%%%%%%%%%%%%%%%%%%%%%%%%%%%%%%%%%%%%%%%%%%%%%%
\section{Biological motivation}\label{sec:motivation-limit-problems}

We now present a  model for receptor-ligand dynamics and justify, through non-dimensionalisation of the model using parameter values 
previously measured in experimental studies, the simplifications and limiting problems considered in this work. 

{
  % \color{MyGreen} \verb!add!
We start with the following model, that corresponds to one of the models presented by \citet{garcia2013mathematical} if one 
neglects the terms involving internalisation of receptors and complexes.
}
 The reaction under consideration is between mobile receptors that reside on the  cell surface with ligands present in the extra-cellular 
 medium (the bulk region surrounding the cell).  We assume a single species of mobile surface (cell membrane) resident receptor whose 
 concentration (surface density) is denoted by $c_r$ and a single species of bulk resident diffusible ligand whose concentration (bulk concentration) 
 is denoted by $c_L$. The receptor and ligand react reversibly on the surface to form a (surface resident, mobile) receptor-ligand complex, whose 
 concentration is denoted by $c_{rl}$. The kinetic constants $k_{\text{on}}$ and $k_{\text{off}}$ represent the forward and reverse reaction rates. 
 Denoting by $\G$ the cell surface and by $\O$ the extra-cellular medium with outer boundary $\partial_0\O$ (c.f., Fig.~\ref{fig:domain}), we have in mind models of the following form,
 {
\begin{subequations}
  \label{eq:k-problem}
  \begin{align}
    \pdt c_L - D_L \lap c_L & =0\quad && \text{ in }\O\times(0,T)\\
    D_L\nabla c_L\cdot \normal & =-k_{\text{on}}c_L c_r+k_{\text{off}}c_{rl} && \text{ on }\G\times(0,T)\\
    c_L=c_D & \mbox{ or } D_L \nabla c_L \cdot \vec{\nu}_\Omega = 0 && \text{ on }\partial_0\O\times(0,T)\\
    \pdt c_r- D_r\lap_\G c_r & =-k_{\text{on}}c_L c_r+k_{\text{off}}c_{rl} && \text{ on }\G\times(0,T)\\
    \pdt c_{rl}- D_{rl}\lap_\G c_{rl}& =k_{\text{on}}c_L c_r-k_{\text{off}}c_{rl} && \text{ on }\G\times(0,T).
  \end{align}
\end{subequations}%
}%
The model is closed by suitable (bounded, non-negative) initial conditions.  For the outer boundary condition we take either a 
Dirichlet or a Neumann boundary condition.
The Dirichlet boundary condition, with $c_D>0$ a positive constant, arises under the modelling assumption that the 
background concentration of ligands sufficiently far away from the cell is uniform.
Alternatively, the Neumann boundary condition arises from assuming zero flux across $\partial_0 \Omega$.

\subsection{Non-dimensionalisation and limit problems}\label{sec:limit-model}

{%
We are interested in different limit problems arising from model \eqref{eq:k-problem} for ligand-receptor binding.
To simplify notation, we write the unknowns as $u = c_L, w = c_r$ and $\chi = c_{rl}$ and the parameters 
$D_\Omega = D_L$, $D_\Gamma = D_r = D_{rl}$, $k = k_{\text{on}}$, $k_{-1} = k_{\text{off}}$.

The first problem we consider is to find $u \colon \Omega \times [0,T) \to \mathbb{R}$ and $w, \chi \colon \Gamma \times [0,T) \to \mathbb{R}$ such that
\begin{subequations}
  \label{eq:full-pb}
  \begin{align}
    \partial_t u - D_\Omega \Delta u & = 0 && \mbox{ in } \Omega \times (0,T) \\
    D_\Omega \nabla u \cdot \vec{\nu} & = - k u w + k_{-1} \chi && \mbox{ on } \Gamma \times (0,T)  \\
    u = u_D & \mbox{ or } D_\Omega \nabla u \cdot \vec{\nu}_\Omega = 0 && \mbox{ on } \partial_0 \Omega  \times (0,T) \\
    \partial_t w - D_\Gamma \Delta_\Gamma w & = D_\Omega \nabla u \cdot \vec\nu && \mbox{ on } \Gamma  \times (0,T) \\
    \partial_t \chi - D_\Gamma \Delta_\Gamma \chi & = -D_\Omega \nabla u \cdot \vec\nu && \mbox{ on } \Gamma  \times (0,T) \\
    u(\cdot,0)&=u^0(\cdot) && \mbox{ in } \O \\
    w(\cdot,0)&=w^0(\cdot) && \mbox{ on } \G \\
    \chi(\cdot,0)&=\chi^0(\cdot) && \mbox{ on } \G.
  \end{align}
\end{subequations}

In order to determine the sizes of each coefficient, we take the following rescaling. We set
\begin{align*}
  \tilde{x} = x / L, \quad
  \tilde{t} = t / S, \quad
  \tilde{u} = u / U, \quad
  \tilde{w} = w / W, \quad
  \tilde{\chi} = \chi / X,
\end{align*}
where $L$ is a length scale, $S$ is a time scale, $U, W$ and $X$ are typical concentrations for $u, w$ and $\chi$ respectively.

Applying the chain rule, this leads to a non-dimensional form of (\ref{eq:full-pb}):

\begin{subequations}
  \label{eq:nondim-pb}
  \begin{align}
    \delta_\Omega \partial_{\tilde{t}} \tilde{u} - \Delta \tilde{u} & = 0 && \mbox{ in } \tilde\Omega \times ( 0, \tilde{T} ) \\
    \nabla \tilde{u} \cdot \vec\nu & = - \frac{1}{\delta_k} \tilde{u} \tilde{w} + \delta_{\chi} \tilde{\chi} && \mbox{ on } \tilde\Gamma \times ( 0, \tilde{T} ) \\
    \tilde{u} = \tilde{u}_D & \mbox{ or } \nabla \tilde{u} \cdot \vec{\nu}_{\tilde\Omega} = 0 && \mbox{ on } \partial_0 \tilde\Omega \times ( 0, \tilde{T} ) \\
    \partial_{\tilde{t}} \tilde{w} - \delta_\Gamma \Delta_{\tilde\Gamma} \tilde{w} & = \mu \nabla \tilde{u} \cdot \vec\nu && \mbox{ on } \tilde\Gamma \times ( 0, \tilde{T} ), \\
    \partial_{\tilde{t}} \tilde{\chi} - \delta_\Gamma \Delta_{\tilde\Gamma} \tilde{\chi} & = -\mu' \nabla \tilde{u} \cdot \vec\nu && \mbox{ on } \tilde\Gamma \times ( 0, \tilde{T} ) \\
    \tilde{u}(\cdot,0)&= \tilde{u}^0(\cdot) := u^0(\cdot) / U && \mbox{ in } \O \\
    \tilde{w}(\cdot,0)&= \tilde{w}^0(\cdot) :=  w^0(\cdot) / W && \mbox{ on } \G \\
    \tilde{\chi}(\cdot,0)&= \tilde{\chi}^0(\cdot) := \chi^0(\cdot) / X && \mbox{ on } \G.
  \end{align}
\end{subequations}
Here we have six non-dimensional coefficients:
\begin{gather*}
  \delta_\Omega = \frac{L^2}{D_\Omega S}, \quad
  \delta_k = \frac{D_\Omega}{k L W}, \quad
  \delta_\chi = \frac{k_{-1} L X}{D_L U}, \quad
  \delta_\Gamma = \frac{D_\Gamma S}{L^2}, \\
  \mu = \frac{D_\Omega S U}{L W}, \quad
  \mu' = \frac{D_\Omega S U}{L X}.
\end{gather*}

Taking values from Table~\ref{tab:parameters}, we infer that
\begin{gather*}
  \delta_\Omega = (5.6 \, \mathrm{s}) \cdot S^{-1}, \quad
  \delta_k = 5.7 \cdot 10^{-2}, \quad
  % \delta_\chi = (3.8 \cdot 10^6 \, \mathrm{m}^2 \, \mathrm{mol}^{-1} ) \cdot X, \\
  \delta_\chi = 8.7 \cdot 10^{-2}, \\
  \delta_\Gamma = (1.8 \cdot 10^{-4} \, \mathrm{s}^{-1} ) \cdot S, \quad
  \mu = (5.7 \cdot 10^{-2} \, \mathrm{s}^{-1} ) \cdot S, \quad
  \mu' = (5.7 \cdot 10^{-2} \, \mathrm{s}^{-1} ) \cdot S.
\end{gather*}

\begin{table}
  \centering
  \begin{tabular}{ccc}
    \hline
    Parameter & Value & Source \\
    \hline
    $L$ & $7.5 \cdot 10^{-6} \, \mathrm{m}$ & \citet{garcia2013mathematical} \\
    $U$ & $1.0 \cdot 10^{-3} \, \mathrm{mol} \, \mathrm{m}^{-3}$ & \citet{garcia2013mathematical} \\
    $W$ & $2.3 \cdot 10^{-8} \, \mathrm{mol} \, \mathrm{m}^{-2}$ & \citet{garcia2013mathematical} \\
    $X$ & $2.3 \cdot 10^{-8} \, \mathrm{mol} \, \mathrm{m}^{-2}$ & limited by total receptor concentration \\
    $D_\Omega$ & $1.0 \cdot 10^{-11} \, \mathrm{m}^2 \, \mathrm{s}^{-1}$ & \citet{LINDERMAN1986295} \\
    $D_\Gamma$ & $1.0 \cdot 10^{-15} \, \mathrm{m}^2 \, \mathrm{s}^{-1}$ & \citet{LINDERMAN1986295} \\
    $k_{\text{on}}$ & $1.0 \cdot 10^3 \, \mathrm{m}^3 \, \mathrm{mol}^{-1} \, \mathrm{s}^{-1}$ & \citet{garcia2013mathematical} \\
    $k_{\text{off}}$ & $5.0 \cdot 10^{-3} \mathrm{s}^{-1}$ & \citet{garcia2013mathematical} \\
    \hline
  \end{tabular}

  \caption{Parameters used for rescaling equations. The values for $U$ and $W$ are extreme values taken from within a physical range from \citet{garcia2013mathematical}.}
  \label{tab:parameters}
\end{table}

First, we note that $\delta_\chi \ll 1$. Considering the limit $\delta_\chi \to 0$ by dropping the terms $\delta_\chi \chi$ decouples the equations for $\tilde{u}, \tilde{w}$ from the equation for $\tilde{\chi}$. This results in the problem, which we have written in terms of the original variables:
\begin{subequations}
  \label{eq:nochi-problem}
  \begin{align}
    \partial_t u - \delta_\Omega^{-1} \Delta u & = 0 && \mbox{ in } \Omega \\
    \nabla {u} \cdot \vec\nu & = - \frac{1}{\delk} {u} {w} && \mbox{ on } \Gamma\times(0,T)  \\
    u = u_D & \mbox{ or } \nabla u \cdot \vec{\nu}_\Omega = 0 && \mbox{ on } \partial_0 \Omega \\
    \partial_t w - \delta_\Gamma \Delta_\Gamma w & = \mu \nabla u \cdot \vec\nu && \mbox{ on } \Gamma\\
        u(\cdot,0)&=u^0(\cdot) && \mbox{ in } \O \\
    w(\cdot,0)&=w^0(\cdot) && \mbox{ on } \G ,
  \end{align}
\end{subequations}
This is the first problem we consider in Section~\ref{sec:model}. Similar methods to those shown in the remaining sections can be used to show 
well posedness of the system \eqref{eq:nondim-pb} and rigorously take the limit $\delta_\chi \to 0$ for $\delta_k, \delta_\Omega, \delta_\Gamma, \mu > 0$ fixed.
The existence and uniqueness theory of \eqref{eq:nondim-pb} and the limit to obtain \eqref{eq:nochi-problem}  in the more general case of time dependent domains are considered by \citet{AlpEllTer-pp}.
}

We see that $\delta_k \ll 1$. {Again using} the original variables, we consider the limit problem: Find $u \colon \Omega \times {[0,T)} \to \mathbb{R}$ and $w \colon \Gamma \times {[0,T)} \to \mathbb{R}$ such that
{
\begin{subequations}
  \label{eq:fast-reaction-pb}
  \begin{align}
    \partial_t u - \delta_\Omega^{-1} \Delta u & = 0 && \mbox{ in } \Omega \\
    \label{eq:fast-reaction-pb-constraint}
    u w & = 0 && \mbox{ on } \Gamma \\
    u = u_D & \mbox{ or } \nabla u \cdot \vec{\nu}_\Omega = 0 && \mbox{ on } \partial_0 \Omega \\
    \partial_t w - \delta_\Gamma \Delta_\Gamma w & = \mu \nabla u \cdot \vec\nu && \mbox{ on } \Gamma\\
        u(\cdot,0)&=u^0(\cdot) && \mbox{ in } \O \\
    w(\cdot,0)&=w^0(\cdot) && \mbox{ on } \G ,
  \end{align}
\end{subequations}
}
where $\delta_\Omega, \delta_\Gamma$ and $\mu$ are positive parameters. We consider this problem as a large ligand--receptor binding rate limit of \eqref{eq:full-pb}.
We consider the well posedness of the problem and the justification of the limit in Section~\ref{sec:limit_parabolic}.

We can consider different problems by choosing different time scales $S$.
We can achieve two different problems by resolving the timescale of the volumetric diffusion ($\delta_\Omega \approx 1$) or the timescale of the surface adsorption flux ($\mu \approx 1$).

% { \color{RedViolet} \verb!remove!
% For $S = L^2 / D_\Omega = 56 \, \mathrm{s}$, we have
% %
% \begin{align*}
%   \delta_\Omega = 1, \qquad
%   \delta_\Gamma = 2.0 \cdot 10^{-3} \ll 1, \qquad
%   \mu = 3.2 \cdot 10^{-1} \approx 1.
% \end{align*}
% }
{
  % \color{MyGreen} \verb!add!
For $S = L^2 / D_\Omega = 5.6 \, \mathrm{s}$, we have
\begin{align*}
  \delta_\Omega = 1, \qquad
  \delta_\Gamma = 1.0 \cdot 10^{-3} \ll 1, \qquad
  \mu = 3.2 \cdot 10^{-1} \approx 1.
\end{align*}
}
This leads to {a} parabolic limit problem with dynamic boundary condition: Find $u \colon \Omega \times [0,T) \to \mathbb{R}$ and $w \colon \Gamma \times [0,T) \to \mathbb{R}$ such that
{
\begin{subequations}
\label{eq:nd_PdBC}
  \begin{align}
    \partial_t u - \delta_\Omega^{-1} \Delta u & = 0 && \mbox{ in } \Omega \\
    u w & = 0 && \mbox{ on } \Gamma \\
    u & = u_D && \mbox{ on } \partial_0 \Omega \\
    \partial_t w & = \nabla u \cdot \vec\nu && \mbox{ on } \Gamma\\
        u(\cdot,0)&=u^0(\cdot) && \mbox{ in } \O \\
    w(\cdot,0)&=w^0(\cdot) && \mbox{ on } \G.
  \end{align}
\end{subequations}
}
In this case, we have resolved the timescale of the diffusion of ligand, but the effect of the diffusion of surface bound receptor is lost.
We consider the well posedness of this problem and the justification of the limit in Section~\ref{sec:parab-limit-pb}.

% {
%  \color{RedViolet} \verb!remove!
% Alternatively, taking $S = 10^3 \, \mathrm{s}$, we have
% \begin{align*}
%   \delta_\Omega = 5.6 \cdot 10^{-2} \ll 1, \qquad
%   \delta_\Gamma = 3.6 \cdot 10^{-2} \ll 1, \qquad
%   \mu = 5.7 \approx 1.
% \end{align*}
% }
{
  % \color{MyGreen} \verb!add!
Alternatively, taking $S = 10^2 \, \mathrm{s}$, we have
\begin{align*}
  \delta_\Omega = 5.7 \cdot 10^{-2} \ll 1, \qquad
  \delta_\Gamma = 1.8 \cdot 10^{-2} \ll 1, \qquad
  \mu = 5.7 \approx 1.
\end{align*}
}
This leads to an elliptic problem with dynamic boundary condition: Find $u \colon \Omega \times [0,T) \to \mathbb{R}$ and $w \colon \Gamma \times [0,T) \to \mathbb{R}$ such that
{
\begin{subequations}
\label{eq:nd_EdBC}
  \begin{align}
    -\Delta u & = 0 && \mbox{ in } \Omega \\
    u w & = 0 && \mbox{ on } \Gamma \\
    u & = u_D && \mbox{ on } \partial_0 \Omega \\
    \partial_t w & = \nabla u \cdot \vec\nu && \mbox{ on } \Gamma\\
    w(\cdot,0)&=w^0(\cdot) && \mbox{ on } \G.
  \end{align}
\end{subequations}
}
In this regime, we {have} chosen a time scale so that the diffusion of ligand has no memory of its previous value, except via the boundary condition.
This means this problem no longer requires an initial condition for $u$ to be a closed system.
We do not consider the exterior Neumann boundary condition in this case, since we arrive at a trivial problem where the solution is $u = 0$ and $w = w_0$.
The well posedness of this problem and a rigorous justification of limit is given in Section~\ref{sec:elliptic-limit-pb}.
We also show in Section~\ref{sec:var-ineq} that we can reformulate problems (\ref{eq:nd_PdBC}) and (\ref{eq:nd_EdBC}) by integrating forwards in time to derive variational inequalities.

{
  \begin{Rem}
    \label{rem:limits}
    In the large ligand-receptor binding rate limit,
    the nonlinear constraint \eqref{eq:fast-reaction-pb-constraint} ($u w = 0$) implies that the domain $\Gamma$ is separated into two regions, for positive times, where $u = 0$ and where $u > 0$.
    In the region $u > 0$, we have a Neumann boundary condition $\nabla u \cdot \vec\nu = 0$. This can be interpreted that there is no flux of ligand onto or off the surface in this region.
    In the region $u = 0$, we have a Dirichlet boundary condition $(u=0)$. This can be interpreted that the ligand in this region is perfectly (i.e.\ instantaneously) absorbed.
  \end{Rem}
}

%%%%%%%%%%%%%%%%%%%%%%%%%%%%%%%%%%%%%%%%%%%%%%%%%%%%%
\section{Preliminaries}\label{sec:prelim}
In this section, we define some of our notation and collect some technical results that will be used in the subsequent sections. We also prove some compact embedding results in Lemmas \ref{lem:H12L2-compact}, \ref{lem:L2H-12-compact} and \ref{lem:trace-compact} that are used to deduce strong convergence from weak convergence in suitable spaces.

%\subsection{Surface fractional Sobolev spaces and trace inequalities}
Given a Hilbert space $Y$ we denote the dual space of a linear functionals on $Y$ by $\dual{Y}$. 
As we consider functions defined on surfaces, along with the surface function spaces $L^2(\Gamma)$ and $H^1(\Gamma)$, we will also use the space $H^{1/2}( \Gamma )$ and its dual $\dual{H^{1/2}( \Gamma )}$. For a Hilbert space $Y$, we consistently use the notation $\dualp{}{\cdot}{\cdot}{Y}$ to denote the duality pairing between the space $Y$ and its dual $\dual{Y}$.

\begin{Defn}
  The space $H^{1/2}( \Gamma )$ is defined by
  \begin{equation}
    H^{1/2}( \Gamma ) := \left\{ \xi \in L^2( \Gamma ) : \norm{ \xi }_{H^{1/2}( \Gamma ) } < + \infty \right\},
  \end{equation}
  where
  \begin{equation}
    \norm{ \xi }_{H^{1/2}( \Gamma ) }
    := \left(
      \int_{\Gamma} \xi^2 \dd \sigma
      + \int_{\Gamma} \int_{\Gamma} \frac{| \xi(x) - \xi(y) |^2}{ | x - y |^{n} }
      \dd \sigma(x) \dd \sigma(y)
      \right)^{\frac{1}{2}}.
  \end{equation}
\end{Defn}

The space can be characterised via the following result.

\begin{Prop}[Trace Theorem]
  The trace operator from $H^1( \Omega )$ to $H^{1/2}( \Gamma )$ is bounded and surjective.
\end{Prop}

\begin{proof}
  The result can be found in \cite[Thm~1.5.1.3]{grisvard2011elliptic}.
\end{proof}

We recall the following interpolated trace inequality.

\begin{Prop}
  [Interpolated trace theorem]
  \label{prop:interpolated-trace}
  For all $\phi\in\Hil{1}(\O)$  and for any $\delta>0$
  \begin{equation}
    \label{eqn:interp_trace}
    \ltwon{\phi}{\G}^2\leq \delta\ltwon{\nabla\phi}{\O}^2+c_\delta\ltwon{\phi}{\O}^2.
  \end{equation}
\end{Prop}

\begin{proof}
  See e.g.\ \cite[Thm 1.5.1.10]{grisvard2011elliptic}.
\end{proof}

Note that for $\xi \in L^2( \Gamma )$ and $\rho\in\Hil{1/2}(\G)$ the following duality pairing is equal to $L^2(\Gamma)$ inner-product:
\begin{equation}
  \label{eq:dual12-L2inner-prod}
  \langle \xi, \rho \rangle_{H^{1/2}(\Gamma)}
  = \int_\Gamma \xi \rho \dd \sigma
\end{equation}

%We conclude this section with a statement of the surface Sobolev embedding results for the dimensions we are interested in.
%
%\begin{Prop}
%  \label{prop:sobolev-embeddings}
%  %
%  The following embeddings hold and are continuous:
%  %
%  \begin{align}
%    \Hil{1/2}(\G) & \hookrightarrow\Lp{q}(\G) && \mbox{ for } 1 \le q \le 4 \\
%    %
%    \Hil{1}(\G) & \hookrightarrow\Lp{q}(\G) && \mbox{ for } 1 \le q < \infty.
%  \end{align}
%\end{Prop}
%
%\begin{proof}
%  \margnote{Hebey does not do fractional case}The result is shown in \citep[Thm.~2.6]{hebey1999nonlinear} for $\Gamma$ with dimension $1$ or $2$.
%\end{proof}

\subsection{Compact embeddings}

Since we are dealing with nonlinear problems, we will need to use some compact embeddings of Bochner spaces.

We recall that if $\{ f_k \}$ is a sequence of bounded functions in $L^p( 0, T; B )$, with $B$ a Banach space, for $1 \le p < \infty$, then there exists a subsequence $\{ f_{k_j} \} \subset \{ f_k \}$ and $f \in L^p( 0, T; B )$ such that
\begin{equation}
  \label{eq:weak-conv}
  f_{k_j} \rightharpoonup f \qquad \mbox{ in } L^p( 0, T; B ).
\end{equation}
Here we are interested to show under what conditions we may assert the existence of a strongly convergent subsequence.
The basic results we require are summarised by \citet{simon1986compact}.

\begin{Lem}
  [Aubin-Lions-Simons compactness theory {\citep{simon1986compact}}]
  \label{lem:als-compact}
  Let $\{ f_k \}$ be a bounded sequence of functions in $L^p( 0, T; B )$ where $B$ is a Banach space and $1 \le p \le \infty$. If
  \begin{enumerate}
  \item the sequence of functions $\{ f_k \}$ is bounded in $L^p( 0, T; X )$ where $X$ is compactly embedded in $B$;
  \item either
    \begin{enumerate}
    \item the derivatives $\{ \partial_t f_k \}$ are bounded in the space $L^p( 0, T; Y )$ where $B \subset Y$; or
    \item for each $k$, the time translates of $\{ f_k \}$ are such that 
      \begin{equation}
        \label{eq:3}
        \int_0^{T - \tau} \norm{ f_k( t + \tau ) - f_k( t ) }_{B}^p \dd t
        \to 0\quad\text{ as }\quad\tau\to0.
      \end{equation}
    \end{enumerate}
  \end{enumerate}
  Then there exists a subsequence $\{ f_{k_j} \} \subset \{ f_k \}$ and $f \in L^p( 0, T; X )$ such that
  \begin{equation}
    \begin{aligned}
      f_{k_j} & \rightharpoonup f && \mbox{ in } L^p( 0, T; X ) \\
      f_{k_j} & \to f && \mbox{ in } L^p( 0, T; B ).
    \end{aligned}
  \end{equation}
\end{Lem}

\begin{Rem}
  Using the criterion 2(a), we see that if $\{ \eta_k \} \subset L^2( 0, T; L^2( \Omega ) )$ with a constant $C > 0$ such that
  \begin{equation*}
    \norm{ \eta_k }_{L^2( 0, T; H^1( \Omega ) ) } + \norm{ \partial_t \eta_k }_{L^2\left( 0, T; \dual{H^{1}( \Omega )} \right) } \le C \qquad \mbox{ for all } k,
  \end{equation*}
  then, there exists a subsequence, for which will use the same subscript $\{ \eta_k \}$, and $\eta \in L^2( 0, T; H^1( \Omega ) )$ such that
  \begin{equation}
    \begin{aligned}
      \eta_{k} & \rightharpoonup \eta && \mbox{ in } L^2( 0, T; H^1( \Omega ) ) \\
      \eta_{k} & \to \eta && \mbox{ in } L^2( 0, T; L^2( \Omega ) ).
    \end{aligned}
  \end{equation}
  This follows from the compact embedding of $H^1(\Omega)$ in $L^2( \Omega )$.

  However, we wish to recover strong convergence of a subsequence with less control over the time derivatives.
  The generality of criterion 2(b) allows a more general weak in time notion of solution to be used.
\end{Rem}
We will apply this result for sequences to derive strongly convergent subsequences in $L^2\left( 0, T; \dual{H^{1/2}( \Gamma )} \right)$.

\begin{Lem}
  \label{lem:H12L2-compact}
  Let $\{ \xi_k \}$ be a bounded sequence in $H^{1/2}( \Gamma )$. Then there exists a subsequence $\{ \xi_{k_j} \} \subset \{ \xi_k \}$ and $\xi \in H^{1/2}( \Gamma )$ such that
  \begin{equation}
    \begin{aligned}
      \xi_{k_j} & \rightharpoonup \xi && \mbox{ in } H^{1/2}( \Gamma ) \\
      \xi_{k_j} & \to \xi && \mbox{ in } L^2( \Gamma ).
    \end{aligned}
  \end{equation}
\end{Lem}

\begin{proof}
  For any $\rho \in H^{1/2}( \Gamma )$, we define an extension to $\Omega$, written $E \rho \in H^1( \Omega )$, as the unique solution of:
  \begin{align*}
    - \Delta ( E \rho ) & = 0 && \mbox{ in } \Omega \\
    E \rho & = 0 && \mbox{ on } \partial_0 \Omega \\
    E \rho & = \rho && \mbox{ on } \Gamma.
  \end{align*}
  We note that for a constant independent of $\rho$, we have
  \begin{equation*}
    \norm{ E \rho }_{H^1( \Omega ) } \le c \norm{ \rho }_{H^{1/2}( \Gamma ) }.
  \end{equation*}

  This implies we have a sequence $\{ E \xi_k \}$ which is uniformly bounded in $H^1(\Omega)$: There exists $C_0 > 0$ such that
  \begin{equation*}
    \norm{ E \xi_k }_{H^1(\Omega)} \le C_0.
  \end{equation*}
  From the compact embedding of $H^1(\Omega)$ into $L^2(\Omega)$, we know that there exists a subsequence $\{ \xi_{k_j} \} \subset \{ \xi_k \}$, and $\eta \in H^{1}( \Omega )$ such that
  \begin{equation*}
    E \xi_{k_j} \to {\eta} \qquad \mbox{ in } L^2( \Omega ).
  \end{equation*} %\margnote[TR]{why is $\eta = E \xi$ - I think it's just ok to replace $E \xi$ with $\eta$ throughout here.}

  {
    Denote by $\xi = \eta|_\Gamma$.
  Fix $\eps > 0$ and choose $\delta \le \eps / ( 4 C_0 )$. From the strong convergence of $\{ E \xi_{k_j} \}$, we know there exists $K$ such that for $j \ge K$,
  \begin{equation*}
    \norm{ E \xi_{k_j} - {\eta} }_{L^2(\Omega)} \le \frac{\eps}{2 c_\delta},
  \end{equation*}
  where $c_\delta$ is from Proposition~\ref{prop:interpolated-trace}.
  It follows that for $j \ge K$, we can infer by applying the interpolated trace inequality (Proposition~\ref{prop:interpolated-trace}), with $\delta$ as above, that
  \begin{align*}
    \norm{ \xi_{k_j} - \xi }_{L^2( \Gamma )}
    & \le \delta \norm{ E \xi_{k_j} - \eta }_{H^1( \Omega )}
    + c_\delta \norm{ E \xi_{k_j} - E \xi }_{L^2( \Omega )} \\
    & \le 2\delta C_0 + \frac{\eps}{2} \le \eps.
  \end{align*}
  }
  Thus, we have shown the strong convergence of $\xi_k$ to $\xi$ in $ L^2( \Gamma )$.
\end{proof}

\begin{Lem}
  \label{lem:L2H-12-compact}
  Let $\{ \xi_k \}$ be a bounded sequence in $L^2( \Gamma )$. Then there exists a subsequence $\{ \xi_{k_j} \} \subset \{ \xi_k \}$ and $\xi \in L^2( \Gamma )$ such that
  \begin{equation}
    \begin{aligned}
      \xi_{k_j} & \rightharpoonup \xi && \mbox{ in } L^2( \Gamma ) \\
      \xi_{k_j} & \to \xi && \mbox{ in } \dual{H^{1/2}( \Gamma )}.
    \end{aligned}
  \end{equation}
\end{Lem}

\begin{proof}
  Since $\{ \xi_k \}$ is uniformly bounded in $L^2( \Gamma )$, we know that is has a subsequence $\{ \xi_{k_j} \}$ which weakly converges to some $\xi \in L^2( \Gamma )$.
  We suppose, for contradiction, that there exists no subsequence of $\{ \xi_{k_j} \}$ that strongly converges to $\xi$ in $\dual{H^{1/2}( \Gamma )}$.
  This implies that there exists $\delta > 0$ such that
  \begin{equation*}
    \norm{ \xi_{k_j} - \xi }_{\dual{H^{1/2}( \Gamma )} } \ge \delta.
  \end{equation*}
  Using the definition of $\dual{H^{1/2}(\Gamma)}$ as the dual space to $H^{1/2}( \Gamma )$, this implies there exists a sequence $\{ \rho_j \} \subset H^{1/2}( \Gamma )$, with $\norm{ \rho_j }_{H^{1/2}(\Gamma)} = 1$, such that {for all $j$}
  \begin{equation*}
    \langle \xi_{k_j} - \xi, \rho_j \rangle_{H^{1/2}(\Gamma)}
    = \int_\Gamma ( \xi_{k_j} - \xi ) \rho_j \dd \sigma \ge {\frac{\delta}{2}}.
  \end{equation*}
  From Lemma~\ref{lem:H12L2-compact}, we know that a subsequence $\{ \rho_{j_l} \} \subset \{ \rho_j \}$ converges strongly to $\rho \in H^{1/2}(\Gamma)$ in $L^2( \Gamma )$. Hence, we can infer
  \begin{equation*}
    \int_\Gamma ( \xi_{k_j} - \xi ) \rho \dd \sigma \ge {\frac{\delta}{2}}.
  \end{equation*}
  However, this contradicts the supposition that $\xi_{k_j}$ converges weakly to $\xi$ in $L^2( \Gamma )$.
\end{proof}
We conclude this section with a result which is similar in nature to the previous results.
\begin{Lem}
  \label{lem:trace-compact}
  Let $\{ \eta_k \}$ be a bounded sequence in $L^2( 0, T; H^1( \Omega ) )$ and $\eta \in L^2( 0, T; H^1( \Omega ) )$ such that
  \begin{equation}
    \eta_k \to \eta \qquad \mbox{ in } L^2( 0, T; L^2( \Omega ) ).
  \end{equation}
  Then the trace sequence converges to the trace of the limit:
  \begin{align}
    \eta_{k}|_{\Gamma} \to \eta|_{\Gamma} \qquad \mbox{ in } L^2( 0, T; L^2( \Gamma ) ).
  \end{align}
\end{Lem}

\begin{proof}
  Denote by $C_0 > 0$ the upper bound of $\{ \eta_k \}$ and $\eta$ in $L^2( 0, T; H^1( \Omega ) )$:
  \begin{equation*}
    \norm{ \eta_k }_{L^2( 0, T; H^1( \Omega ) )}
    + \norm{ \eta }_{L^2( 0, T; H^1( \Omega ) )}
    \le C_0.
  \end{equation*}
  Fix $\eps > 0$ and choose $\delta \le \eps / ( 2 C_0 )$.
  Then from the convergence of $\{ \eta_k \}$ in $L^2( 0, T; L^2( \Omega ) )$, there exists $K$ such that for $k \ge K$,
  \begin{equation*}
    \norm{ \eta_k - \eta }_{L^2( 0, T; L^2( \Omega ) )}
    \le \frac{\eps}{2 c_\delta},
  \end{equation*}
  where $c_\delta$ is from Proposition~\ref{prop:interpolated-trace}.
  It follows that for $k \ge K$, we can infer by applying the interpolated trace inequality (Proposition~\ref{prop:interpolated-trace}), with $\delta$ as above, that
  \begin{align*}
    \norm{ \eta_k - \eta }_{L^2( 0, T; L^2( \Gamma ) ) }
    & \le \delta \norm{ \eta_k - \eta }_{L^2( 0, T; H^1( \Omega ) )}
    + c_\delta \norm{ \eta_k - \eta }_{L^2( 0, T; L^2( \Omega ) )} \\
    & \le \delta C_0 + \frac{\eps}{2} \le \eps.
  \end{align*}
  Thus, we have shown the strong convergence of $\eta_k$ to $\eta$ in $L^2( 0, T; L^2( \Gamma ) )$.
\end{proof}

%%%%%%%%%%%%%%%%%%%%%%%%%%%%%%%%%%%%%%%%%%%%%%%%%%%%%

\section{Ligand-receptor model}\label{sec:model}

%%%%%%%%%%%%%%%%%%%%%%%%%%%%%%%%%%%%%%%%
In this section, we establish an existence and uniqueness theory for (\ref{eq:eps-problem}). As described in \S \ref{sec:limit-model},  (\ref{eq:eps-problem}) arises from (\ref{eq:k-problem}) if one neglects the receptor-ligand complexes,  non-dimensionalises as in (\ref{eq:nondim-pb}) and (for simplicity) sets the surface interchange flux $\mu=1$. 

In order to introduce the concept of a weak solution to (\ref{eq:eps-problem}), for $\gamma\in\Reals$, we introduce the Sobolev space
\[H^1_{e_\gamma}( \Omega ):=\{v\in\Hil{1}(\O)\vert v=\gamma\text{ on }\partial_0\O\},\]
where the boundary values are understood in the sense of traces and we adopt the notation, of using the same symbol for a function and its trace. We now introduce our concept of a weak solution to (\ref{eq:eps-problem}).

{
\begin{Defn}[Weak solution of  (\ref{eq:eps-problem})]\label{def:ws-eps-problem}
  For the Dirichlet boundary data case, we say that a pair $(u,w) \in \Lp{2}( {0,T} ; \Hil{1}_{e_{u_D}}(\O)) \times \Lp{2}( {0,T} ; \Hil{1}(\G))$ with $u,w\geq 0$ and with $(\partial_t u, \partial_t w ) \in \Lp{2}\left( {0,T};\dual{\Hil{1}_{e_{0}}(\O)}\right) \times \Lp{2}\left( {0,T}; \dual{\Hil{1}(\G)} \right)$  is a weak solution of (\ref{eq:eps-problem}) if for all $(\eta,\rho)\in\Hil{1}_{e_0}(\O)\times\Hil{1}(\G)$ {and for a.e. $t\in(0,T)$}
  \begin{subequations}
    \label{eqn:wf_eps_problem}
    \begin{align}
      \delO\dualp{\dual{\Hil{1}_{e_0}(\O)}}{\pdt u}{\eta}{\Hil{1}_{e_0}(\O)}
      + \int_\Omega  \nabla u \cdot \nabla \eta \dd x
      & = - \frac{1}{\delk} \int_\Gamma u w \eta \dd \sigma \\
      \label{eqn:wf_eps_problem_w}
      \dualp{\dual{\Hil{1}(\G)}}{\pdt w}{\rho}{\Hil{1}(\G)}
      + \delG \int_\Gamma \nabla_\Gamma w \cdot \nabla_\Gamma \rho \dd \sigma
      & = - \frac{1}{\delk} \int_\Gamma u w \rho \dd \sigma.
    \end{align}
  \end{subequations}
  In the case of Neumann boundary data, we say that a pair $(u,w) \in \Lp{2}( {0,T} ; \Hil{1}(\O)) \times \Lp{2}( {0,T} ; \Hil{1}(\G))$ with $u,w\geq 0$ and with $(\partial_t u, \partial_t w ) \in \Lp{2}\left( {0,T}, \dual{\Hil{1}(\O)}\right) \times \Lp{2}\left( {0,T}; \dual{\Hil{1}(\G)} \right)$  is a weak solution of (\ref{eq:eps-problem}) if for all $(\eta,\rho)\in\Hil{1}(\O)\times\Hil{1}(\G)$ {and for a.e. $t\in(0,T)$}
  \begin{subequations}
    \label{eqn:wf_eps_problem_neumann}
    \begin{align}
\delO\dualp{\dual{\Hil{1}(\O)}}{\pdt u}{\eta}{\Hil{1}(\O)}      + \int_\Omega  \nabla u \cdot \nabla \eta \dd x
      & = - \frac{1}{\delk} \int_\Gamma u w \eta \dd \sigma \\
      \dualp{\dual{\Hil{1}(\G)}}{\pdt w}{\rho}{\Hil{1}(\G)}
      + \delG \int_\Gamma \nabla_\Gamma w \cdot \nabla_\Gamma \rho \dd \sigma
      & = - \frac{1}{\delk} \int_\Gamma u w \rho \dd \sigma,
    \end{align}
  \end{subequations}
\end{Defn}
}

We note that if $u \in \Lp{2}( {0,T}; \Hil{1}( \O ) )$ then by the trace theorem $u \in \Lp{2}( {0,T}; \Hil{1/2}( \G ) )$.
We now show the well posedness of problem (\ref{eq:eps-problem}) in the sense of the following Theorem.
\begin{The}[Existence and uniqueness of a  bounded solution pair to (\ref{eq:eps-problem})]
  \label{thm:eps-problem}
  Given bounded, non-negative initial data $u_0$ and $w_0$, there exists a unique solution pair $(u,w)$ to {the systems~\eqref{eqn:wf_eps_problem} and \eqref{eqn:wf_eps_problem_neumann}}. Furthermore, we have that in the case of Dirichlet data 
  \begin{equation}
    \begin{aligned}
      & 0 \le u(x,t) \le \max( \| u_0 \|_{L^\infty(\Omega)}, u_D )
      && \mbox{ for a.e. } ( x,t ) \in \Omega \times (0,T) \\
      & 0 \le w(x,t) \le \| w_0 \|_{L^\infty(\Gamma)}
      && \mbox{ for a.e. } ( x,t ) \in \Gamma \times (0,T),
    \end{aligned}
  \end{equation}
  or in the case of Neumann data
  \begin{equation}
    \begin{aligned}
      & 0 \le u(x,t) \le \| u_0 \|_{L^\infty(\Omega)}
      && \mbox{ for a.e. } ( x,t ) \in \Omega \times (0,T) \\
      & 0 \le w(x,t) \le \| w_0 \|_{L^\infty(\Gamma)}
      && \mbox{ for a.e. } ( x,t ) \in \Gamma \times (0,T).
    \end{aligned}
  \end{equation}
\end{The}
\begin{proof} 
In the interests of brevity we give the full details of the proof only in the Dirichlet case. An analogous argument holds for the case of Neumann boundary conditions.

We start by replacing $w$ by $M(w)$ in the nonlinear coupling terms, where $M \colon \mathbb{R} \to \mathbb{R}^+$ is the cut off function
  \begin{equation}\label{eqn:M_def}
    M( r ) =
    \begin{cases}
      0 & r < 0 \\
      r &  0\le r  \le M \\
      M & r > M,
    \end{cases}
  \end{equation}
  with $M\geq \| w_0 \|_{L^\infty(\G)}$.   This leads us to consider the following
 problem. Find $(u,w)$, in the same spaces as Definition~\ref{def:ws-eps-problem}, that satisfy for all $( \eta, \rho ) \in H^1_{e_0} ( \Omega ) \times H^1( \Gamma )$ {and for a.e. $t\in(0,T)$}
  \begin{align}
    \label{eq:M-problem-a}
   \delO\dualp{\dual{\Hil{1}_{e_0}(\O)}}{\pdt u}{\eta}{\Hil{1}_0(\O)}     + \int_\Omega  \nabla u \cdot \nabla \eta \dd x
    & = - \frac{1}{\delk} \int_\Gamma u  M( w ) \eta \dd \sigma \\
    \label{eq:M-problem-b}
   \dualp{\dual{\Hil{1}(\G)}}{\pdt w}{\rho}{\Hil{1}(\G)}
    + \delG \int_\Gamma \nabla_\Gamma w \cdot \nabla_\Gamma \rho \dd \sigma
    & = - \frac{1}{\delk} \int_\Gamma u M( w ) \rho \dd \sigma.
  \end{align}
 {As $M(w)$ is bounded}, existence  for this problem with the cutoff nonlinearity can be shown via a Galerkin method and standard energy arguments.
We now show positivity of the solutions to \eqref{eq:M-problem-a}, \eqref{eq:M-problem-b}: $u,w \ge 0$ almost everywhere in their domains and that the trace of $u\ge0$ {on $\G$}. Testing \eqref{eq:M-problem-a} with $u_-= \min( u, 0 )$ and using the fact that $M(w) \ge 0$, we have
  \begin{equation*}
    \frac{\delO}{2} \frac{d}{dt} \int_\Omega (u_-)^2 \dd x
    + \int_\Omega  | \nabla (u_-) |^2 \dd x
    = - \frac{1}{\delk} \int_\Gamma (u_-)^2M(w) \dd \sigma
    \le 0.
  \end{equation*}
 Since $u_0 \ge 0$, we have $u \ge 0$ almost everywhere in $\Omega \times (0,T)$. Moreover, by the trace inequality, applied to $u_-$, we have that the trace of $u$ is non-negative.
 We next test \eqref{eq:M-problem-b} with $w_- = \min( w, 0 )$ to get
  \begin{align*}
    \frac{1}{2} \frac{d}{dt} \int_\Gamma (w_-)^2 \dd \sigma
    + \delG \int_\Gamma | \nabla_\Gamma (w_-) |^2 \dd \sigma
    = - \frac{1}{\delk} \int_\Gamma u M(w)w_- \dd \sigma
    =0,
  \end{align*}
  as $M(w)w_-=0$ from the definition of $M()$ \eqref{eqn:M_def}.
 Since $w_0 \ge 0$, we see that $w \ge 0$ almost everywhere in $\Gamma \times (0,T)$.
 We now show  pointwise bounds. Let $(u,w)$ be solutions of \eqref{eq:M-problem-a} and \eqref{eq:M-problem-b} and set $\theta^w = ( w - \| w_0 \|_{L^\infty(\G)})$. The variable $\theta^w$ satisfies
  \begin{equation*}
  \dualp{\dual{\Hil{1}(\G)}}{\pdt \theta^w}{\rho}{\Hil{1}(\G)}
    + \delG \int_\Gamma \nabla_\Gamma \theta^w \cdot \nabla_\Gamma \rho \dd \sigma
    = - \frac{1}{\delk} \int_\Gamma  u w  \rho \dd \sigma.
  \end{equation*}
  We test with $\rho = ( \theta^w )_+ \ge 0$ and recall that $u , w \ge 0$ then
  \begin{equation*}
    \frac{1}{2} \frac{d}{dt} \int_\Gamma ( \theta^w_+ )^2 \dd \sigma
    + \delG \int_\Gamma | \nabla \theta^w_+ |^2 \dd \sigma
    = - \frac{1}{\delk} \int_\Gamma u w \theta^w_+ \dd \sigma \le 0.
  \end{equation*}
  This implies that $\theta^w_+ = 0$ and hence $w \le \| w_0 \|_{L^\infty(\G)}$. The same argument for $u$ with $\theta^u=(u-\max(u_D,\Lpn{\infty}{u}{\O})$ so that $\theta^u_+\in\Hil{1}_{e_0}$, gives $u \le \max(u_D,\Lpn{\infty}{u}{\O})$.  As  $M$ was chosen such that $M\geq  \| w_0 \|_{L^\infty(\G)}$ and $w\ge 0$, we have that $M(w)=w$,  hence we have constructed a solution to (\ref{eqn:wf_eps_problem}) which satisfies 
  \begin{equation}\label{eq:infinity_bounds_eps_pb}
    0 \le u \le \max( \| u_0 \|_\infty, u_D) \quad \mbox{ and }
    0 \le w \le \| w_0 \|_\infty.
  \end{equation}

It remains to show that the solution is unique. To do this, we argue as follows.
  Let $(u_1,w_1)$ and $(u_2,w_2)$ be two (weak) solutions of {\eqref{eqn:wf_eps_problem}}. Defining $e^u:=u_1-u_2$ and $e^w:=w_1-w_2$ we 
  have that $e^u,e^w$ satisfy for all $(\eta,\rho)\in\Hil{1}_{e_0}(\O)\times\Hil{1}(\G)$ {and for a.e. $t\in(0,T)$}
  \begin{subequations}  \label{eqn:eps_pb_uniq_wf}
  \begin{align}
  \label{eqn:eps_pb_uniq_wf_u}
\delO\dualp{\dual{\Hil{1}_{e_0}(\O)}}{\pdt e^ u}{\eta}{\Hil{1}_{e_0}(\O)} +\int_\O\nabla e^u\cdot\nabla\eta \dd x&=-\frac{1}{\delk}\int_\G(u_1w_1-u_2w_2)\eta \dd \sigma\\
 \dualp{\dual{\Hil{1}(\G)}}{\pdt e^ w}{\rho}{\Hil{1}(\G)}\dd \sigma+\int_\G\delG\nabla_\G e^w\cdot\nabla_\G\rho \dd \sigma&=-\frac{1}{\delk}\int_\G(u_1w_1-u_2w_2)\rho \dd \sigma
    \label{eqn:eps_pb_uniq_wf_w}
  \end{align} 
  \end{subequations}
  {
Let $\psi:\Reals\to\Reals$ be a smooth convex function satisfying $\psi(0)=\psi^{\prime}(0)=0$. Setting $\eta=\psi^{\prime}(e^u)$ and $\rho=\psi^{\prime}(e^w)$ in (\ref{eqn:eps_pb_uniq_wf})  and combining the equations gives
  \begin{align}
\frac{\diff }{\diff t}\left(\int_\O\delO\psi(e^u) \dd x+\int_\G \psi(e^w) \dd \sigma\right)&+\int_\O\psi^{\prime\prime}(e^u)\lv\nabla e^u\rv^2\dd x+\int_\G\delG\psi^{\prime\prime}(e^w)\lv\nabla_\G e^w\rv^2\dd \sigma\\
\notag
&=-\frac{1}{\delk}\int_\G(u_1w_1-u_2w_2)\left(\psi^{\prime}(e^u)+\psi^{\prime}(e^w)\right)\dd \sigma.
\end{align}
Hence as $\psi$ is convex we have
\begin{align}
  \label{eq:4}
\frac{\diff }{\diff t}\left(\int_\O\delO\psi(e^u)\dd x+\int_\G \psi(e^w)\dd \sigma\right)\leq-\frac{1}{\delk}\int_\G(u_1w_1-u_2w_2)\left(\psi^{\prime}(e^u)+\psi^{\prime}(e^w)\right) \dd \sigma.
\end{align}
Integration in time gives
\begin{align}
  \label{eq:5}
\int_\O\delO\psi(e^u(\cdot,t))\dd x+\int_\G \psi(e^w(\cdot,t))\dd \sigma\leq-\frac{1}{\delk}\int_0^t\int_\G(u_1w_1-u_2w_2)\left(\psi^{\prime}(e^u)+\psi^{\prime}(e^w)\right) \dd \sigma \dd t,
\end{align}
as $e^u(\cdot,0)=0$ and $e^w(\cdot,0)=0$ and we have chosen $\psi$ such that $\psi(0)=0$.
Defining the function 
\[
\sgn(\eta)=
\begin{cases}
1\quad&\text{if }\eta>0\\
0\quad&\text{if }\eta=0\\
-1\quad&\text{if }\eta<0,
\end{cases}
\] 
we replace $\psi$ by a sequence of smooth functions $\psi_k$ such that
\begin{equation*}
\psi_k(x)\to\lv x\rv,\quad \psi^\prime_k(x)\to \sgn(x),\quad x\in \Reals,
\end{equation*}
pointwise
and pass to the limit ($k \rightarrow \infty$), which yields
} 
\begin{align}\label{eqn:uniqueness_RHS}
\int_\O\delO\lv e^u(\cdot,t)\rv \dd x+\int_\G \lv e^w(\cdot,t)\rv \dd \sigma\leq-\frac{1}{\delk}\int_0^t\int_\G(u_1w_1-u_2w_2)\left(\sgn(e^u)+\sgn(e^w)\right) \dd \sigma \dd t.
\end{align}
%We now show that the right hand side of \eqref{eqn:uniqueness_RHS} is non-positive. The result is trivial if $e^u$ and $e^w$ have different signs since then $\sgn(e^u)+\sgn(e^w) = 0$. As
%\[
%-\frac{1}{\delk}(u_1w_1-u_2w_2)\left(\sgn(e^u)+\sgn(e^w)\right)=-\frac{1}{\delk}(u_1e^w+w_2e^u)\left(\sgn(e^u)+\sgn(e^w)\right),
%\]
%we have that if $e^u=0$ 
%\[
%-\frac{1}{\delk}(u_1w_1-u_2w_2)\left(\sgn(e^u)+\sgn(e^w)\right)=-\frac{1}{\delk}u_1\lv e^w\rv\leq 0
%\]
%as we have assumed positivity of the solutions. A similar argument applies if $e^w=0$. If $e^u,e^w>0$ we have
%\[
%-\frac{1}{\delk}(u_1w_1-u_2w_2)\left(\sgn(e^u)+\sgn(e^w)\right)=-\frac{2}{\delk}(u_1e^w+w_2e^u)\leq 0
%\]
%once again due to the assumption of positivity of the solutions. If  $e^u,e^w<0$ we have
%\[
%-\frac{1}{\delk}(u_1w_1-u_2w_2)\left(\sgn(e^u)+\sgn(e^w)\right)=\frac{2}{\delk}(u_1e^w+w_2e^u)\leq 0
%\]
%  due to the  positivity of solutions to (\ref{eq:eps-problem}). Hence the right hand side of (\ref{eqn:uniqueness_RHS}) is {non-positive}. 
{
For $a_1,b_1,a_2,b_2\in\Reals^+$ it is easily verified that 
\[
(a_1b_1-a_2b_2)(\sgn(a_1-a_2)+\sgn(b_1-b_2))\geq 0,
\]
hence the right hand side of (\ref{eqn:uniqueness_RHS}) is non-positive
Thus  
  $\text{for a.e., }{t\in(0,T)}$
 \[
\left( \int_\O\delO\lv e^u\rv \dd x+\int_\G \lv e^w\rv \dd \sigma\right)=0,
 \] 
 }
 which completes the proof of uniqueness and hence the proof of the theorem.
\end{proof}

In the subsequent sections we will consider the limit problems obtained on sending  $\delO,\delG$ and $\delk$ to zero in \eqref{eq:eps-problem}. To this end we derive some estimates on the solution pair $(u,w)$ of {\eqref{eqn:wf_eps_problem}}, which we will use in the subsequent sections to deduce the existence of convergent subsequences which converge to solutions of the limit problems.
We note that the bounds hold for constants which are independent of $\delta_k, \delta_\Gamma$ and $\delta_\Omega$.

\begin{Lem}[Estimates for the solution of (\ref{eqn:wf_eps_problem}) and (\ref{eqn:wf_eps_problem_neumann})] \label{Lem:unif_est}
  The solution pair $(u,w)$ to (\ref{eqn:wf_eps_problem}) and (\ref{eqn:wf_eps_problem_neumann}) satisfy the following estimates,
\begin{equation}\begin{split}
\label{eqn:eps_pb_estimate_1}
    \delO\Lpn{\infty}{u}{(0,T);\Lp{2}(\O)}^2 +
    {2\Lpn{2}{\nabla u}{(0,T);\Lp{2}(\O)}^2}
 & \le \delO\int_\Omega u_0^2 \dd x + C_D\\
         \Lpn{\infty}{w}{(0,T);\Lp{2}(\G)}^2
      +     2\delG\Lpn{2}{\nabla_\G w}{(0,T);\Lp{2}(\G)}^2
    &
  \le 
  \int_\Gamma w_0^2 \dd \sigma,
  \end{split}
\end{equation}
where  $C_D\in\Reals^+$ depends on the Dirichlet boundary data $u_D$ and $C_D=0$ in the case of the Neumann boundary condition.
Furthermore, we have an estimate on the nonlinearity:
\begin{align}\label{eq:uw_estimate}
\frac{1}{\delk}\Lpn{1}{u w}{(0,T)\times\G}\leq \Lpn{1}{w_0}{\G}.
\end{align}
The following estimate on time translates of $u$ and $w$ along with Lemma \ref{lem:als-compact} will be used to deduce the necessary compactness
\begin{align}\label{eq:time_trans_estimate}
\delO\int_0^{T-\tau}\int_\O \left(u (\cdot,t+\tau)-u (\cdot,t)\right)^2 \dd x \dd t+\int_0^{T-\tau}\int_\G \left(w (\cdot,t+\tau)-w (\cdot,t)\right)^2 \dd \sigma \dd t\leq C\tau,
\end{align}
where the constant $C$ is independent of $\tau,\delO,\delG$ and $\delk$.
\end{Lem}
\begin{proof} The first estimate (\ref{eqn:eps_pb_estimate_1}) 
follows from a straightforward energy argument due to the non negativity of $u$ and $w$. Specifically, test with $(u-\Du, w)$ where $\Du$ satisfies $\lap\Du=0$ in $\O$, $\Du=0$ on $\G$ and $\Du=u_D$ on $\partial_0\Omega$ in the Dirichlet case (\ref{eqn:wf_eps_problem})  or simply with $(u,w)$ in the Neumann case (\ref{eqn:wf_eps_problem_neumann}).

For the estimate (\ref{eq:uw_estimate}) we have using the {non-negativity} of $u,w$
\begin{align*}
\frac{1}{\delk}\Lpn{1}{u w}{{\G \times (0,T)}}&=\frac{1}{\delk}\int_0^T\int_\G w u \dd \sigma \dd t\\
&=\int_0^T\int_\G -\pdt w \dd \sigma \dd t\\
&=\int_\G -w(\cdot,T)+w^0(\cdot) \dd \sigma\\
&\leq \Lpn{1}{w^0}{\G},
\end{align*}
where we have used the {non-negativity}  of $w$ in the last step.

For the estimate (\ref{eq:time_trans_estimate}) we argue as follows. For a fixed $\tau\in(0,T)$ and for $t\in[0,T-\tau)$ introducing the notation
$\pdtau f(t):=f(t+\tau)-f(t)$ we have using (\ref{eqn:wf_eps_problem})
\begin{align*}
\int_\G &\left(w (\cdot,t+\tau)-w (\cdot,t)\right)^2 \dd \sigma = \int_0^\tau\int_\G\pdt w(\cdot,t+s)\pdtau w(\cdot,t)\dd \sigma\dd s\\
&= \int_0^\tau\int_\G-\delG\nabla_\G w(\cdot,t+s)\cdot\nabla_\G\pdtau w(\cdot,t)- \frac{1}{\delk}[u w](\cdot,t+s)\pdtau w(\cdot,t)\dd \sigma\dd s.
\end{align*}
Integrating in time gives 
\begin{align}\label{eq:time_trans_estimate_pf_1}
&\int_0^{T-\tau}\int_\G \left(w (\cdot,t+\tau)-w (\cdot,t)\right)^2\dd \sigma \dd t\\
&= \int_0^\tau\int_0^{T-\tau}\int_\G-\delG\nabla_\G w(\cdot,t+s)\cdot\nabla_\G\pdtau w(\cdot,t)- \frac{1}{\delk}[u w](\cdot,t+s)\pdtau w(\cdot,t)\dd \sigma\dd t\dd s\notag\\
&\leq \int_0^\tau {2}\delG\Lpn{2}{\nabla_\G w}{{\G \times (0,T)}}^2+\Lpn{\infty}{\pdtau w}{{\G \times (0,T)}} \frac{1}{\delk}\Lpn{1}{u w}{{\G \times (0,T)}}\dd s,\notag
\end{align}
where we have used Young's inequality  in the last step. Applying the estimates (\ref{eq:infinity_bounds_eps_pb}), (\ref{eqn:eps_pb_estimate_1}) and (\ref{eq:uw_estimate}) in (\ref{eq:time_trans_estimate_pf_1}) yields the desired estimate for the second term in (\ref{eq:time_trans_estimate}). For the bound on the first term in (\ref{eq:time_trans_estimate}), we note that as $\pdtau u\in\Hil{1}_{e_0}(\O)$
\begin{align*}
\delO\int_\O &\pdtau u(\cdot,t)^2\dd x \\
&= \int_0^\tau\int_\O-\nabla u(\cdot,t+s)\cdot\nabla\pdtau u(\cdot,t)- \frac{1}{\delk}[u w](\cdot,t+s)\pdtau u(\cdot,t)\dd \sigma\dd s,
\end{align*}
from which the desired bound follows from an analogous calculation to \eqref{eq:time_trans_estimate_pf_1} together with the estimates (\ref{eq:infinity_bounds_eps_pb}), (\ref{eqn:eps_pb_estimate_1}) and (\ref{eq:uw_estimate}).
\end{proof}

\section{Fast reaction limit problem ($\delk=0$)}\label{sec:limit_parabolic}
We now show that for fixed $\delO,\delG>0$ as $\delk\to0$ the solution to (\ref{eq:eps-problem}) converges to a (weak) solution to the following  constrained parabolic limit problem. For convenience we work with $v=-w$ and set $v^0=-w^0$.
\begin{Pbm}[Problem for instantaneous reaction rate]\label{pbm:parabolic-problem}
Find {$\ubar \colon \Omega \times [0,T) \to \mathbb{R}^+$, $\vbar \colon \Gamma \times [0,T) \to \mathbb{R}^-$} such that
\begin{subequations}
  \label{eq:parabolic-problem}
  \begin{align}
    \delO\partial_t \ubar -  \Delta \ubar & = 0 && \mbox{ in } \Omega \times (0,T) \\
     \nabla \ubar \cdot \vec{\nu} +\partial_t \vbar - \delG \Delta_\Gamma \vbar & = 0 \quad \mbox{ and }\quad \vbar\in\beta(\ubar) && \mbox{ on } \Gamma \times (0,T) \\
    \ubar &= u_D \quad \mbox{ or }\quad \nabla {\ubar} \cdot \vec{\nu}_{\Omega}=0 && \mbox{ on } \partial_0 \Omega \times (0,T) \\
    \ubar(\cdot,0) & = u^0(\cdot) \ge 0 && \mbox{ in } \Omega \\
    \vbar(\cdot,0) & = v^0(\cdot) \le 0 && \mbox{ on } \Gamma.
  \end{align}
\end{subequations}
Here $\beta \colon \mathbb{R} \to \{0,1\}^{\mathbb{R}}$ is the set valued function (c.f., Figure \ref{Fig:beta})
\begin{equation}\label{eqn:beta_def}
  \beta( r ) =
  \begin{cases}
    \emptyset & \mbox{ if } r < 0 \\
    [ -\infty, 0 ] & \mbox{ if } r = 0 \\
    \{ 0 \} & \mbox{ if } r > 0.
  \end{cases}
\end{equation}
We consider (\ref{eq:parabolic-problem}) as a  parabolic equation with dynamic boundary conditions interpreted {as} a differential inclusion.
\end{Pbm}

\begin{figure}
\centering
\includegraphics{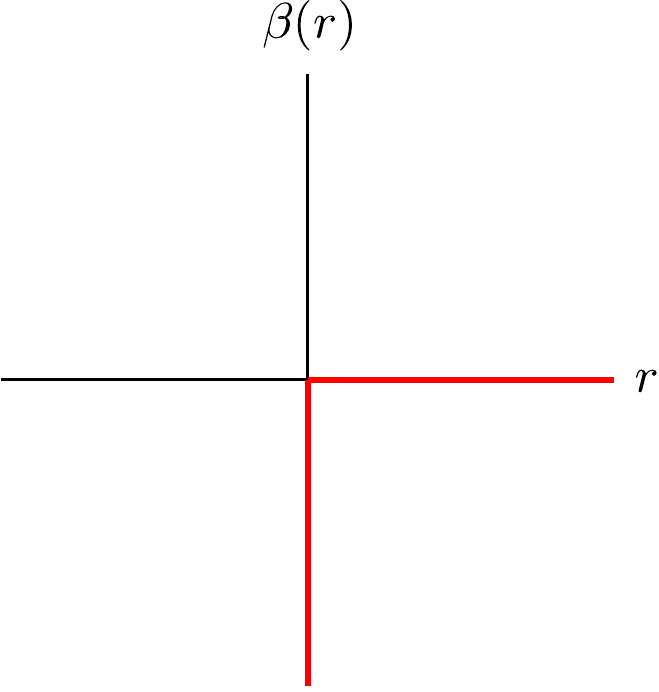}
\caption{Sketch of the function $\beta$ c.f., (\ref{eqn:beta_def})}\label{Fig:beta}
\end{figure}

{
In order to define a weak solution to (\ref{eq:parabolic-problem}) we define the Bochner spaces
\[
\mathcal{V}_{e_0}(\O)=\left\{v\in\Lp{2}\left({0,T};{\Hil{1}_{e_0}(\O)}\right) : \pdt v\in\Lp{2}\left({0,T};\dual{\Hil{1}_{e_0}(\O)}\right)\right\}
\]
and
\[
\mathcal{V}(\G)=\left\{v\in\Lp{2}\left({0,T};{\Hil{1}(\G)}\right) : \pdt v\in\Lp{2}\left({0,T};\dual{\Hil{1}(\G)}\right)\right\}.
\]
We will make use of the following function space
\begin{align*}
\mathcal{V}_{e_0}(\O,\G):=\Bigg\{v\in\mathcal{V}_{e_0}(\O) : v\vert_\G\in\mathcal{V}(\G)\Bigg\}.
\end{align*}

We note that similar spaces have been introduced for the weak formulation of a parabolic problems with dynamic boundary conditions,  \citep[see, for example,][]{calatroni2013global}.

\begin{Defn}[Weak solution of Problem \ref{pbm:parabolic-problem}]\label{def:ws-parabolic-problem}
\noindent
We say that a pair $(\ubar,\vbar)$ with $u\in \Lp{2}( {0,T} ; \Hil{1}_{e_{u_D}}(\O))\cap \Lp{\infty}( {0,T} ; \Lp{2}(\O))$ and $\vbar\in\Lp{2}( {0,T} ; \Hil{1}(\G))\cap \Lp{\infty}( {0,T} ; \Lp{2}(\G))$ {with $\ubar\geq0$ and $\vbar\leq0$} is a weak solution of Problem \ref{pbm:parabolic-problem} if for all $\eta\in\mathcal{V}_{e_0}(\O,\G)$ with  $\eta(\cdot,T)=0$, we have
\begin{multline}
  \int_0^T \Bigg(-\delO\dualp{\dual{\Hil{1}_{e_0}(\O)}}{\pdt\eta}{\ubar}{{\Hil{1}_{e_{u_D}}(\O)}}+ \int_\O\nabla \ubar\cdot\nabla \eta\dd x-\dualp{\dual{\Hil{1}(\G)}}{\pdt\eta}{\vbar}{{\Hil{1}(\G)}}\\
  + \int_\G\delG \nabla_\G \vbar\cdot\nabla_\G\eta\dd \sigma \Bigg) \dd t 
  = \int_\O \delta_\Omega u^0\eta(\cdot,0)\dd x
  + \int_\G v^0\eta(\cdot,0)\dd \sigma
\end{multline}
and
\begin{align}
\vbar\in\beta(\ubar)\quad\mbox{ a.e.} \mbox{ on }\G\times{(0,T)}.
\end{align}
We make the corresponding modifications to the function spaces for the Neumann boundary condition.
\end{Defn}
}

\begin{The}[Convergence of the solution of (\ref{eq:eps-problem}) to a solution of (\ref{eq:parabolic-problem})]
  \label{thm:parab-pbm}
  As $\delk\to0$ the  solution pair $(u,w)$ to {(\ref{eqn:wf_eps_problem})} converge (up to a subsequence) to a pair $(\bar{u},\bar{w})$ in the following topologies
\begin{align}
 u &\rightharpoonup \bar{u} \quad \mbox{ in } L^2( 0, T; H^1_{e_{u_D}}( \Omega ) ) \label{eqn:weak_conv_parab_u}\\
 \notag (  u &\rightharpoonup \bar{u} \quad \mbox{ in } L^2( 0, T; H^1( \Omega ) ),\quad \text{in the Neumann case}) \\
  w &\rightharpoonup \bar{w} \quad \mbox{ in } L^2( 0, T; H^1( \Gamma ) ),\\
    u &\rightarrow \bar{u} \quad \mbox{ in } \Lp{2}(\O\times{(0,T)}),\label{eqn:strong_conv_parab_u}\\
  w &\rightarrow \bar{w} \quad \mbox{ in } \Lp{2}(\G\times{(0,T)})\label{eqn:strong_conv_parab_w}.
  \end{align}
   Moreover, the pair $(\bar{u},\bar{v})$, with $\bar{v}=-\bar{w}$  are a weak solution to Problem{~\ref{pbm:parabolic-problem}}.
 \end{The}
\begin{proof} 
In the interests of brevity we give the details for the Dirichlet boundary condition case. The Neumann case is handled similarly.

From standard weak compactness arguments \eqref{eq:weak-conv} together with the estimate \eqref{eqn:eps_pb_estimate_1}, we can extract a subsequence which we will still denote $( u, w )$ such that
\begin{align*}
  u &\rightharpoonup \bar{u} \quad \mbox{ in } L^2( 0, T; H^1_{e_{u_D}}( \Omega ) ), \\
  w &\rightharpoonup \bar{w} \quad \mbox{ in } L^2( 0, T; H^1( \Gamma ) ).
\end{align*}
From the Aubin-Lions-Simon compactness theory (Lemma~\ref{lem:als-compact}), the estimate on time translates (\ref{eq:time_trans_estimate}) means we can extract a subsequence which we will still denote $( u, w )$ such that
\begin{align*}
    u &\rightarrow \bar{u} \quad \mbox{ in } \Lp{2}(\O\times{(0,T)}),\\
  w &\rightarrow \bar{w} \quad \mbox{ in } \Lp{2}(\G\times{(0,T)}).
\end{align*}
{
We now show the pair $(\bar{u},\bar{v})$, with $\bar{v}=-\bar{w}$, are a weak solution to Problem \ref{pbm:parabolic-problem}.  We start by noting that for all $\eta\in \mathcal{V}_{e_0}(\O,\G)$ with $\eta(\cdot,T)=0$, we have
\begin{multline*}
  \int_0^T- \delO\dualp{}{ \pdt\eta}{u}{{\Hil{1}_{e_{u_D}}(\O)}}+ \int_\O\nabla u\cdot\nabla \eta\dd x\dd t-{\delO}\int_\O u^0\eta(\cdot,0)\dd x =\int_0^T\int_\G-\frac{1}{\delk}u w\eta\dd \sigma\dd t \\
  = \int_0^T- \dualp{}{\pdt \eta}{w}{\Hil{1}(\G)} +\delG\int_\G\nabla_\G w\cdot\nabla_\G\eta\dd \sigma\dd t-\int_\G w^0\eta(\cdot,0)\dd \sigma.
\end{multline*}
Letting $\delk\to0$ the convergence results (\ref{eqn:weak_conv_parab_u})---(\ref{eqn:strong_conv_parab_w}) give
\begin{multline*}
\int_0^T \left( - \delO\dualp{}{ \pdt\eta}{\bar u}{{\Hil{1}_{e_{u_D}}(\O)}}+\int_\O\nabla \bar{u}\cdot\nabla \eta\dd x \right) \dd t-{\delO}\int_\O u^0\eta(\cdot,0)\dd x\\
=\int_0^T \left(- \dualp{}{\pdt \eta}{\bar w}{\Hil{1}(\G)} +\delG\int_\G\nabla_\G \bar{w}\cdot\nabla_\G\eta\dd \sigma \right) \dd t-\int_\G w^0\eta(\cdot,0)\dd \sigma,
\end{multline*}
and hence with $\bar{v}=-\bar{w}$
\begin{multline}\label{eqn:barubarv}
  \int_0^T \left( - \delO\dualp{}{ \pdt\eta}{\bar u}{{\Hil{1}_{e_{u_D}}(\O)}}+\int_\O\nabla \bar{u}\cdot\nabla \eta\dd x -\dualp{}{\pdt \eta}{\bar v}{\Hil{1}(\G)}+\delG\int_\G\nabla_\G \bar{v}\cdot\nabla_\G\eta\dd \sigma \right) \dd t \\
  = {\delO}\int_\O u^0\eta(\cdot,0)\dd x + \int_\G v^0\eta(\cdot,0)\dd \sigma.
\end{multline}
}
It remains to show that $\bar{v}\in\beta(\bar{u})$. As $u, w\geq 0$ for all $\delk$, we have $\bar{u}\geq0$ and $\bar{v}=-\bar{w}\leq 0$.  
Moreover from (\ref{eq:uw_estimate})  we have
\begin{align*}
\int_0^T\int_\G u w\leq \delk\Lpn{1}{w_0}{\G},
\end{align*}
and hence the strong convergence results (\ref{eqn:strong_conv_parab_u}) and (\ref{eqn:strong_conv_parab_w}) imply
\begin{align*}
\bar{u}\bar{v}=-\bar{u}\bar{w}=0\quad\mbox{a.e.  in }\G\times{(0,T)}.
\end{align*}
Thus the limit pair $(\bar{u},\bar{v})$ are a weak solution to Problem \ref{pbm:parabolic-problem} in the sense of Definition \ref{def:ws-parabolic-problem}.
\end{proof}

\begin{Rem}[Uniqueness of the solution to Problem \ref{pbm:parabolic-problem}.]
Theorem \ref{thm:parab-pbm} ensures existence of a solution to Problem \ref{pbm:parabolic-problem}. However we are unable at present to prove uniqueness. 
In particular, the strategy employed  for the proof of uniqueness to the limiting problems \ref{pbm:PdBC-limit-problem} and \ref{pbm:EdBC-limit-problem} does not seem applicable in this case.
\end{Rem}

\section{Parabolic limit problem with dynamic boundary condition ($\delk=\delG=0$)}
\label{sec:parab-limit-pb}

We now present a rigorous derivation of the parabolic problem with dynamic boundary conditions presented in \S \ref{sec:limit-model} as a limit 
of (\ref{eq:eps-problem}). Specifically we show that for fixed $\delO>0$, in the limit  $\delk=\delG\to0$ the unique solution of the problem (\ref{eq:eps-problem})
% find $u \colon \Omega \times (0,T) \to \mathbb{R}$ and $w \colon \Gamma \times (0,T) \to \mathbb{R}$ such that
%\begin{equation}
%  \label{eq:aux-eps-problem}
%  \begin{cases}
%    \delO\pdt u -  \lap u=0\quad&\text{ in }\O\times(0,T]\\
%    \nabla u\cdot \normal=-\frac{1}{\eps}u w\quad&\text{ on }\G\times(0,T]\\
%    u=u_D\quad\text{ or }\quad\nabla\tu\cdot\normal=0\quad&\text{ on }\partial_0\O\times(0,T]\\
%    u(\cdot,0)=u^0(\cdot)\geq0\quad&\text{ in }\O\\
%    \pdt w-\eps\lap_\Gw= \nabla u\cdot \normal \quad&\text{ on }\G\times(0,T]\\
%    w(\cdot,0)=w^0(\cdot)\geq0\quad&\text{ on }\G,
%  \end{cases}
%\end{equation}
converges to the unique solution of the following problem.

\begin{Pbm}\label{pbm:PdBC-limit-problem}
  Find {$\tu \colon \Omega \times [0,T) \to \mathbb{R}^+$ and $\tv \colon \Gamma \times [0,T) \to \mathbb{R}^-$} such that
\begin{subequations}
\label{eq:PdBC-limit-problem}
  \begin{align}
    \delO\pdt \tu- \Delta \tu  & = 0 && \mbox{ in } \Omega \times (0,T) \\
    \nabla \tu \cdot \vec{\nu} + \partial_t \tv  & = 0 && \mbox{ on } \Gamma \times (0,T) \\
    \tv & \in \beta(\tu) && \mbox{ on } \Gamma \times (0,T) \\
    \tu  & = u_D\quad\text{ or }\quad\nabla\tu\cdot\normal=0 && \mbox{ on } \partial_0 \Omega \times (0,T) \\
        \tu(\cdot, 0)  & = u_0(\cdot)\geq 0 && \mbox{ on } \O\\
    \tv(\cdot, 0) & = v_0(\cdot)\leq 0 && \mbox{ on } \Gamma,
  \end{align}
\end{subequations}
where $\beta \colon \mathbb{R} \to \{0,1\}^{\mathbb{R}}$ is the set valued function defined in (\ref{eqn:beta_def}).
\end{Pbm}

{
In order to define a weak solution of  Problem{~\ref{pbm:PdBC-limit-problem}} we introduce the space
\[
\Hil{1}\left({0,T};\Hil{1}_{e_0}(\O)\right):=\left\{v\in\Lp{2}\left({0,T};\Hil{1}_{e_0}(\O)\right):\pdt v\in\Lp{2}\left({0,T};\Hil{1}_{e_0}(\O)\right)\right\}.
\]

\begin{Defn}[Weak solution of  (\ref{eq:PdBC-limit-problem})]\label{def:ws-PdBC-limit-problem}

\noindent
  We say a function pair $(\tu,\tv)$ with $\tu \in L^2( 0,T ; H^1_{e_{u_D}}(\Omega) )\cap \Lp{\infty}( {0,T} ; \Lp{2}(\O))$ and $\tv \in \Lp{\infty}( 0,T ; L^2( \Gamma ) )$ is a weak solution of \eqref{eq:PdBC-limit-problem}, if for all $\eta\in\Hil{1}\left({0,T} ;\Hil{1}_{e_0}(\O)\right)$  with $\eta(\cdot,T)=0$, we have
  \begin{equation}\label{eqn:PdBC-limit-weak-soln}
    \begin{aligned}
      &\int_0^T \left( \int_\O- \delO\tu\pdt\eta + \nabla \tu \cdot \nabla \eta \dd x
      + \int_\Gamma -\tv \pdt\eta \dd \sigma \right) \dd t
      = {\delO}\int_\O u^0\eta(\cdot,0)\dd x
      + \int_\G v^0\eta(\cdot,0)\dd \sigma \\
      &\mbox{and } \quad \tv \in \beta( \tu ) \mbox{ a.e. in } \Gamma \times (0,T).
    \end{aligned}
  \end{equation}
  We make the obvious modifications for the Neumann case.
  \end{Defn}
  }
\begin{The}[Convergence of the solution of (\ref{eq:eps-problem}) to a solution of (\ref{eq:PdBC-limit-problem})]
  \label{thm:PdBC-pbm}
  As $\delk=\delG\to0$ the  solution pair $(u,w)$ to {(\ref{eqn:wf_eps_problem})} converge  to a pair $(\tu,\tw)$ in the following topologies
\begin{align}
 u &\rightharpoonup \tu \quad \mbox{ in } L^2( 0, T; H^1_{e_{u_D}}( \Omega ) )\label{eqn:wcuh1pbcl}\\
\notag(u &\rightharpoonup \tu \quad \mbox{ in } L^2( 0, T; H^1( \Omega ) ) \text{ in the Neumann case})\\
  w &\rightharpoonup \tw \quad \mbox{ in } L^2( 0, T; \Lp{2}( \Gamma ) ),\label{eqn:weak_conv_PdBC_w}
\\
    u &\rightarrow \tu \quad \mbox{ in } \Lp{2}(\O\times{(0,T)}),\label{eqn:strong_conv_PdBC_u} \\
    u|_\Gamma &\rightarrow \tu|_\Gamma \quad \mbox{ in } \Lp{2}(\G\times{(0,T)}).\label{eqn:strong_conv_PdBC_u_G}
  \end{align}
   Moreover, the pair $\tu,\tv$, with $\tv=-\tw$  are the unique weak solution to \eqref{eq:PdBC-limit-problem} in the sense of Definition (\ref{eqn:PdBC-limit-weak-soln}).
 \end{The}  
  \begin{proof}
  As in the proof of Theorem \ref{thm:parab-pbm}, the uniform estimates of Lemma \ref{Lem:unif_est}  together with the compactness results of Lemma~\ref{lem:als-compact} and Lemma~\ref{lem:trace-compact} imply the weak and strong convergence results given in the theorem. 
  
   We now show that the limit pair $(\tu,\tv)$, with $\tv=-\tw$ are a weak solution of (\ref{eq:PdBC-limit-problem}).
  We start by noting that for all $\eta\in C^\infty(\O\times(0,T))$ with $\eta=0$ on $\partial_0\Omega\times(0,T)$ and $\eta(\cdot,T)=0$, we have
\begin{align*}
  & \int_0^T \int_\O- \delO u\pdt \eta+\nabla u\cdot\nabla \eta\dd x\dd t-{\delO}\int_\O u^0\eta(\cdot,0)\dd x    =\int_0^T\int_\G-\frac{1}{\delk}u w\eta\dd \sigma\dd t \\
  & \qquad =\int_0^T\int_\G- w\pdt\eta +\delG\nabla_\G w\cdot\nabla_\G\eta\dd \sigma\dd t-\int_\G w^0\eta(\cdot,0)\dd \sigma\\
  & \qquad =\int_0^T\int_\G- w\pdt\eta -\delG w\lap_\G\eta\dd \sigma\dd t-\int_\G w^0\eta(\cdot,0)\dd \sigma
\end{align*}

Letting $\delk=\delG\to0$, the convergence results (\ref{eqn:wcuh1pbcl})---(\ref{eqn:strong_conv_PdBC_u}) give
\begin{align*}
\int_0^T \int_\O- \delO \tu\pdt \eta+\nabla \tu\cdot\nabla \eta\dd x\dd t-{\delO}\int_\O u^0\eta(\cdot,0)\dd x&=\int_0^T\int_\G- \tw\pdt\eta\dd \sigma\dd t-\int_\G w^0\eta(\cdot,0)\dd \sigma,
\end{align*}
and hence with $\bar{v}=-\bar{w}$, we infer that
\begin{align}\label{eqn:barubarv2}
\int_0^T \int_\O- \delO \tu\pdt \eta+\nabla \tu\cdot\nabla \eta\dd x\dd t-{\delO}\int_\O u^0\eta(\cdot,0)\dd x-\int_0^T\int_\G\tv\pdt\eta\dd \sigma\dd t-\int_\G v^0\eta(\cdot,0)\dd \sigma=0.
\end{align}
A density argument yields that the above holds for all test functions $\eta$ in the spaces of Definition \ref{def:ws-PdBC-limit-problem}.
As $u,w\geq 0$ we have $\tu\geq0,\tv=-\tw\leq0$.
To check $\tv\in\beta(\tu)$ it remains to show that $\int_\G\tu\tv=0$.
This follows since
\begin{align*}
  \int_0^T \int_\G \tu\tv \dd\sigma \dd t
  & = \int_0^T \int_\G -\tu\tw \dd\sigma \dd t \\
  & = -\lim_{\delta_k, \delta_\Gamma \to 0} \int_0^T \int_\G ( \tu - u ) \tw + u ( \tw - w ) + u w \dd\sigma \dd t = 0,
\end{align*}
where we have used that the first term on the right hand side is zero since {$u \to \tu$ and $\tw\in L^2( 0, T; L^2( \Gamma ) )$} \eqref{eqn:strong_conv_PdBC_u_G}, \eqref{eqn:eps_pb_estimate_1}, the second term is zero since {$w \rightharpoonup \tw$} and $u$ is bounded in $L^2( 0, T; L^2(\Gamma) )$ \eqref{eqn:weak_conv_PdBC_w} and the final term is zero from the estimate \eqref{eq:uw_estimate}.

  To prove that the solution is unique we argue as follows. Let $(\tu_1,\tv_1)$ and $(\tu_2,\tv_2)$ be solutions of (\ref{eq:PdBC-limit-problem}) in the sense of Definition \ref{def:ws-PdBC-limit-problem}. We define $\theta^\tu(\cdot,t):=(\tu_1(\cdot,t)-\tu_2(\cdot,t)), \theta^\tv(\cdot,t):=(\tv_1(\cdot,t)-\tv_2(\cdot,t))$. 
  The pair $(\theta^\tu,\theta^\tv)$ satisfy
  \begin{align}\label{eqn:barthetaubarthetav}
\int_0^T \int_\O- \delO \theta^\tu\pdt \eta+\nabla \theta^\tu\cdot\nabla \eta\dd x\dd t-\int_0^T\int_\G\theta^\tv\pdt\eta\dd \sigma\dd t=0,
\end{align}
for all {$\eta\in\Hil{1}\left(0,T;\Hil{1}_{e_0}(\O)\right)$ with $\eta(\cdot,T)=0$}. For $t\in{(0,T)}$ we define $\theta^\tz(\cdot,t)=\int_t^T\theta^\tu(\cdot,s)\dd s$. Noting that $\theta^\tz$ is an admissible test function, we set $\eta=\theta^\tz$ in (\ref{eqn:barthetaubarthetav}) which gives
  \begin{equation*}
\delO\int_0^T \int_\O(\theta^\tu)^2\dd x\dd t-\int_0^T\frac{1}{2}\frac{\diff}{\diff t}\int_\O\lv\nabla\theta^\tz\rv^2\dd x\dd t+\int_0^T\int_\G\theta^\tv\theta^\tu\dd \sigma\dd t=0.
\end{equation*}
{
As $\theta^\tz(\cdot,T)=0$ we have
 \begin{equation*}
\delO\int_0^T \int_\O(\theta^\tu)^2\dd x\dd t+\frac{1}{2}\int_0^T\int_\O\lv\nabla\theta^\tu\rv^2\dd x\dd t+\int_0^T\int_\G\left(\tv_1-\tv_2\right)\left(\tu_1-\tu_2\right)\dd \sigma\dd t=0.
\end{equation*}
Recalling that $\tv_i\in\beta(\tu_i),i=1,2$, the monotonicity of $\beta$ gives}
 \begin{equation*}
 \Lpn{2}{\theta^\tu}{(0,T);\Hil{1}(\O)}^2=0.
\end{equation*}
Finally, (\ref{eqn:barthetaubarthetav}) and the above bound yield
\begin{equation*}
\int_0^T\int_\G\theta^\tv\pdt\eta\dd \sigma\dd t=0
\end{equation*}
for all $\eta$ that are admissible test functions in the sense of Definition \ref{def:ws-PdBC-limit-problem}. For any $\phi\in\Lp{2}{(0,T;\Hil{1/2}(\G))}$ we define $\mathbb{D}\phi$ such that
$\mathbb{D}\phi=\phi$ on $\G$, $\lap\mathbb{D}\phi=0$ in $\O$ and $\mathbb{D}\phi=0$ on $\partial_0\O$. Then we may take  $\eta(\cdot,t)=\int_t^T\mathbb{D}\phi(\cdot,s)\dd s$ as a test function in the above which gives 
\begin{equation*}
\int_0^T\int_\G\theta^\tv\phi\dd \sigma\dd t=0,
\end{equation*}
for all $\phi\in\Lp{2}{({0,T};\Hil{1/2}(\G))}$. Hence
\begin{equation*}
  \| \theta^\tv \|_{L^2\left( {0,T};\dual{\Hil{1/2}(\G)} \right)} = 0
\end{equation*}
which completes the proof of the theorem.
\end{proof}

\section{Elliptic limit problem with dynamic boundary condition ($\delO=\delG=\delk=0$)}
\label{sec:elliptic-limit-pb}

We now present a rigorous derivation of the elliptic problem with dynamic boundary conditions presented in \S \ref{sec:limit-model} as a limit 
of (\ref{eq:eps-problem}). As mentioned in \S \ref{sec:limit-model} we will only consider the case of Dirichlet boundary data. Specifically we show that as $\delO=\delG=\delk\to0$ the unique solution to (\ref{eq:eps-problem}) with Dirichlet boundary data, converges to the unique solution of the following problem.

\begin{Pbm}\label{pbm:EdBC-limit-problem}
Find {$\hu \colon \Omega \times (0,T) \to \mathbb{R}^+$ and $\hv \colon \Gamma \times [0,T) \to \mathbb{R}^-$} such that
\begin{subequations}
\label{eq:EdBC-limit-problem}
  \begin{align}
     -\Delta \hu  = 0 & \mbox{ in } \Omega \times (0,T) \\
    \nabla \hu \cdot \vec{\nu} + \partial_t \hv  = 0 & \mbox{ on } \Gamma \times (0,T) \\
    \hv \in \beta(\hu) & \mbox{ on } \Gamma \times (0,T) \\
    \hu  = u_D & \mbox{ on } \partial_0 \Omega \times (0,T) \\
    \hv(\cdot, 0)  = v^0(\cdot)\leq 0 & \mbox{ on } \Gamma,
  \end{align}
\end{subequations}
where $\beta \colon \mathbb{R} \to \{0,1\}^{\mathbb{R}}$ is the set valued function defined in (\ref{eqn:beta_def}).
\end{Pbm}

{
\begin{Defn}[Weak solution of  (\ref{eq:EdBC-limit-problem})]\label{def:ws-EdBC-limit-problem}

\noindent
  We say a function pair $(\hu,\hv)$ with $\hu \in L^2( 0, T; H^1_{e_{u_D}}(\Omega) )$ and $\hv \in \Lp{\infty}( 0, T; L^2( \Gamma ) )$ is a weak solution of \eqref{eq:EdBC-limit-problem}, if for all $\eta\in\Hil{1}\left({0,T};\Hil{1}_{e_0}(\O)\right)$ with $\eta(\cdot,T)=0$ on $\G$, we have
  \begin{equation}\label{eqn:EdBC-limit-weak-soln}
    \begin{aligned}
      &\int_0^T \left( \int_\O \nabla \hu \cdot \nabla \eta \dd x
      -\int_\Gamma \hv \pdt\eta \dd \sigma \right) \dd t-\int_\G v^0\eta(\cdot,0)\dd \sigma= 0,\\
      &\mbox{and } \quad \hv \in \beta( \hu ) \mbox{ a.e. in } \Gamma \times (0,T).
    \end{aligned}
  \end{equation}
  \end{Defn}
  }
  
  The strategy of passing to the limit follows that of \S \ref{sec:parab-limit-pb}.

\begin{The}[Convergence of the solution of (\ref{eq:eps-problem}) to a solution of (\ref{eq:EdBC-limit-problem})]
  \label{thm:EdBC-pbm}
  As $\delO=\delG=\delk\to0$ the  solution pair $(u,w)$ to {(\ref{eqn:wf_eps_problem})} converge  to a pair $(\hu,\hw)$ in the following topologies
\begin{align}
 u &\rightharpoonup \hu \quad \mbox{ in } L^2( 0, T; H^1_{e_{u_D}}( \Omega ) ) \label{eqn:weak_conv_EdBC_u}
\\
  w &\rightharpoonup \hw \quad \mbox{ in } L^2( 0, T; \Lp{2}( \Gamma ) ),\label{eqn:weak_conv_EdBC_w}
  \\
  w &\rightarrow \hw \quad \mbox{ in } L^2( 0, T; \Hil{-1/2}( \Gamma ) ),\label{eqn:strong_conv_EdBC_w}
  \end{align}
   Moreover, the pair $\hu,\hv$, with $\hv=-\hw$  are the unique solution to Problem (\ref{eq:EdBC-limit-problem}) in the sense of Definition (\ref{eqn:EdBC-limit-weak-soln}). 
    \end{The}
  \begin{proof}
  As in the proof of Theorems \ref{thm:parab-pbm} and \ref{thm:PdBC-pbm}, the  estimates of Lemma \ref{Lem:unif_est}, specifically \eqref{eqn:eps_pb_estimate_1}   together with  the compactness results recalled in \eqref{eq:weak-conv} imply the convergence results (\ref{eqn:weak_conv_EdBC_u}) and (\ref{eqn:weak_conv_EdBC_w}).  The strong convergence result (\ref{eqn:strong_conv_EdBC_w}) follows due to the Lions-Aubin-Simon compactness theory (Lemma~\ref{lem:als-compact}) together with the estimate on the time translates of $w$ \eqref{eq:time_trans_estimate} and the compact embedding of $\Lp{2}(\G)$ into $\Hil{-1/2}(\G)$ shown in Lemma~\ref{lem:L2H-12-compact}.

  The fact that the limits $\hu,\hv=-\hw$ satisfy
  \begin{align*}
     \int_0^T \left( \int_\O\nabla\hu\cdot\nabla\eta\dd x-\int_\G\hv\pdt\eta \dd \sigma \right) \dd t-\int_\G v^0\eta(\cdot,0)\dd \sigma=0,
  \end{align*}
for all $\eta$ as in Definition \ref{def:ws-EdBC-limit-problem},  follows from the weak convergence results (\ref{eqn:weak_conv_EdBC_u}) and (\ref{eqn:weak_conv_EdBC_w}) together with an analogous density argument   to that used in the proof of Theorem \ref{thm:PdBC-pbm}.   It remains to check $\hv\in\beta(\hu)$. As previously we have $\hu\geq0$ and $\hv\leq0$. The fact that $\hu,\hv\in\Lp{2}({\G \times (0,T)})$, the strong convergence result (\ref{eqn:strong_conv_EdBC_w}), the weak convergence result (\ref{eqn:weak_conv_EdBC_u}) which implies weak convergence of the trace of $u$ in $L^2( 0, T; H^{1/2}( \G ) )$ and the estimate (\ref{eq:uw_estimate}) imply
\[
\int_0^T \int_\Gamma \hu \hv \dd \sigma \dd t
=\int_0^T\dualp{}{\hv}{\hu }{{H^{1/2}(\Gamma)}}\dd t
=0,
\]
and hence $\hv\in\beta(\hu)$.
      
  Similarly the uniqueness argument mirrors that used in the proof of Theorem  \ref{thm:PdBC-pbm}. Letting $(\hu_1,\hv_1)$ and $(\hu_2,\hv_2)$ be two solutions of (\ref{eq:EdBC-limit-problem}) in the sense of Definition \ref{def:ws-EdBC-limit-problem} and setting $\theta^\hu(\cdot,t):=(\hu_1(\cdot,t)-\hu_2(\cdot,t)), \theta^\hv(\cdot,t):=(\hv_1(\cdot,t)-\hv_2(\cdot,t))$. 
  The pair $(\theta^\hu,\theta^\hv)$ satisfy
    \begin{align}\label{eqn:hthetauhthetav}
\int_0^T \int_\O\nabla \theta^\hu\cdot\nabla \eta\dd x\dd t-\int_0^T\int_\G\theta^\hv\pdt\eta\dd \sigma\dd t=0,
\end{align}
for all {$\eta\in\Hil{1}\left(0,T;\Hil{1}_{e_0}(\O)\right)$ with $\eta(\cdot,T)=0$ on $\G$}. { For $t\in{(0,T)}$ we define $\theta^\hz(\cdot,t)=\int_t^T\theta^\hu(\cdot,s)\dd s$. Noting $\theta^\hz$ is an admissible test function, we set $\eta=\theta^\hz$ in (\ref{eqn:hthetauhthetav}) which gives, using the fact that $\theta^\tz(\cdot,T)=0$ , 
 \begin{equation*}
\frac{1}{2}\int_0^T\int_\O\lv\nabla\theta^\hu\rv^2\dd x\dd t+\int_0^T\int_\G\left(\hv_1-\hv_2\right)\left(\hu_1-\hu_2\right)\dd \sigma\dd t=0.
\end{equation*}
Recalling that $\hv_i\in\beta(\hu_i),i=1,2$, the monotonicity of $\beta$, together with the Poincare inequality as $\theta^\hz\in\Hil{1}_{e_0}(\O)$ gives}
 \begin{equation*}
 \Lpn{2}{\theta^\tu}{(0,T);\Hil{1}(\O)}^2=0.
\end{equation*}
Finally, via the same argument used in the proof of Theorem  \ref{thm:PdBC-pbm}, (\ref{eqn:hthetauhthetav}) and the above bound yield
\begin{equation*}
\Lpn{2}{\theta^\hv}{(0,T);\Hil{-1/2}(\G)}=0,
\end{equation*}
which completes the Proof of the Theorem.
  \end{proof}

\section{Degenerate parabolic equations}

{
In this Section we give alternative formulations of the limiting problems of \S \ref{sec:limit_parabolic}-\ref{sec:elliptic-limit-pb}. 
Solutions to the problems \ref{parpar}, \ref{pardyn} and \ref{elldyn}  introduced in this section are solutions of problems \ref{pbm:parabolic-problem}, \ref{pbm:PdBC-limit-problem} and \ref{pbm:EdBC-limit-problem} respectively.

The structure of the equations is revealed by writing them as abstract degenerate parabolic equations holding on the surface $\Gamma$. Doing this, one observes that the problems are the analogues of the Hele-Shaw 
and steady one phase Stefan problems   with the half-Laplacian  replacing the usual Laplacian $(-\lap)$ (see \citep{crowley1979weak,EllOck82} for further details on the formulation of the Hele-Shaw 
and one phase Stefan problems).
}

First, we define a parabolic extension operator $$P^{\delta_\Omega} \colon L^2( 0, T; H^{1/2}(\Gamma) ) \to L^2( 0, T; H^1_{e_{u_D}}( \Omega ) ),$$
 or $$P^{\delta_\Omega} \colon L^2( 0, T; H^{1/2}(\Gamma) ) \to L^2( 0, T; H^1( \Omega ) )$$ in the Neumann case.
We fix $\eta \in L^2( 0, T ; H^{1/2}( \Gamma ) )$ we define $P^{\delta_\Omega} \eta$ to be the unique solution of
\begin{equation}
  \label{eq:Peps-defn}
  \begin{aligned}
    \delta_\Omega \partial_t ( P^{\delta_\Omega} \eta ) - \Delta ( P^{\delta_\Omega} \eta )
    & = 0 && \mbox{ in } \Omega \times ( 0, T ) \\
    P^{\delta_\Omega} \eta
    & = \eta && \mbox{ on } \Gamma \times ( 0, T ) \\
    P^{\delta_\Omega} \eta
    = 0 \mbox{ or }
    \nabla ( P^{\delta_\Omega} \eta ) \cdot \vec{\nu}_\Omega & = 0 && \mbox{ on } \partial_0 \Omega \times ( 0, T ) \\
    ( P^{\delta_\Omega} \eta )( \cdot, 0 )
    & = 0 && \mbox{ in } \Omega.
  \end{aligned}
\end{equation}

This allows us to define a parabolic Dirichlet to Neumann (DtN) map $\mathcal A^{\delta_\Omega} \colon L^2( 0, T; H^{1/2}( \Gamma ) ) \to L^2\left(  ( 0, T ) ; \dual{H^{1/2}( \Gamma )} \right)$ by
\begin{equation}
  \label{eq:Aeps-defn}
  \mathcal A^{\delta_\Omega} \eta := \nabla ( P^{\delta_\Omega} \eta ) \cdot \vec\nu
  \qquad \mbox{ for } \eta \in L^2( 0, T; H^{1/2}( \Gamma ) ).
\end{equation}

Next, we define a new elliptic extension  operator $P^0 \colon L^2( 0, T; H^{1/2}(\Gamma) ) \to L^2( 0, T; H^1_{e_{u_D}}( \Omega ) )$, which formally is a limit of  {$P^{\delO}$} from \eqref{eq:Peps-defn}.
For $\eta \in L^2( 0, T ; H^{1/2}( \Gamma ) )$ we define $P^{0} \eta$ to be the unique solution of
\begin{equation}
  \label{eq:P0-defn}
  \begin{aligned}
    - \Delta ( P^{0} \eta )
    & = 0 && \mbox{ in } \Omega \times ( 0, T ) \\
    P^{0} \eta
    & = \eta && \mbox{ on } \Gamma \times ( 0, T ) \\
    P^{0} \eta
     & = 0 && \mbox{ on } \partial_0 \Omega \times ( 0, T ).
  \end{aligned}
\end{equation}

This allows us to define the elliptic DtN  map $\mathcal A^{0} \colon L^2( 0, T; H^{1/2}( \Gamma ) ) \to L^2\left(  ( 0, T ) ; \dual{H^{1/2}( \Gamma )} \right)$ by
\begin{equation}
  \label{eq:A0-defn}
  \mathcal A^{0} \eta := \nabla ( P^{0} \eta ) \cdot \vec\nu
  \qquad \mbox{ for } \eta \in L^2( 0, T; H^{1/2}( \Gamma ) ).
\end{equation}
We note that the operator $\mathcal{A}^0$ may also be viewed as  the half-Laplacian $(-\lap_\G)^{1/2}$ for functions on $\G$ \citep{caffarelli2007extension}.

It is also convenient to introduce extensions of the data. First we introduce $U_D^{\delta_\Omega}$ as the solution of the parabolic problem
\begin{equation}
  \label{eq:DirichData-defn}
  \begin{aligned}
    \delta_\Omega \partial_t  U_D^{\delta_\Omega}  - \Delta  U_D^{\delta_\Omega}    & = 0 && \mbox{ in } \Omega \times ( 0, T ) \\
   U_D^{\delta_\Omega}
    & = 0 && \mbox{ on } \Gamma \times ( 0, T ) \\
   U_D^{\delta_\Omega}
    = u_D \mbox{ or }
    \nabla (U_D^{\delta_\Omega}) \cdot \vec{\nu}_\Omega & = 0 && \mbox{ on } \partial_0 \Omega \times ( 0, T ) \\
   U_D^{\delta_\Omega}( \cdot, 0 )
    & = 0 && \mbox{ in } \Omega.
  \end{aligned}
\end{equation}
Second we have $U_D$ as the solution of an elliptic problem%
\begin{equation}
  \label{eq:DirichData0-defn}
  \begin{aligned}
    - \Delta  U_D    & = 0 && \mbox{ in } \Omega \\
   U_D
    & = 0 && \mbox{ on } \Gamma \\
   U_D
    & = u_D && ~~\mbox{ or }
    \nabla U_D \cdot \vec{\nu}_\Omega & = 0 && \mbox{ on } \partial_0 \Omega.   \end{aligned}
\end{equation}
In the Neumann case  we have $U_D^{\delta_\Omega}=U_D=0$.

Third, we introduce $U_I^{\delta_\Omega}$ as the solution of the parabolic problem
\begin{equation}
  \label{eq:InitData-defn}
  \begin{aligned}
    \delta_\Omega \partial_t  U_I^{\delta_\Omega}  - \Delta  U_I^{\delta_\Omega}    & = 0 && \mbox{ in } \Omega \times ( 0, T ) \\
   U_I^{\delta_\Omega}
    & = 0 && \mbox{ on } \Gamma \times ( 0, T ) \\
   U_I^{\delta_\Omega}
    =0  \mbox{ or }
    \nabla (  U_I^{\delta_\Omega}) \cdot \vec{\nu}_\Omega & = 0 && \mbox{ on } \partial_0 \Omega \times ( 0, T ) \\
   U_I^{\delta_\Omega}( \cdot, 0 )
    & = u_0 && \mbox{ in } \Omega.
  \end{aligned}
\end{equation}
Note that as $\delta_\Omega \rightarrow 0$ that
$U^{\delta_\Omega}_I \rightarrow 0$ and $U_D^{\delta_\Omega}\rightarrow U_D$ in $L^2(0,T;H^1(\Omega))$.
 Finally, we write
  $L$ for $-\Delta_\Gamma$ as an operator $L^2( 0, T; H^1( \Gamma ) ) \to L^2\left( 0, T; \dual{H^{1}( \Gamma )}\right)$
%
%\margnote{Can we name the operators $A^{\delta_\Omega}$ and $A^0$}
%
\begin{Pbm}[Fast reaction limit, $\delk=0$]\label{parpar}
  Find $\bar{u}\geq0$ and $\bar{v}\leq0$ with $\bar{u} \in L^2( 0, T; H^{1/2}( \Gamma ) )$ and $\bar{v} \in L^2( 0, T; H^1( \Gamma ) )$ with $\partial_t \bar{v} \in L^2\left( 0, T; \dual{H^{1}( \Gamma )}\right)$ such that
  \begin{equation}
    \begin{aligned}
      \partial_t \bar{v} + \delta_\Gamma L \bar{v} + \mathcal A^{\delta_\Omega} \bar{u} + \nabla (U_D^{\delta_\Omega}+U_I^{\delta_\Omega}) \cdot \nu
      & = 0 && \mbox{ in } L^2\left( 0, T; \dual{H^{1}( \Gamma )}\right) \\
      \bar{v} \in \beta( \bar{u} ) & && \mbox{ on } \Gamma \times ( 0, T ) \\
      \bar{v}( \cdot, 0 ) & = v^0 && \mbox{ in } \Omega.
    \end{aligned}
  \end{equation}
\end{Pbm}

\begin{Pbm}[Bulk parabolic limit equation with dynamic boundary condition, $\delk=\delG=0$]\label{pardyn}

  Find $\tu\geq0$ and $\tv\leq0$ with $\tu \in L^2( 0, T; H^{1/2}( \Gamma ) )$ and $\tv \in L^2( 0, T; L^2( \Gamma ) )$ with $\partial_t \tv \in L^2\left( 0, T; \dual{H^{1}( \Gamma )}\right)$ such that
  \begin{equation}
    \begin{aligned}
      \partial_t \tv + \mathcal A^{\delta_\Omega} \tu + \nabla (U_D^{\delta_\Omega}+U_I^{\delta_\Omega}) \cdot \nu
      & = 0 && \mbox{ in } L^2\left( 0, T; \dual{H^{1}( \Gamma )}\right) \\
      \tv \in \beta( \tu ) & && \mbox{ on } \Gamma \times ( 0, T ) \\
      \tv( \cdot, 0 ) & = v^0 && \mbox{ in } \Omega.
    \end{aligned}
  \end{equation}
\end{Pbm}

\begin{Pbm}[Elliptic equation with dynamic boundary condition, $\delk=\delG=\delO=0$] \label{elldyn}
  Find $\hu\geq0$ and $\hv\leq0$ with $\hu \in L^2( 0, T; H^{1/2}( \Gamma ) )$ and $\hv \in L^2( 0, T; L^2( \Gamma ) )$ with $\partial_t \hv \in L^2\left( 0, T; \dual{H^{1}( \Gamma )}\right)$ such that
  \begin{equation}
    \begin{aligned}
      \partial_t \hv + \mathcal A^{0} \hu + \nabla U_D \cdot \nu
      & = 0 && \mbox{ in } L^2\left( 0, T; \dual{H^{1}( \Gamma )}\right) \\
      \hv \in \beta( \hu ) & && \mbox{ on } \Gamma \times ( 0, T ) \\
      \hv( \cdot, 0 ) & = v^0 && \mbox{ in } \Omega.
    \end{aligned}
  \end{equation}
\end{Pbm}

%%%%%%%%%%% SECTION %%%%%%%%%%%%%%%%%%%%%%%%%%
%%%%%%%%%%%%%%%%%%%%%%%%%%%%%%%%%%%%%%%%%%%%
\section{Variational inequality formulation}
\label{sec:var-ineq}

%\begin{Rem}[Connections with the (steady) one-phase Stefan and Hele-Shaw problems]\label{Rem:FBP}
{
Similarly to the Hele-Shaw and one phase Stefan  problems, that may  be reformulated as  variational inequalities via an integration in time
  \citep{duvaut1973resolution,ell80,elliott1981variational,rodrigues1987variational},  via integrating in time, the systems (\ref{eq:PdBC-limit-problem}) and (\ref{eq:EdBC-limit-problem})  and  Problems  \ref{pardyn} and   \ref{elldyn} may be reformulated, respectively,  as   parabolic and elliptic variational inequalities of obstacle type. The obstacle problem  lies on the surface $\Gamma$ and is a consequence of the complementarity  which is maintained after an  integration with respect to time and noting that 
this integration commutes with the operators $\mathcal A^{\delta_\Omega}$ and  $\mathcal A^{0}$.}
%For simplicity, we set $u_D = 1$.

We set
\begin{equation}\label{eqn:zdef}
z(\cdot,t)=\int_0^t\hu(\cdot,s)\dd s,
\end{equation}
where $\hu$ satisfies (\ref{eq:EdBC-limit-problem}). We find it convenient to introduce $Z_D^{\delta_{\Omega}}$ as
\begin{equation}
Z_D^{\delO}(\cdot,t)=tU_D.
\end{equation}

\begin{comment} the solution of the elliptic problem
%
\begin{equation}
  \begin{aligned}
    - \Delta  Z_D    & = 0 && \mbox{ in } \Omega \times(0,T) \\
    %
   Z_D
    & = 0 && \mbox{ on } \Gamma \times(0,T) \\
    %
   Z_D
    & = t && \mbox{ on } \partial_0 \Omega\times(0,T).
  \end{aligned}
\end{equation}
\end{comment}

Proceeding formally, we claim that if the pair $(\hu,\hv)$ satisfy   \ref{eq:PdBC-limit-problem} (or   \eqref{eq:EdBC-limit-problem} with 
$ \delta_{\Omega} =0$) then the pair $(z,\hv)$  satisfy the following problem
\begin{Pbm}
  For each $t \in (0,T)$, find $z(t) \in H^1(\Omega)$ and $\hv(t) \in L^2(\Gamma)$ such that
  \begin{subequations}\label{eq:z-problem-strong}
    \begin{align}
    \delta_{\Omega} \pdt z  - \delta_{\Omega} u_0 - \Delta z & = 0 && \mbox{ in } \Omega \\
{      \nabla z \cdot \vec{\nu} +  \hv - v^0 }&{= 0 }&& \mbox{ on } \Gamma \\
      \hv & \in \beta(z) && \mbox{ on } \Gamma \\
      z & = Z_D && \mbox{ on } \partial_0 \Omega.
    \end{align}
  \end{subequations}
\end{Pbm}

We check the condition $\hv \in \beta(z)  \mbox{ on } \Gamma \times (0,T)$. The remaining conditions follow formally from integration in time of \eqref{eq:EdBC-limit-problem}.
Let $\chi_B$ denote the characteristic function of the set $B$, then we have
\[
\int_\G \hv\chi_{z>0} \dd \sigma = \int_\G \hv(\chi_{z>0}-\chi_{\hu>0}) \dd \sigma + \int_\G \hv\chi_{\hu>0} \dd \sigma.
\]
Noting that $\chi_{z>0}\geq \chi_{\hu>0}$ as $\hu\geq 0$ and recalling $\hv\leq 0$ we have
\[
\int_\G \hv\chi_{z>0} \dd \sigma \geq\int_\G \hv\chi_{\hu>0} \dd \sigma =0.
\]
as $\hv\in\beta(\hu)$. Finally as $\hv\leq0$ and $\hz\geq0$ this yields $\hv\in\beta(z)$.

We now show that \eqref{eq:z-problem-strong}, in the case  $\delta_{\Omega} =0$, may be formulated as an 
elliptic variational inequality. For all $\eta\in\Hil{1}_{e_0}(\O)$
\begin{equation}
  0 = \int_\O-\lap z\eta \dd x
  = \int_\O\nabla z\cdot\nabla \eta \dd x -\int_\G\nabla z\cdot\normal\eta \dd \sigma.
\end{equation}
Thus defining the convex set
\[
K_t:=\{\eta \in \Hil{1}_{e_{Z_D(\cdot,t)}}(\O)|~\eta\geq 0~\mbox{on}~~\Gamma\}.
\] 
We see that for any $\eta\in K_t$ we have
\begin{equation}
\begin{aligned}
\int_\O\nabla z\cdot\nabla( \eta-z) \dd x &=\int_\G\nabla z\cdot\normal(\eta-z) \dd \sigma \\
&=\int_\G(v^0-v)(\eta-z) \dd \sigma.
  \end{aligned}
\end{equation}
Now since $z\geq0,v\leq0$ and $zv=0$ we arrive at the following elliptic variational inequality where time enters as a parameter, find $z\in K_t$ such that
\begin{equation}\label{eqn:EVI}
\begin{aligned}
\int_\O\nabla z\cdot\nabla( \eta-z) \dd x \geq\int_\G v^0(\eta-z) \dd \sigma \myall\eta \in K_t.
  \end{aligned}
\end{equation}

The same argument outlined above yields that if $z$ is defined by (\ref{eqn:zdef}) with $\hu$ replaced by $\tu$, the unique solution to the parabolic problem (\ref{eq:PdBC-limit-problem}) then
$z$ satisfies the parabolic variational inequality,  find $z\in K_t$ such that
\begin{equation}\label{eqn:PVI}
\begin{aligned}
\int_\O\delO\pdt z\eta+\nabla z\cdot\nabla( \eta-z) \dd x \geq\int_\O\delO u^0(\eta-z) \dd x +\int_\G v^0(\eta-z) \dd \sigma \myall\eta \in K_t.
  \end{aligned}
\end{equation}

 We may also integrate the appropriate degenerate parabolic problems in time yielding for example in the case $\delO=0$
\begin{equation}\label{eqn:ell_VI_DtN}
    \begin{aligned}
      \mathcal A^{0} z + \nabla Z_D \cdot \nu -v^0&=-\hat v && \mbox{ on } \Gamma\\
\hat v \le 0,  ~~    z &\geq 0, ~~z\hat v=0 && \mbox{ on } \Gamma
    \end{aligned}
\end{equation}
and obtain the elliptic variational inequality from this calculation.
%We denote by $\A:\Hil{1/2}(\G)\to\Hil{-1/2}(\G)$ the so called Dirichlet-to-Neumann (DtN) map defined
%as follows, for $\phi\in\Hil{1}(\O)$
%\begin{equation}
%\int_\G \A \phi\eta \dd \sigma = \int_\G \nabla \phi\cdot\normal\eta  \dd \sigma \myall\eta\in \Hil{1}_{e_0}(\O).
%\end{equation}
%Thus (\ref{eqn:EVI}) may be written as  find $z\in K_t$ s.t,
%\begin{equation}\label{eqn:ell_VI_DtN}
%\begin{aligned}
%\int_\G\A z( \eta-z) \dd \sigma \geq\int_\G v^0(\eta-z)  \dd \sigma \myall\eta \in K_t.
%  \end{aligned}
%\end{equation}

% \margnote{TR: how does this all relate to section 8? CV: Tried to relate to sec. 8, does this make sense?}

\section{Numerical experiments}
We now present some numerical simulations that support the theoretical results of the previous sections and illustrate a robust numerical method for the simulation of coupled bulk-surface systems of equations. We employ a piecewise linear coupled bulk surface finite element method for the approximation. The method is based on the coupled bulk-surface finite element method  proposed and analysed (for linear elliptic systems) by \citet{elliott2013finite}.
\subsection{Coupled bulk-surface finite element method}
We define computational domains $\O_h$ and $\G_h$ by requiring that $\O_h$ is a polyhedral approximation to $\O$ and we set $\G_h=\partial\O_h \setminus \partial_0\O_h$, i.e., $\G_h$ is the interior boundary of the polyhedral domain $\O_h$. We assume that $\O_h$ is the union of $n+1$ dimensional simplices (triangles for $n=1$ and tetrahedra for $n=2$) and hence the faces of $\G_h$ are $n$ dimensional simplices. 
%See Figure \ref{fig:triang_example} for a sketch of the setup.

 We define $\T_h$ to be a triangulation of $\O_h$ consisting of  closed simplices. Furthermore, we assume the triangulation is such that for every $k\in\T_h$, $k\cap\G_h$ consists of at most one face of $k$. 
 We define the bulk and surface finite element spaces $\Vc{\gamma}, \gamma\in\Reals$ and $\Sc$ respectively by
 \[
 \Vc{\gamma}=\left\{\Phi\in C(\O_h):\Phi=\gamma\text{ on }\partial_0\O_h\text{ and }\Phi\vert_k\in\mathbb{P}^1(k),\myall k\in\T_h\right\},
 \]
and
\[
 \Sc=\left\{\Psi\in C(\Gc):\Psi\vert_s\in\mathbb{P}^1(s),\myall k\in\T_h\text{ with }s=k\cap\Gc\neq\emptyset\right\}.
\]
\subsection{Numerical schemes}
In the interests of brevity we only present numerical schemes for the  approximation of (\ref{eqn:wf_eps_problem}) and (\ref{eqn:EVI}), i.e., the original problem with Dirichlet boundary conditions and the elliptic variational inequality respectively. For simplicity we take $u_D=1$. The modifications for the Neumann case and the parabolic variational inequality are standard. We divide the time interval $[0,1]$ into {$M$} sub-intervals $0=t_0<t_1<\dots<t_{M-1}<t_M=1$ and denote by $\tau:=t_m-t_{m-1}$ the time step, which for simplicity is taken to be uniform. For a time discrete sequence, we introduce the shorthand $f^m:=f(t_m)$.

 For the time discretisation of (\ref{eqn:wf_eps_problem}) we employ an IMEX method where the diffusion terms are treated implicitly whilst the reaction terms are treated explicitly \citep{lakkis2013implicit} which leads to two decoupled parabolic systems. 
 The fully discrete scheme for the approximation of (\ref{eqn:wf_eps_problem}) reads as follows, for $m=1,\dots,M$ find $(U^m,W^m)\in(\Vc{u_D}\times\Sc)$ such that for all $(\Phi,\Psi)\in(\Vc{0}\times\Sc)$
\begin{equation}\label{eqn:eps-fd-scheme}
\begin{split}
\int_{\O_h}\delO\frac{1}{\tau}\left(U^{m}-U^{m-1}\right)\Phi\dd x+\int_{\O_h}\nabla U^{m+1}\cdot\nabla\Phi\dd x&=-\frac{1}{\delk}\int_{\G_h}\Lagrange\left[U^{m-1}W^{m-1}\right]\Phi\dd \sigma_h\\
\int_{\G_h}\frac{1}{\tau}\left(W^{m}-W^{m-1}\right)\Psi\dd \sigma_h+\int_{\G_h}\delG\nabla_{\G_h} W^{m+1}\cdot\nabla_{\G_h}\Psi\dd \sigma_h&=-\frac{1}{\delk}\int_{\G_h}\Lagrange\left[U^{m-1}W^{m-1}\right]\Psi\dd \sigma_h\\
U^0=\Clement u^0\quad\text{and}\quad W^0&=\Lagrange w^0,
\end{split}
\end{equation}
where $\Clement:C(\O_h)\to\Vc{u_D}$ and $\Lagrange:C(\G_h)\to\Sc$ denote the Lagrange interpolants into the bulk and surface finite element spaces respectively.

 For the approximation of (\ref{eqn:EVI}), we note that at each time step a single elliptic variational inequality must be solved, {the solution of which} may be obtained independently of the values at other times.  Introducing the bulk finite element space
 \[
 \Kc{t}=\left\{\Phi\in C(\O_h):\Phi\geq0, \Phi=t\text{ on }\partial_0\O_h\text{ and }\Phi\vert_k\in\mathbb{P}^1(k),\myall k\in\T_h\right\},
 \]
the fully discrete scheme for the approximation of (\ref{eqn:EVI}) reads, for $m=1,\dots,N,$ find $Z^m\in\Kc{t}$ such that for all $\Phi\in \Kc{t}$
\begin{equation}\label{eqn:z-fd-scheme}
\int_{\O_h}\nabla Z^m\cdot\nabla(\Phi- Z^m)\dd x\geq\int_{\G_h}v^0(\Phi-Z^m)\dd \sigma_h.
\end{equation}
For a discussion of the analysis of discretisation of this problem we refer to \citet{nochetto2014convergence}.

\subsection{2D simulations}\label{sec:sim-2d}
For all the simulations we use of the finite element toolbox \ALBERTA \citep{schmidt2005design}. For the visualisation  we use \PARAVIEW \citep{henderson2004paraview}.  We start with the case where $\O$ is two dimensional, i.e., the surface $\G$ is a curve.  We set $\partial_0\O$ to be the boundary of the square of length four centred at the origin and  define the surface of the cell $\G$ by the level set function $\G=\{\vec x\in\Reals^2\vert (x_1+0.2-x_2^2)^2+x_2^2-1=0\}$.  We generated a bulk triangulation of the domain $\O_h$ and the corresponding induced surface triangulation of $\G_h$ using \Program{DistMesh}~\citep{persson2004simple}. We used a graded mesh-size with small elements near $\G$, the bulk mesh had 2973 DOFs (degrees of freedom) and the induced surface triangulation had 341 DOFs.   Figure \ref{fig:2d_triang_example} shows the mesh used for all the 2D simulations.

\begin{figure}
  \centering
  \includegraphics[trim = 10mm 90mm 10mm 10mm,  clip, width=0.5\textwidth]{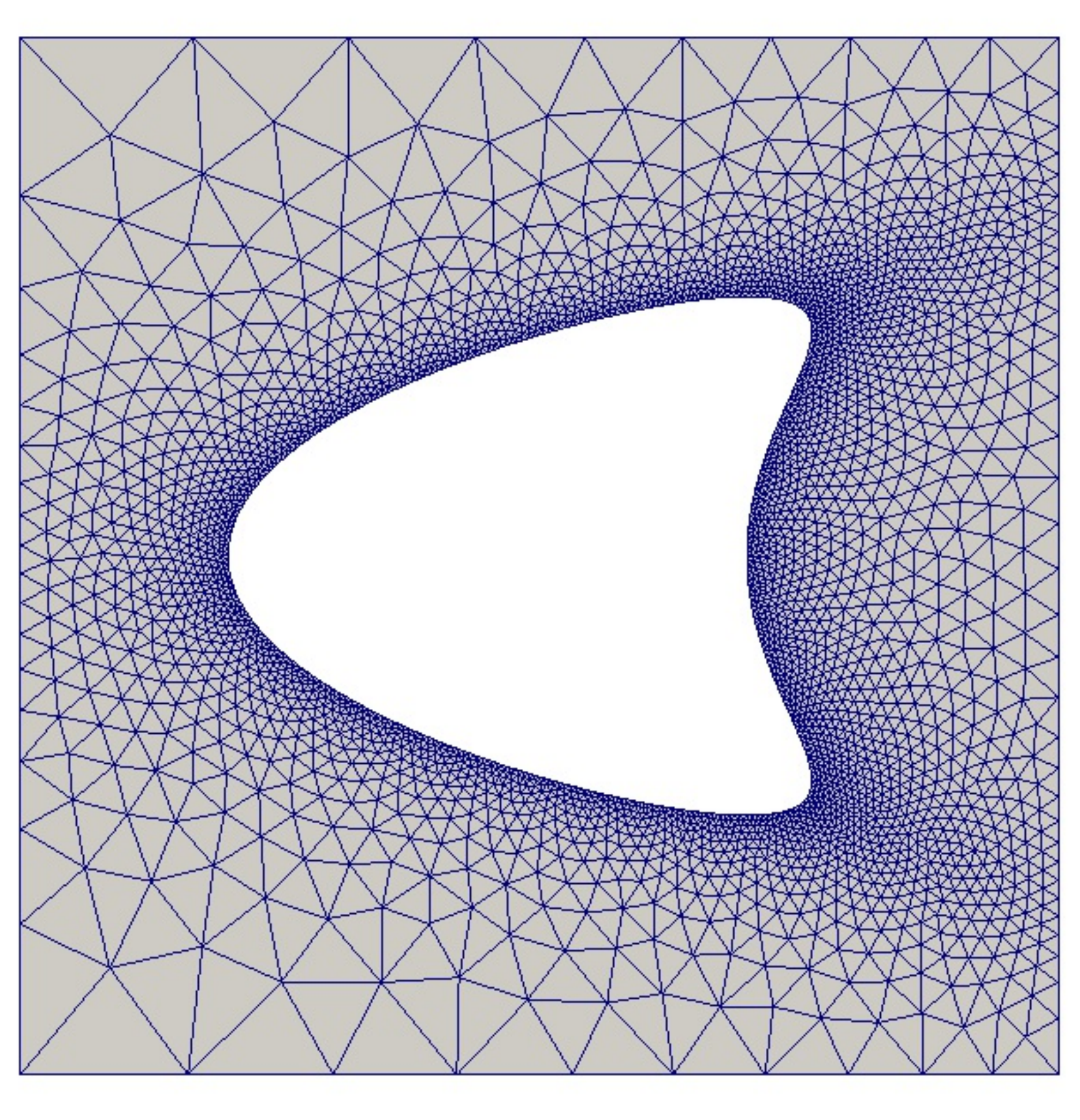}
  \caption{The computational domain for the simulations in $2d$ of \S  \ref{sec:sim-2d}, generated using \Program{DistMesh} \citep{persson2004simple}.}
  \label{fig:2d_triang_example}
\end{figure}

In light of the theoretical results of the previous sections, we consider (\ref{eqn:wf_eps_problem}) with $\eps=\delk=\delO=\delG=10^{-1},10^{-2}$ and $10^{-3}$ respectively and compare the simulation results with the results of simulations of (\ref{eqn:EVI}). For the problem data for (\ref{eqn:wf_eps_problem}), we took the end time $T=0.7$ and $u_D=1$.   For the initial data for (\ref{eqn:wf_eps_problem}) we took $w^0=\max(0,\cos(\pi x_2)+\sin(\pi x_1))$, $\vec x\in\G$ and $u^0=u_D=1$  and for (\ref{eqn:EVI}) we took $v^0=-w^0$. For each of the simulations of (\ref{eqn:wf_eps_problem}) we used same uniform time step, $\tau=10^{-8}$. In order to compare the solutions of (\ref{eqn:wf_eps_problem}) with those of (\ref{eqn:EVI}), we solve (\ref{eqn:EVI}) at a series of distinct times and post-process the solution to obtain $u=\pdt z$ and $w=\nabla z\cdot \normal +w^0$.

Snapshots of the solution $Z$ to (\ref{eqn:EVI}) at a series of distinct times is shown in Figure \ref{fig:2d-res-z}. We note that to post-process $U^{t_m}:=(Z^{t_m}-Z^{t_m-\tau})/\tau$ we solve (\ref{eqn:EVI}) at  $t_m$ and $t_m-\tau$ fixing $\tau=10^{-2}$. We stress that as time simply enters as a parameter in (\ref{eqn:EVI}) its solution may be approximated independently at any given time, it is simply for the recovery of $U$ for which we require values of $Z$ at a previous time.

Figure \ref{fig:2d-res-uv} shows snapshots of the simulated $U$ and $W$. Initially we observe depletion of the bulk ligand concentration $U$ in each case near regions where the initial data for the surface receptors $w^0$ is large. As time progresses we observe a decay in $W$ with larger decreases in $W$ observed for smaller values of $\eps$. Similarly the speed at which the system approaches the steady state corresponding to constant solutions  $u=1$ and $w=0$ appears to be an increasing function of $\eps$. The post-processed $U$ and $W$ obtained from the solution to (\ref{eqn:EVI}) show qualitatively similar behaviour  with faster dynamics towards the steady state which is attained by the end time $t=0.7$, with none of the simulations with $\eps>0$ attaining this steady state by $t=0.7$. In order to illustrate more clearly the formation of the free boundary as $\eps\to0$, in Figure \ref{fig:2d-plot-uv} we show plots of $W$ and the trace of $U$ over the surface $\G_h$. We observe that $\eps\to 0$ the supports of the trace of $U$ and $W$ become disjoint and their profiles approach that obtained on post-processing the solution of (\ref{eqn:EVI}).

\begin{figure}[htbp!]
\begin{minipage}{0.075\linewidth}
\includegraphics[ trim = 0mm 0mm 0mm 0mm,  clip,  width=\textwidth
]{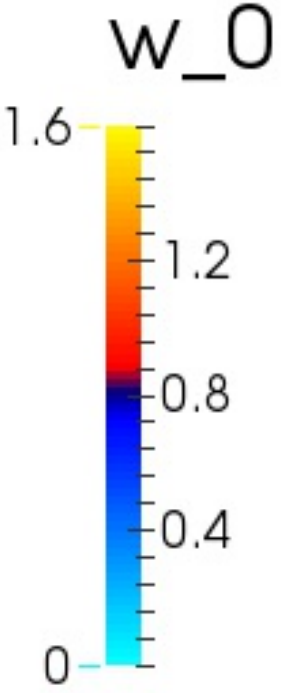}
\end{minipage}
\hskip .2em
\centering
\begin{minipage}{0.9\linewidth}
\centering
\includegraphics[ trim = 10mm 110mm 10mm 30mm,  clip,  width=.24\textwidth
]{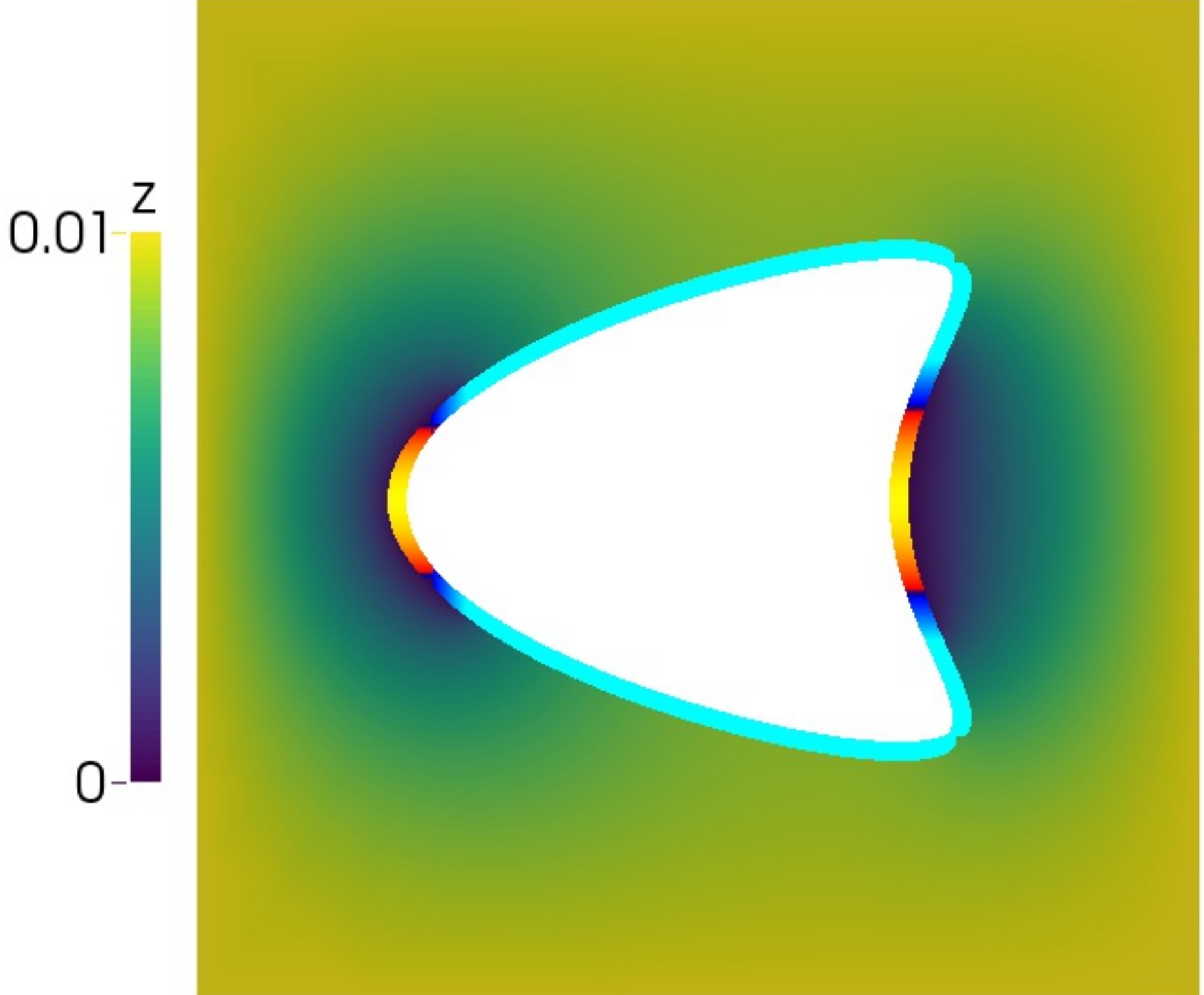}
\includegraphics[ trim = 10mm 110mm 10mm 30mm,  clip,  width=.24\textwidth
]{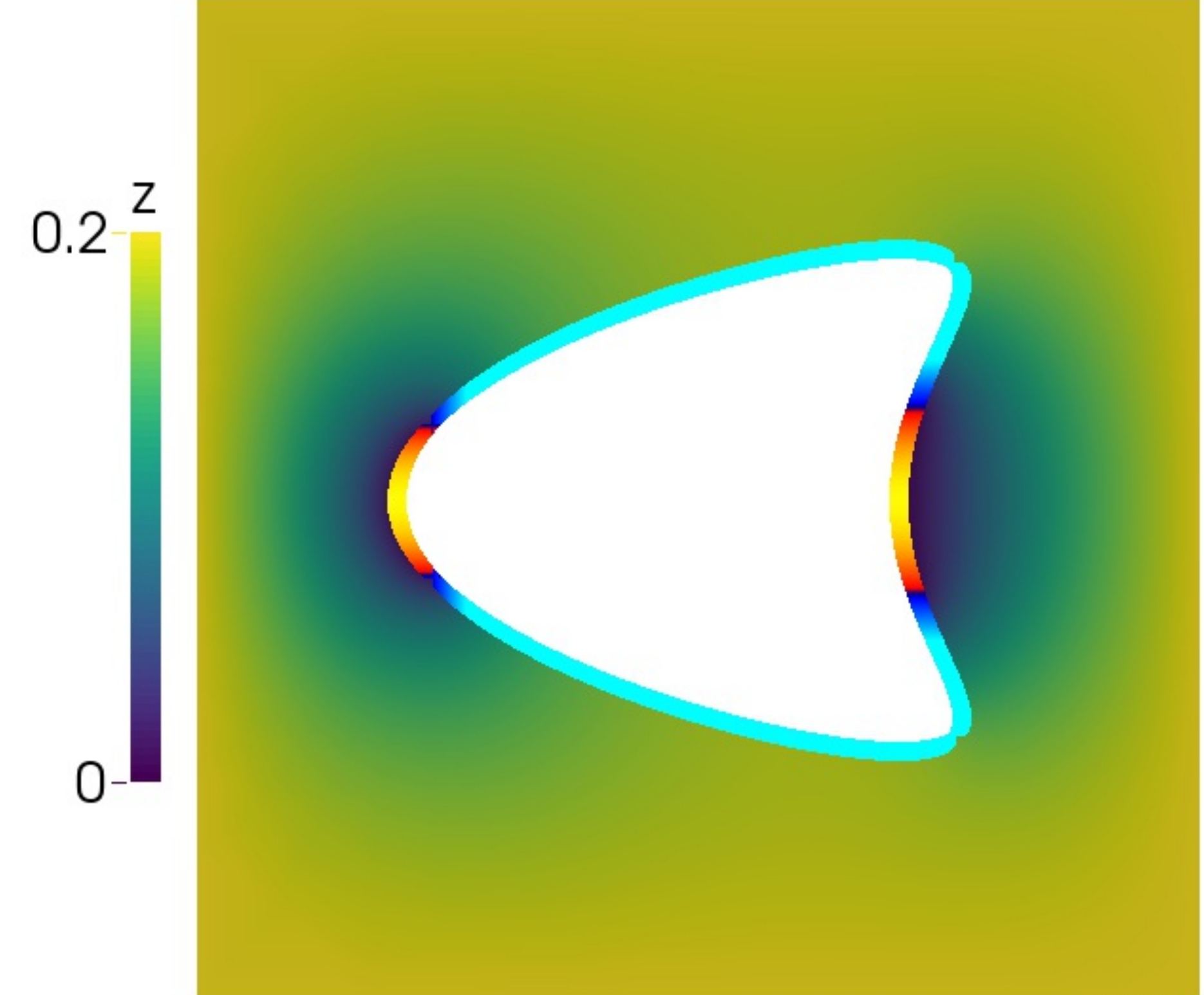}
\includegraphics[ trim = 10mm 110mm 10mm 30mm,  clip,  width=.24\textwidth
]{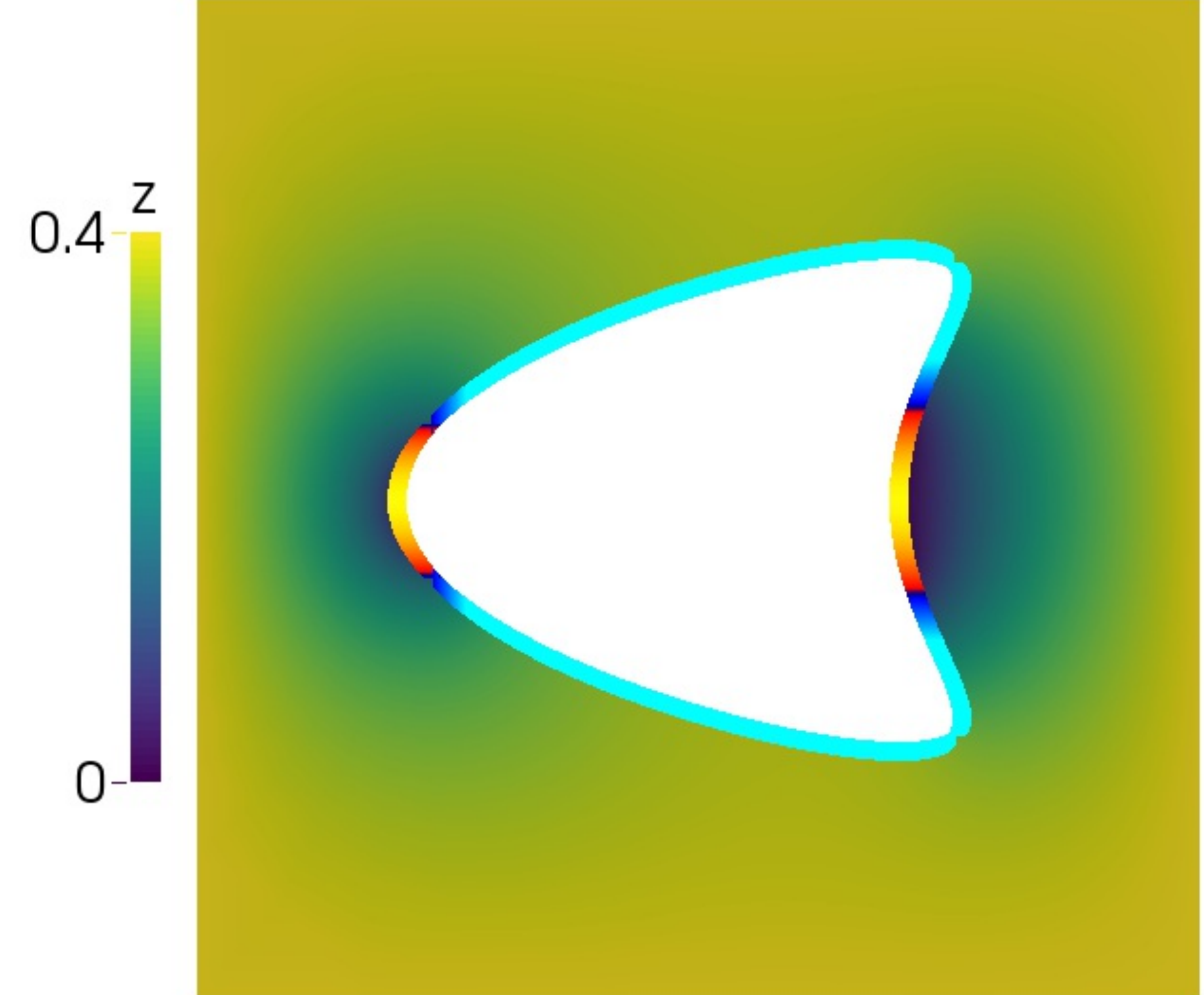}
\includegraphics[ trim = 10mm 110mm 10mm 30mm,  clip,  width=.24\textwidth
]{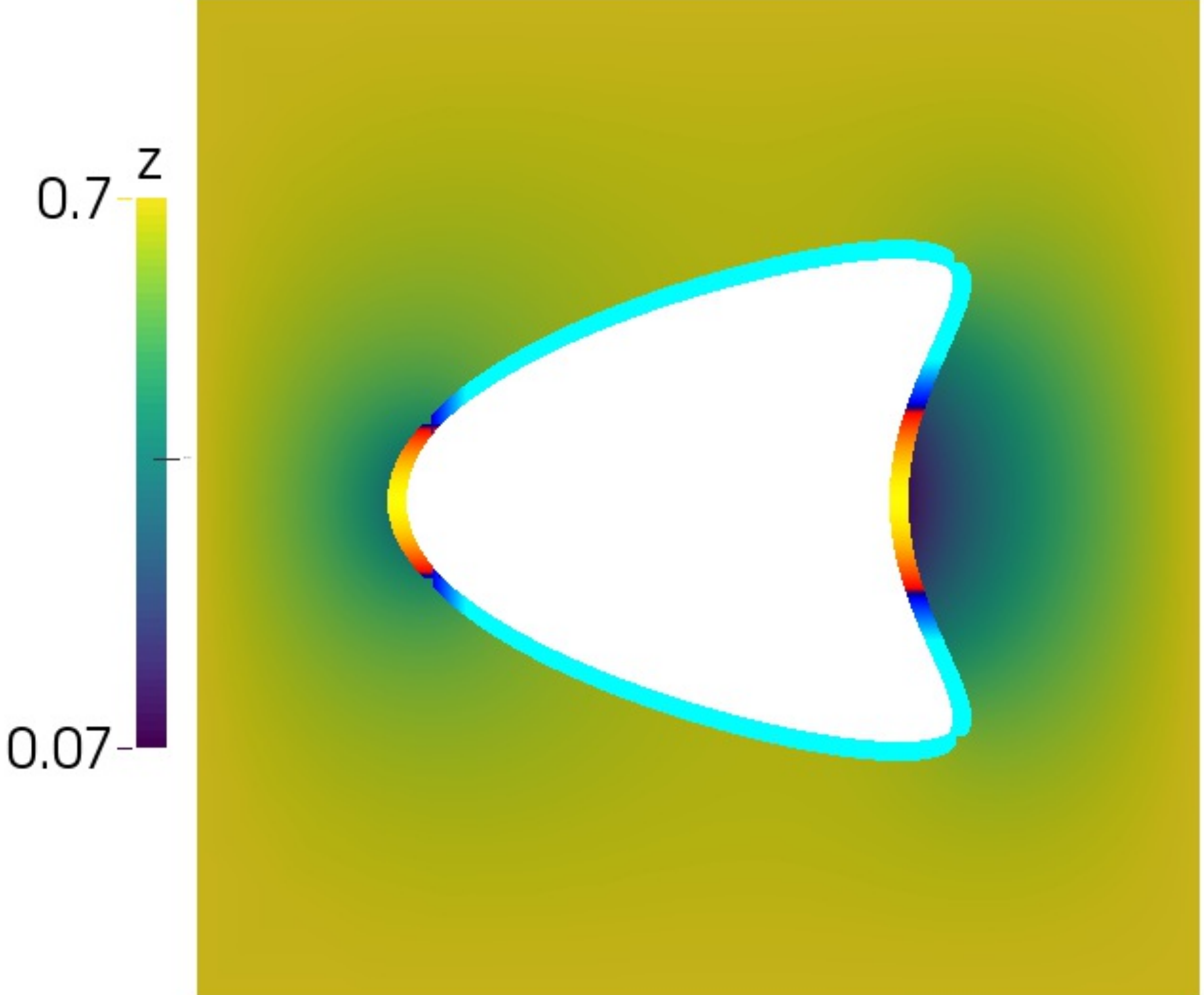}
\end{minipage}
\caption{Simulation results of \S \ref{sec:sim-2d}. Snapshots of the computed solution $Z$ together with the initial data  $W^0$ of the elliptic variational inequality (\ref{eqn:EVI}) at times $0.01, 0.2, 0.4$ and $0.7$ reading from left to right. The colour scale for $W^0$ is fixed in every figure.}
\label{fig:2d-res-z}
\end{figure}

\begin{figure}[htbp!]
\centering
\begin{minipage}{0.1\linewidth}
\includegraphics[ trim = 0mm 0mm 0mm 0mm,  clip,  width=\textwidth
]{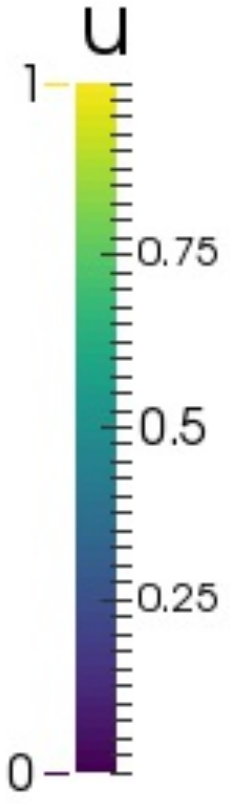}
\end{minipage}
\hskip .5em
\begin{minipage}{0.75\linewidth}
\centering
\subfigure[][{$\eps=10^{-1}$}]{
\includegraphics[ trim = 10mm 110mm 10mm 30mm,  clip,  width=.24\textwidth
]{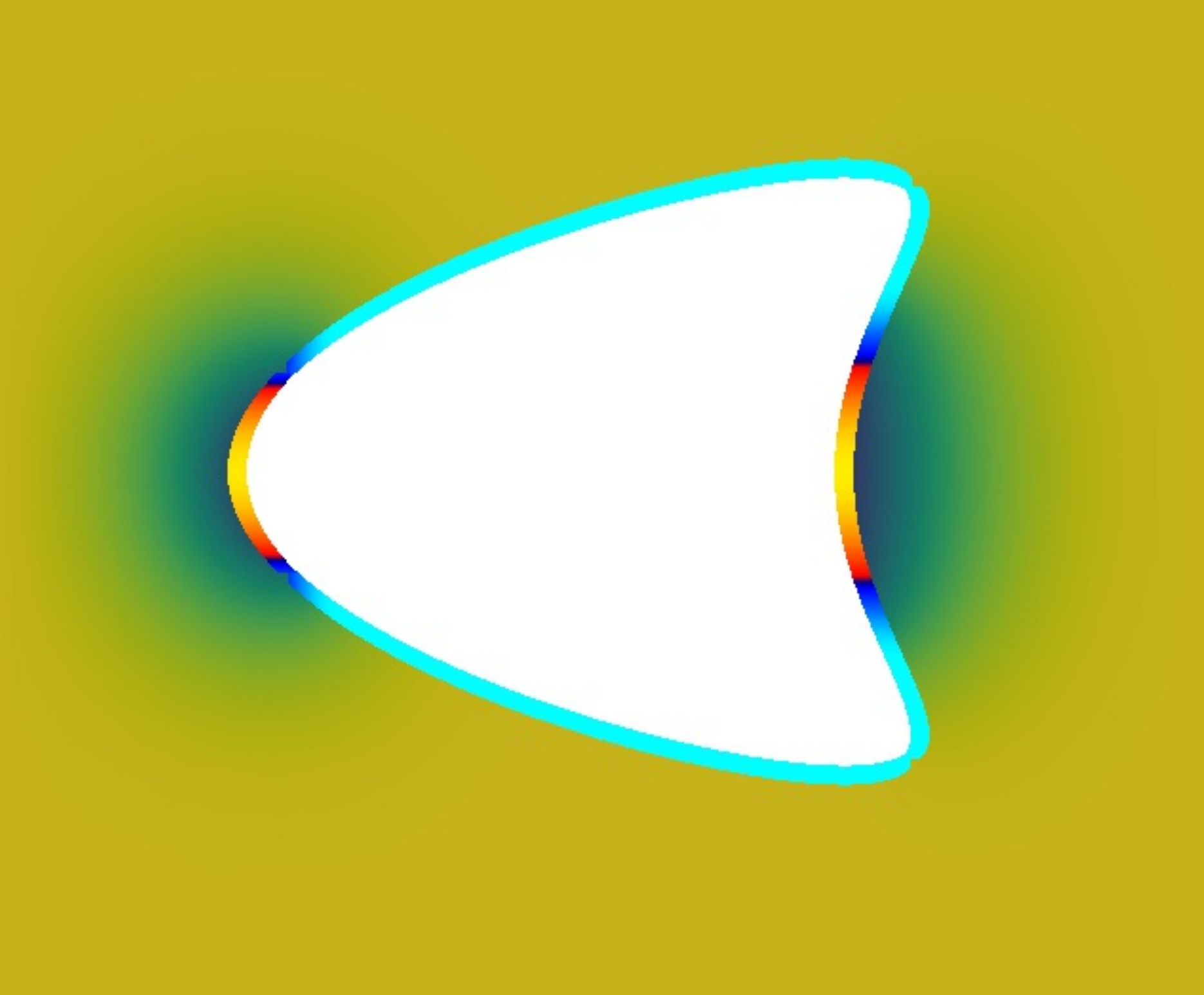}
\includegraphics[ trim = 10mm 110mm 10mm 30mm,  clip,  width=.24\textwidth
]{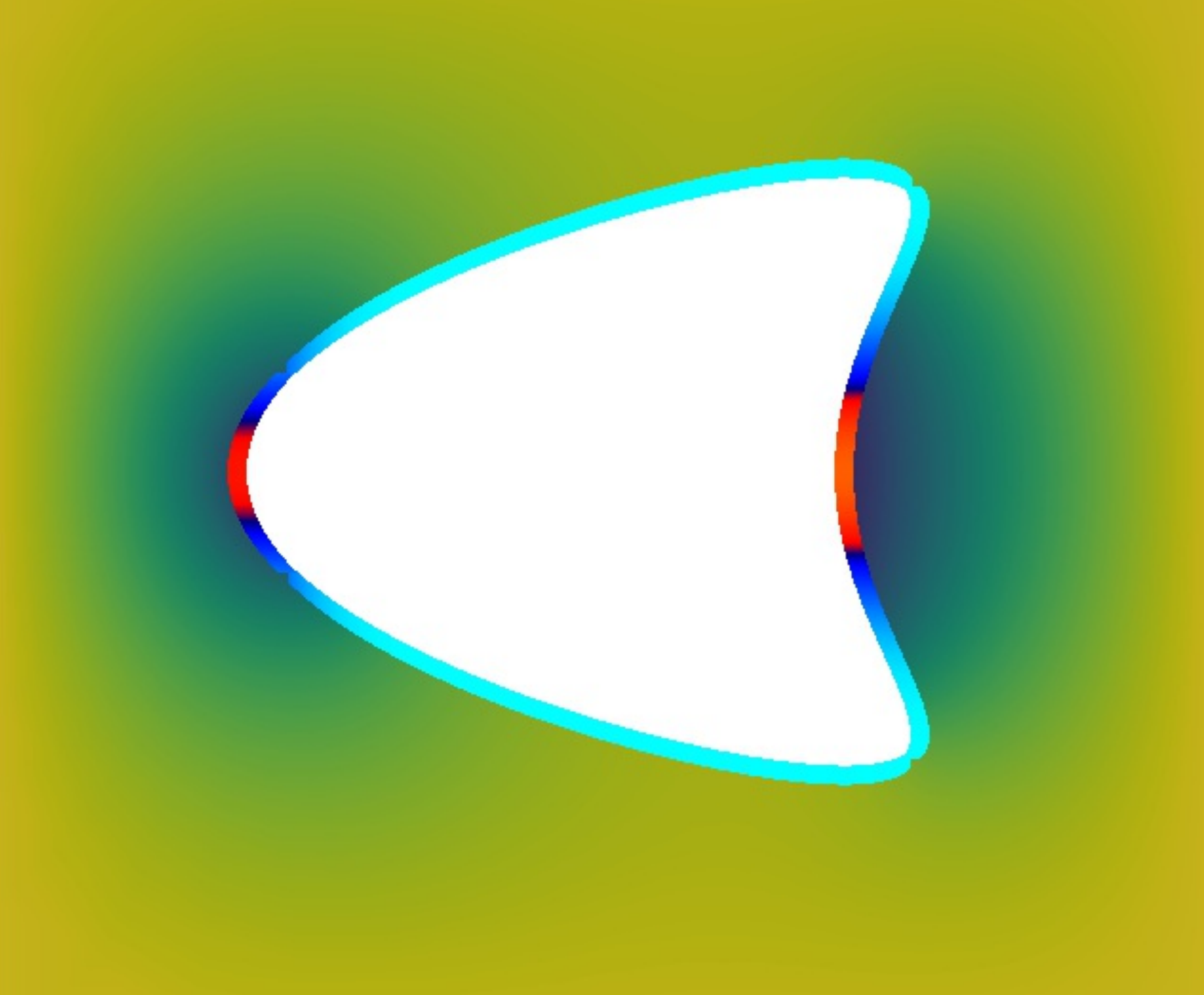}
\includegraphics[ trim = 10mm 110mm 10mm 30mm,  clip,  width=.24\textwidth
]{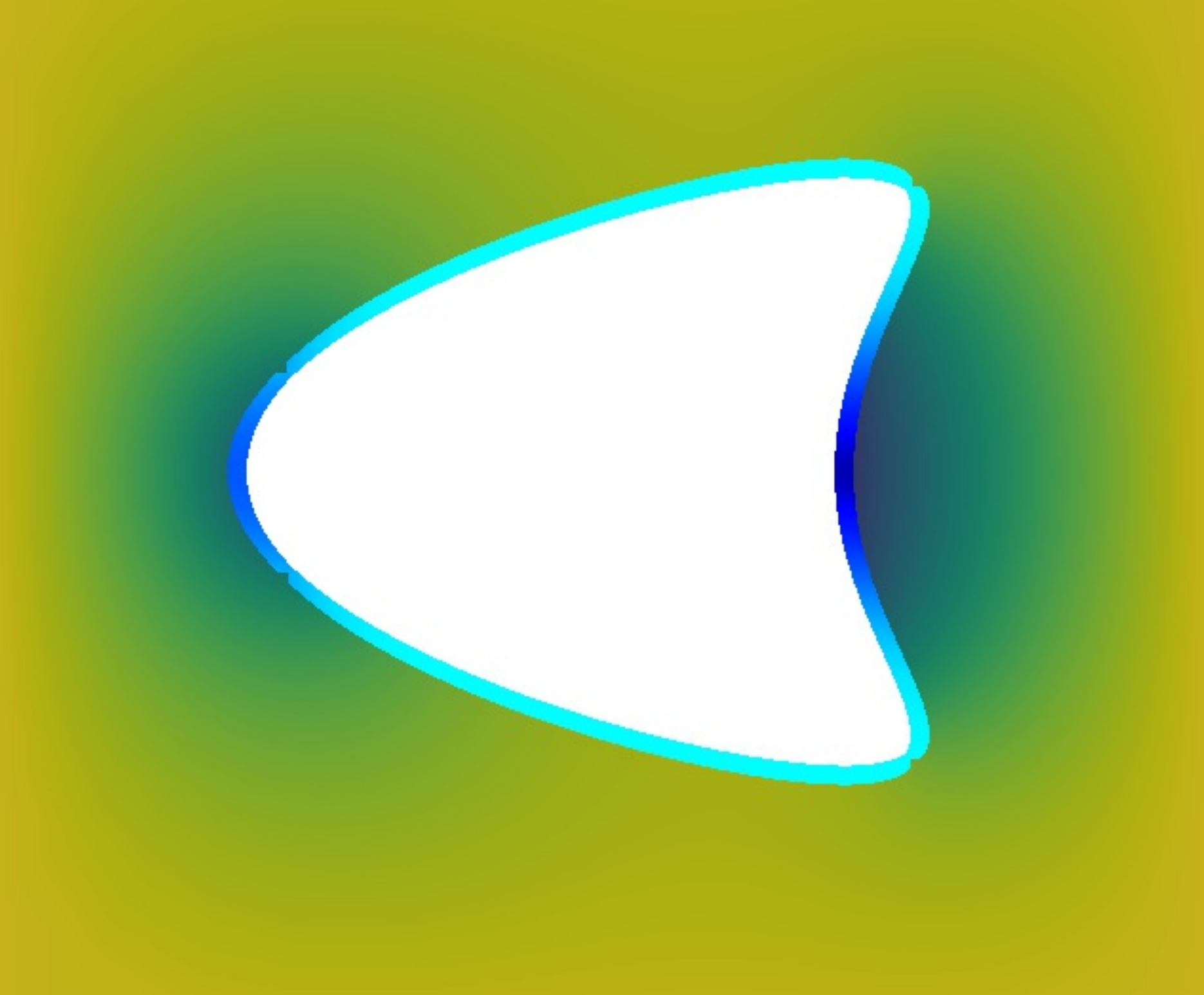}
\includegraphics[ trim = 10mm 110mm 10mm 30mm,  clip,  width=.24\textwidth
]{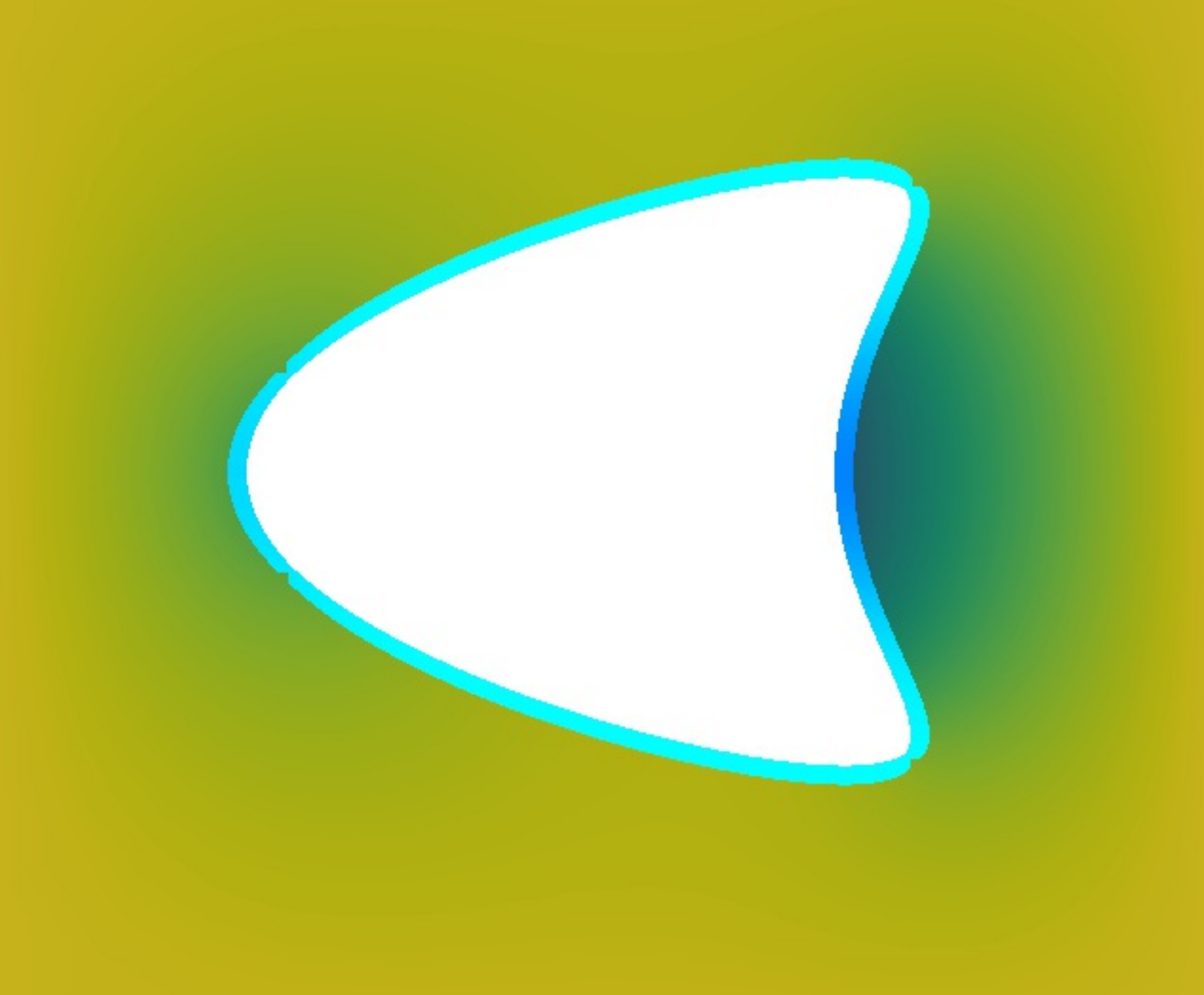}
}
\subfigure[][{$\eps=10^{-2}$}]{
\includegraphics[ trim = 10mm 110mm 10mm 30mm,  clip,  width=.24\textwidth
]{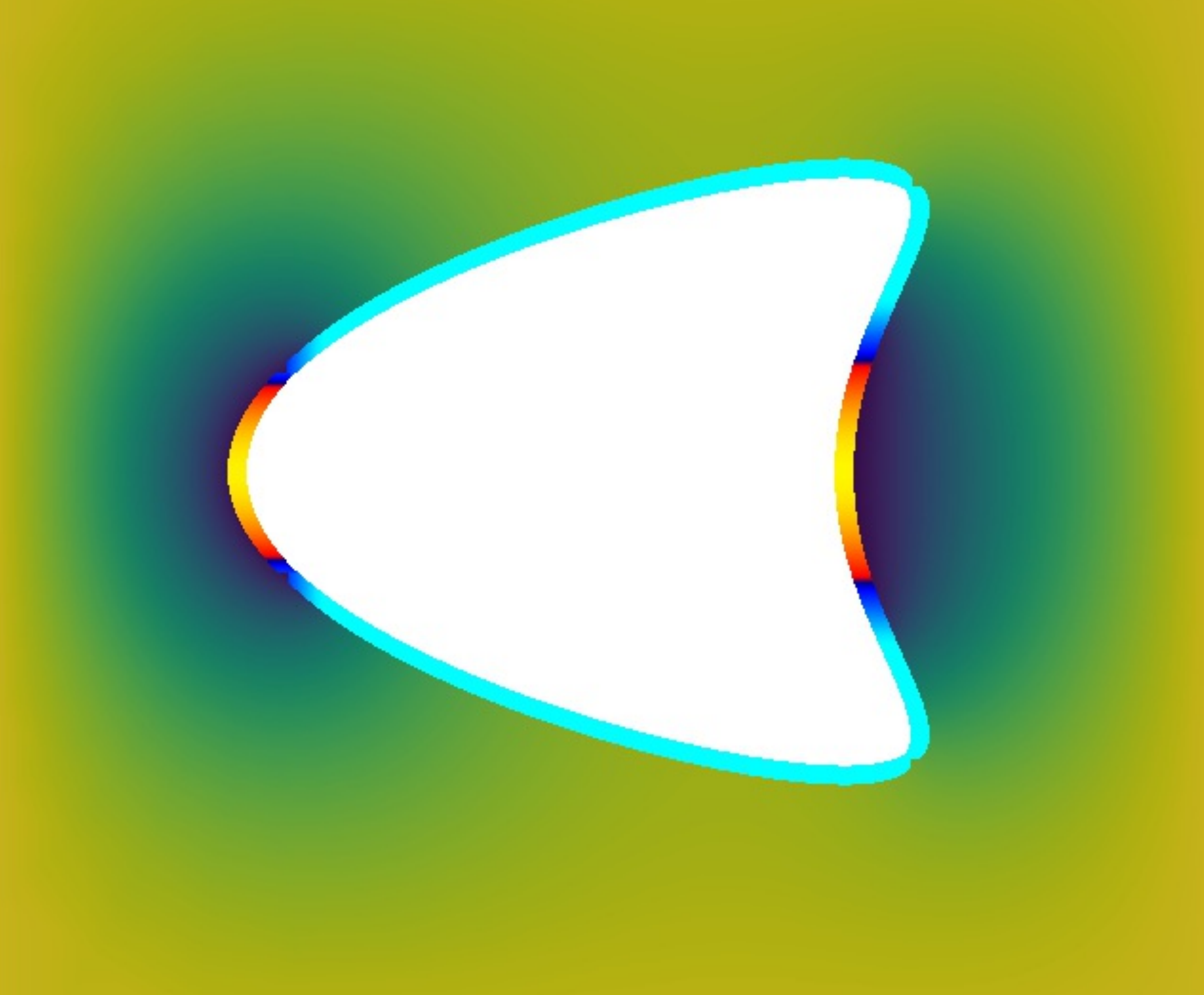}
\includegraphics[ trim = 10mm 110mm 10mm 30mm,  clip,  width=.24\textwidth
]{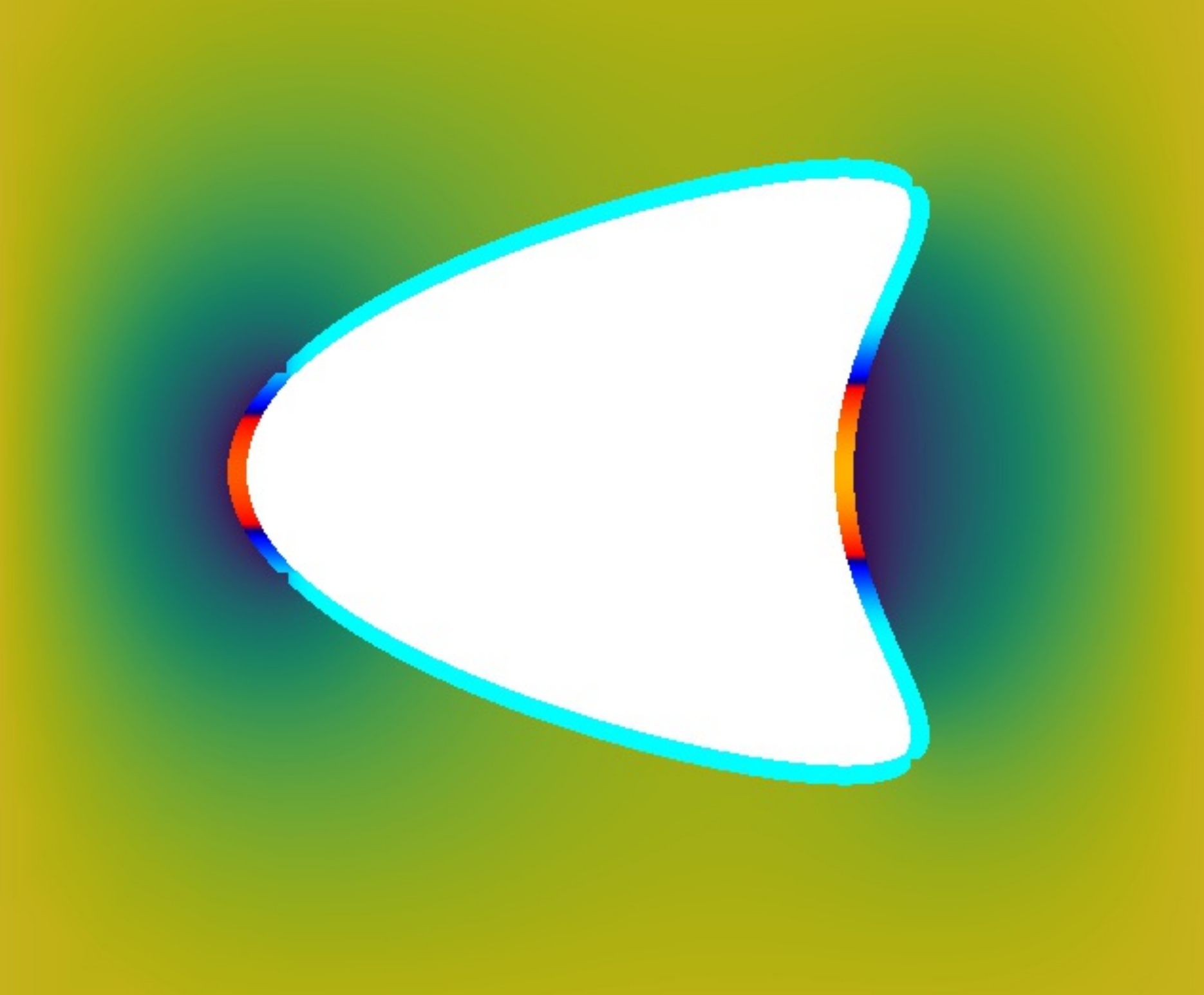}
\includegraphics[ trim = 10mm 110mm 10mm 30mm,  clip,  width=.24\textwidth
]{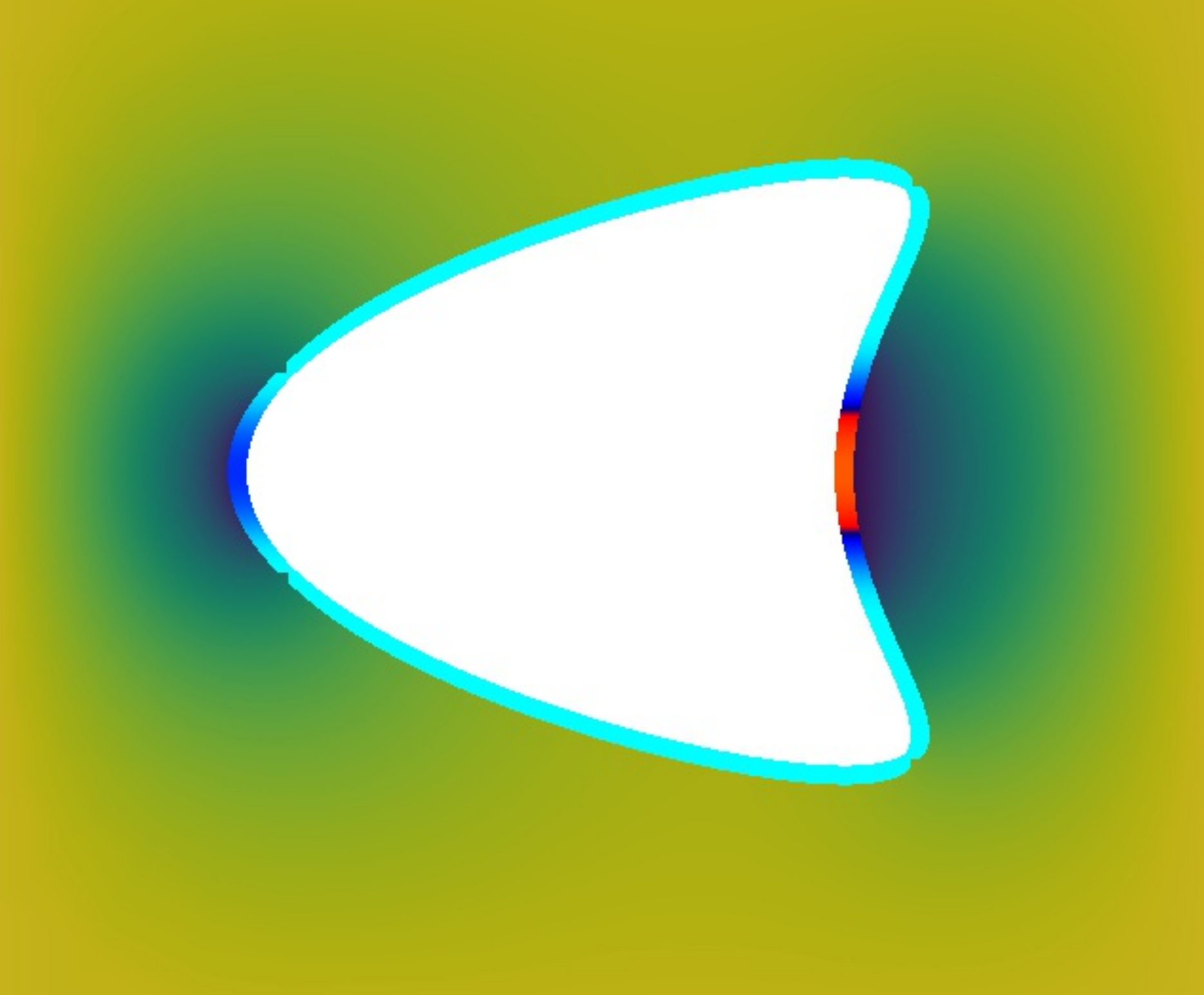}
\includegraphics[ trim = 10mm 110mm 10mm 30mm,  clip,  width=.24\textwidth
]{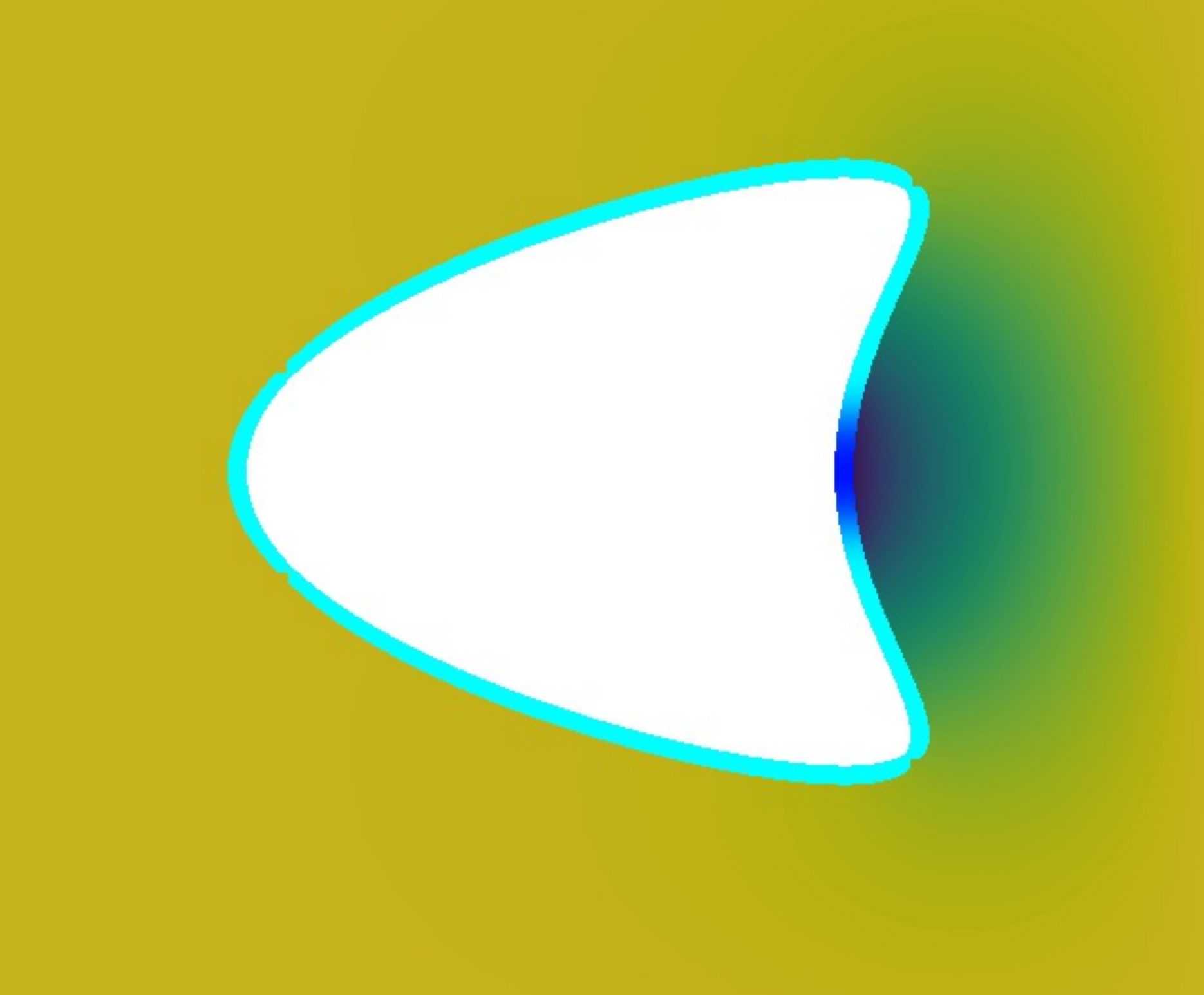}
}
\subfigure[][{$\eps=10^{-3}$}]{
\includegraphics[ trim = 10mm 110mm 10mm 30mm,  clip,  width=.24\textwidth
]{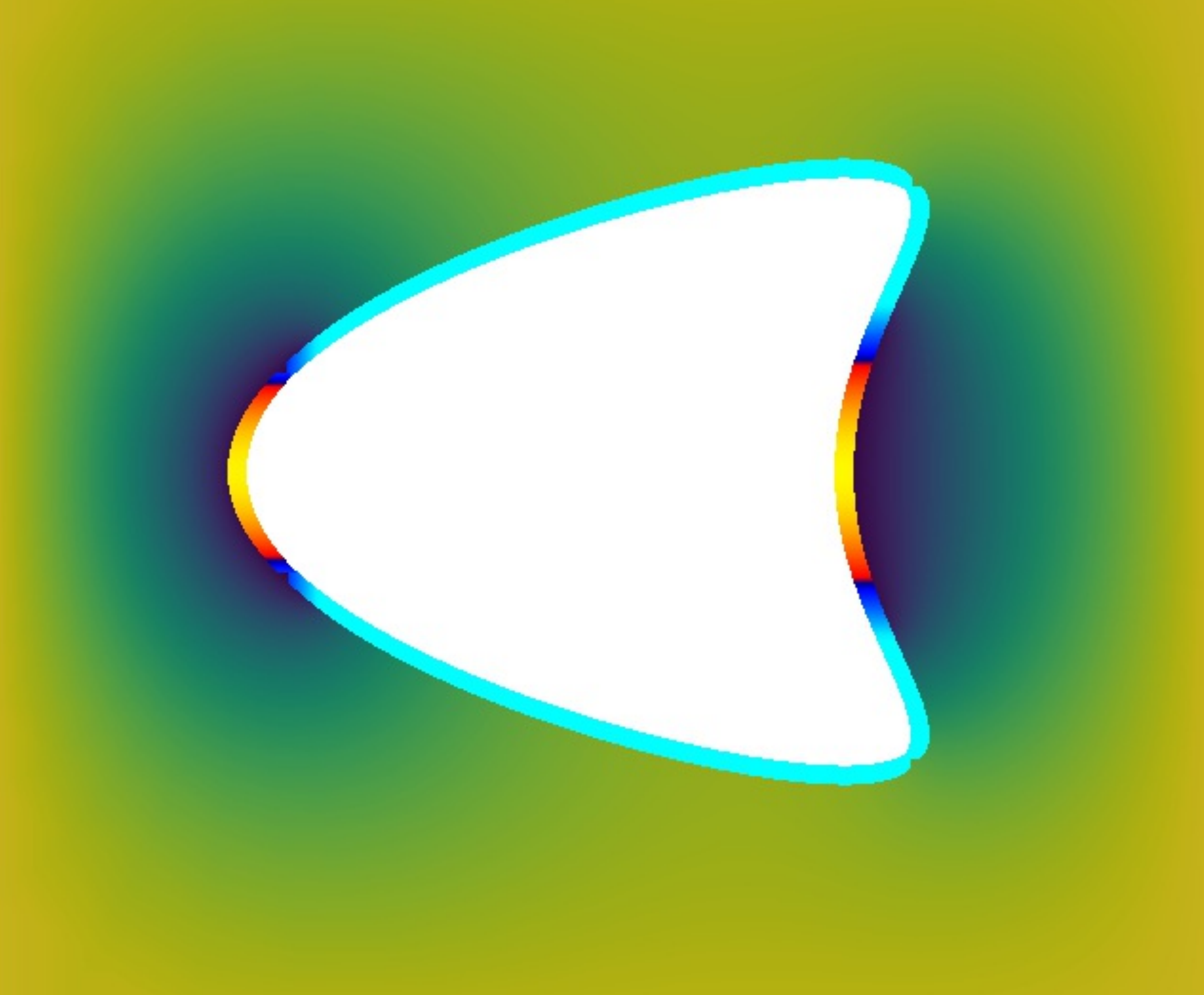}
\includegraphics[ trim = 10mm 110mm 10mm 30mm,  clip,  width=.24\textwidth
]{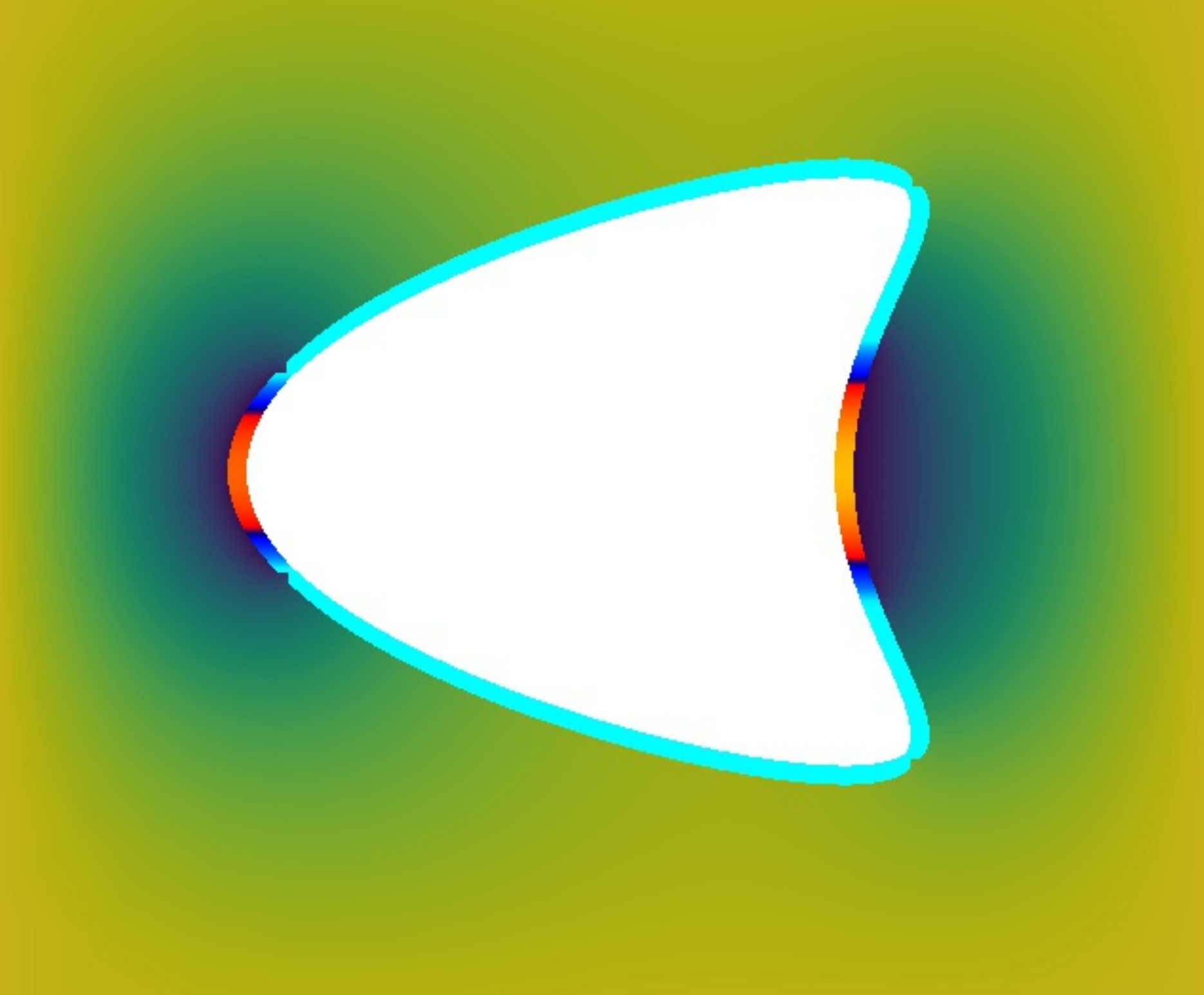}
\includegraphics[ trim = 10mm 110mm 10mm 30mm,  clip,  width=.24\textwidth
]{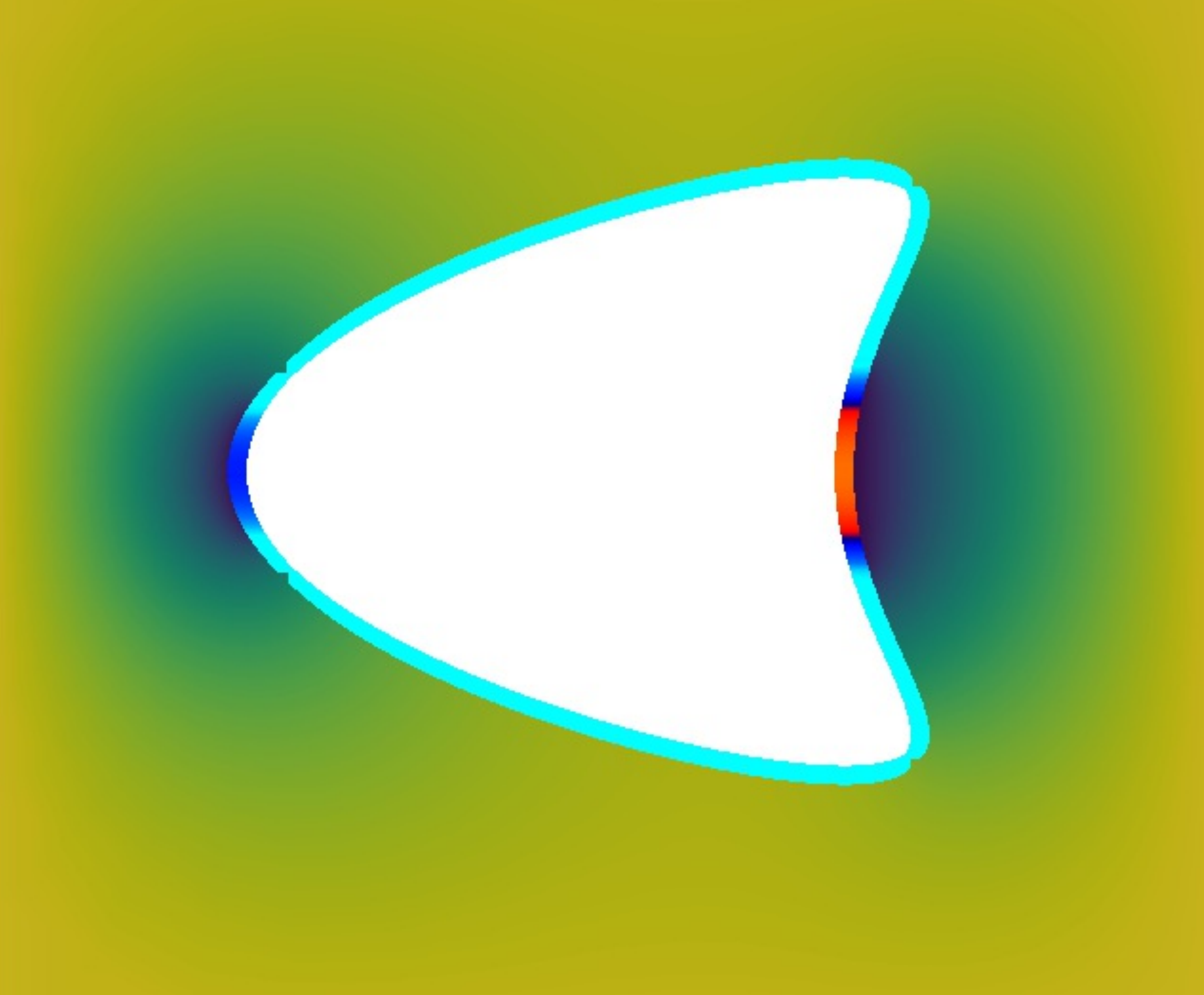}
\includegraphics[ trim = 10mm 110mm 10mm 30mm,  clip,  width=.24\textwidth
]{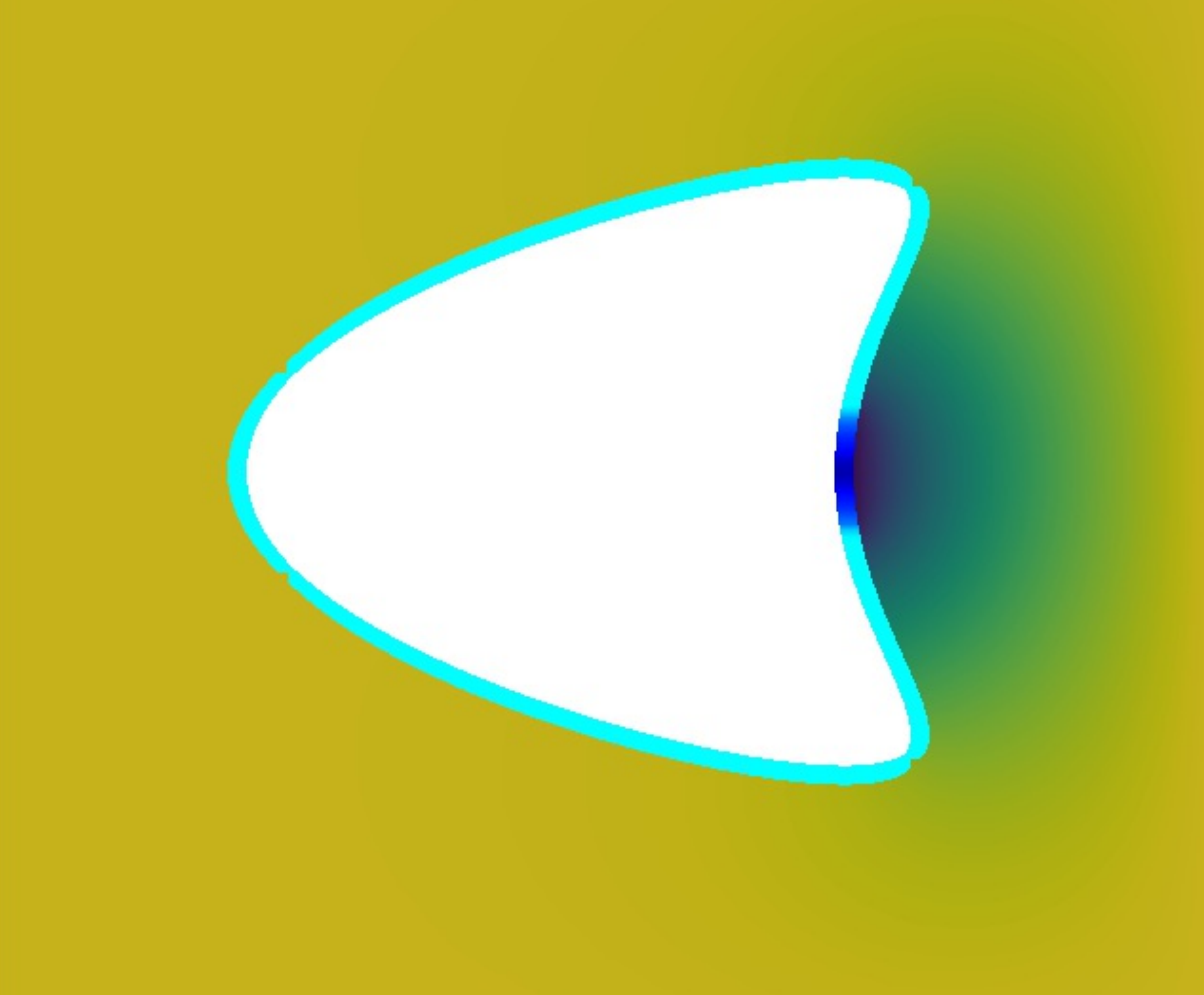}
}
\subfigure[][{$u^n=(z^{t_n}-z^{t_n-0.01})/0.01, \quad w^n=w^0+\nabla z^n\cdot\normal$}]{
\includegraphics[ trim = 10mm 110mm 10mm 30mm,  clip,  width=.24\textwidth
]{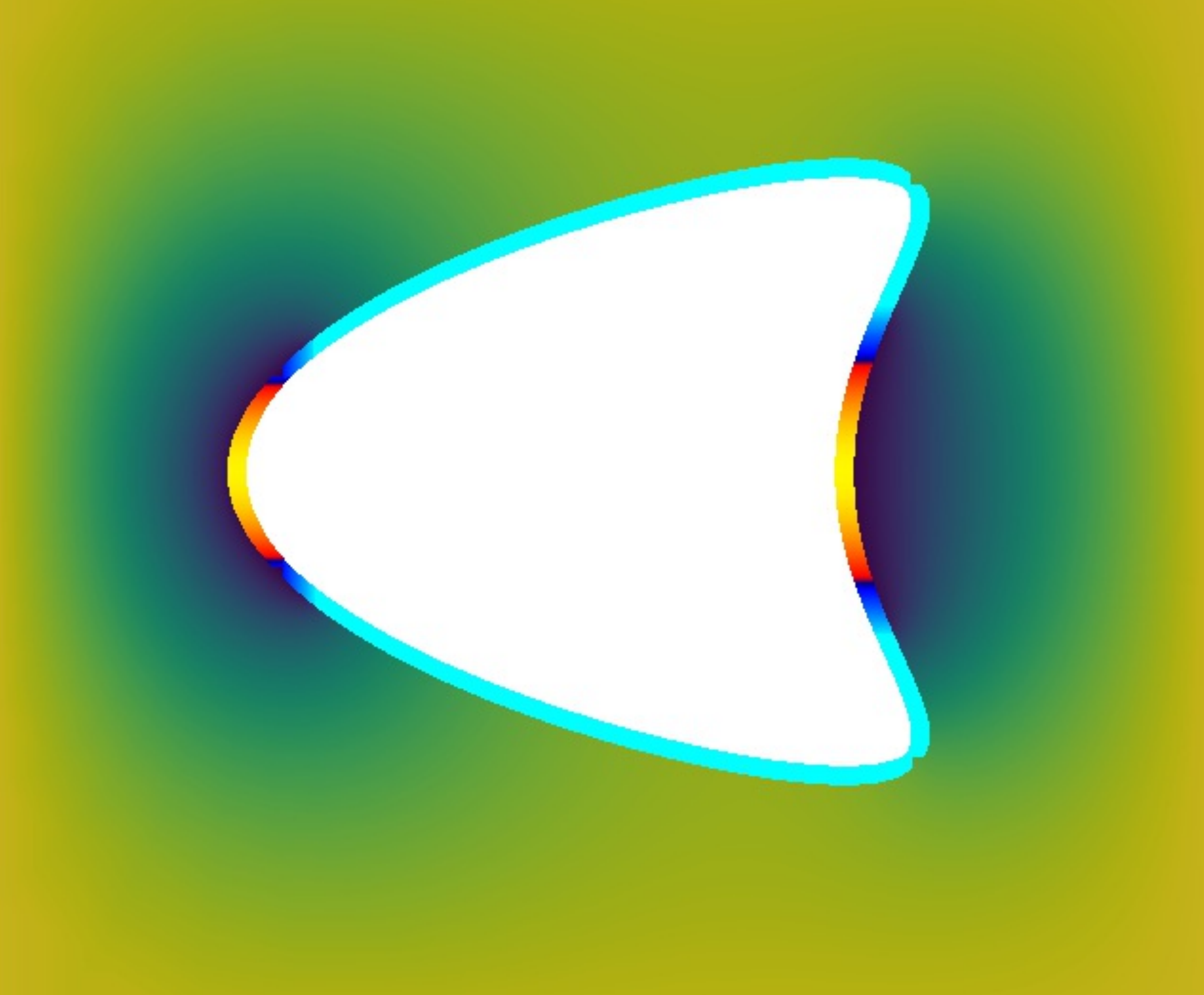}
\includegraphics[ trim = 10mm 110mm 10mm 30mm,  clip,  width=.24\textwidth
]{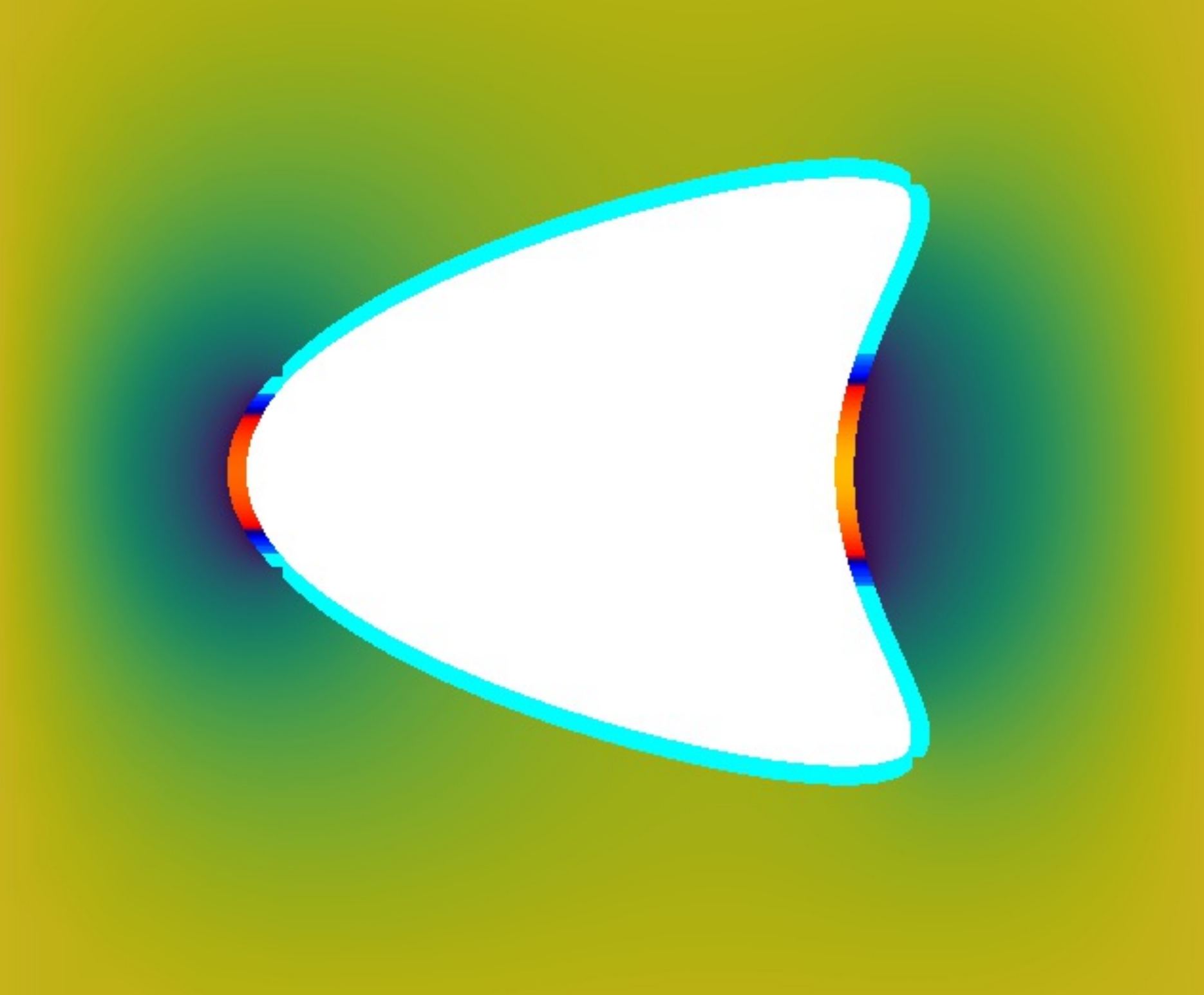}
\includegraphics[ trim = 10mm 110mm 10mm 30mm,  clip,  width=.24\textwidth
]{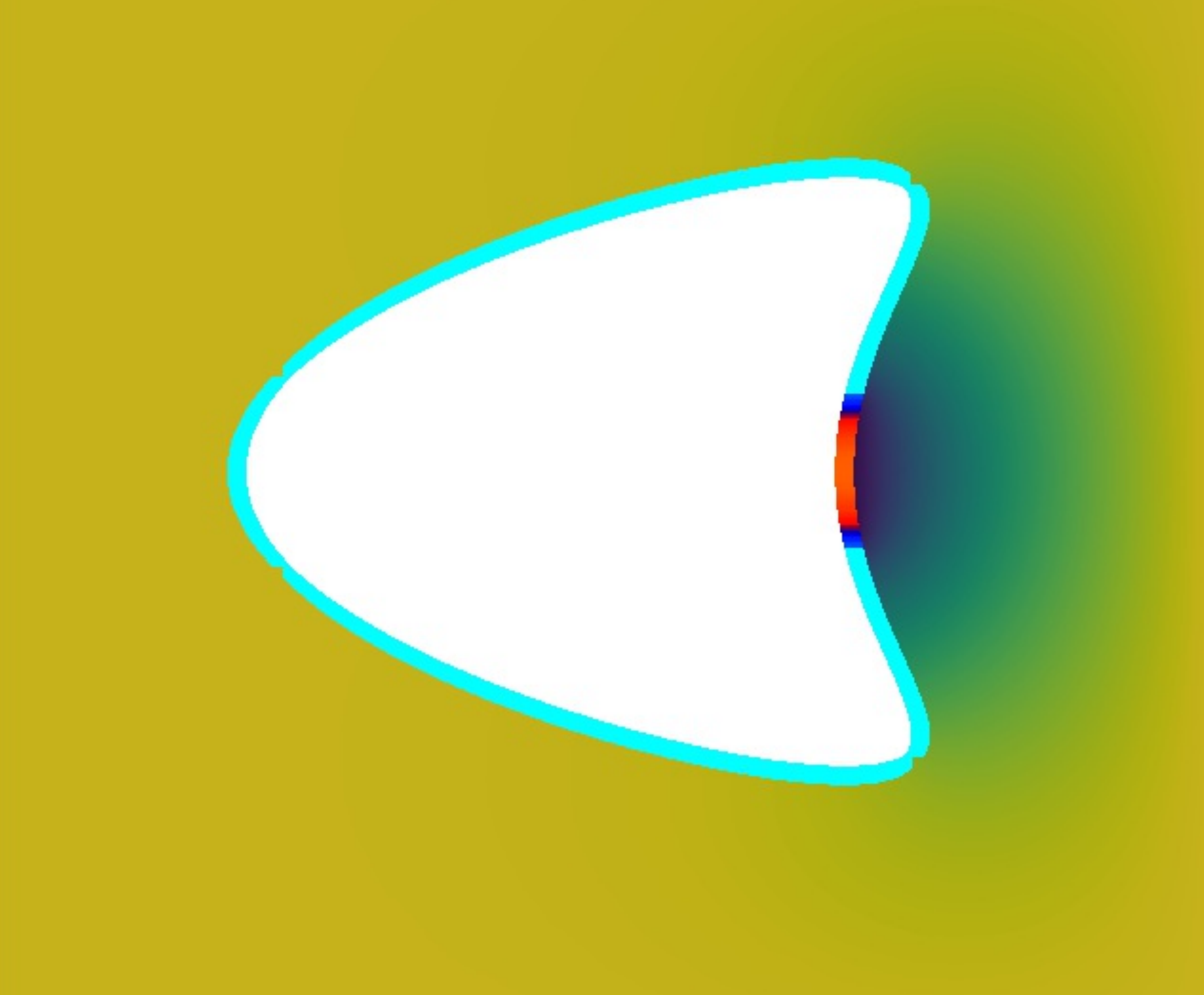}
\includegraphics[ trim = 10mm 110mm 10mm 30mm,  clip,  width=.24\textwidth
]{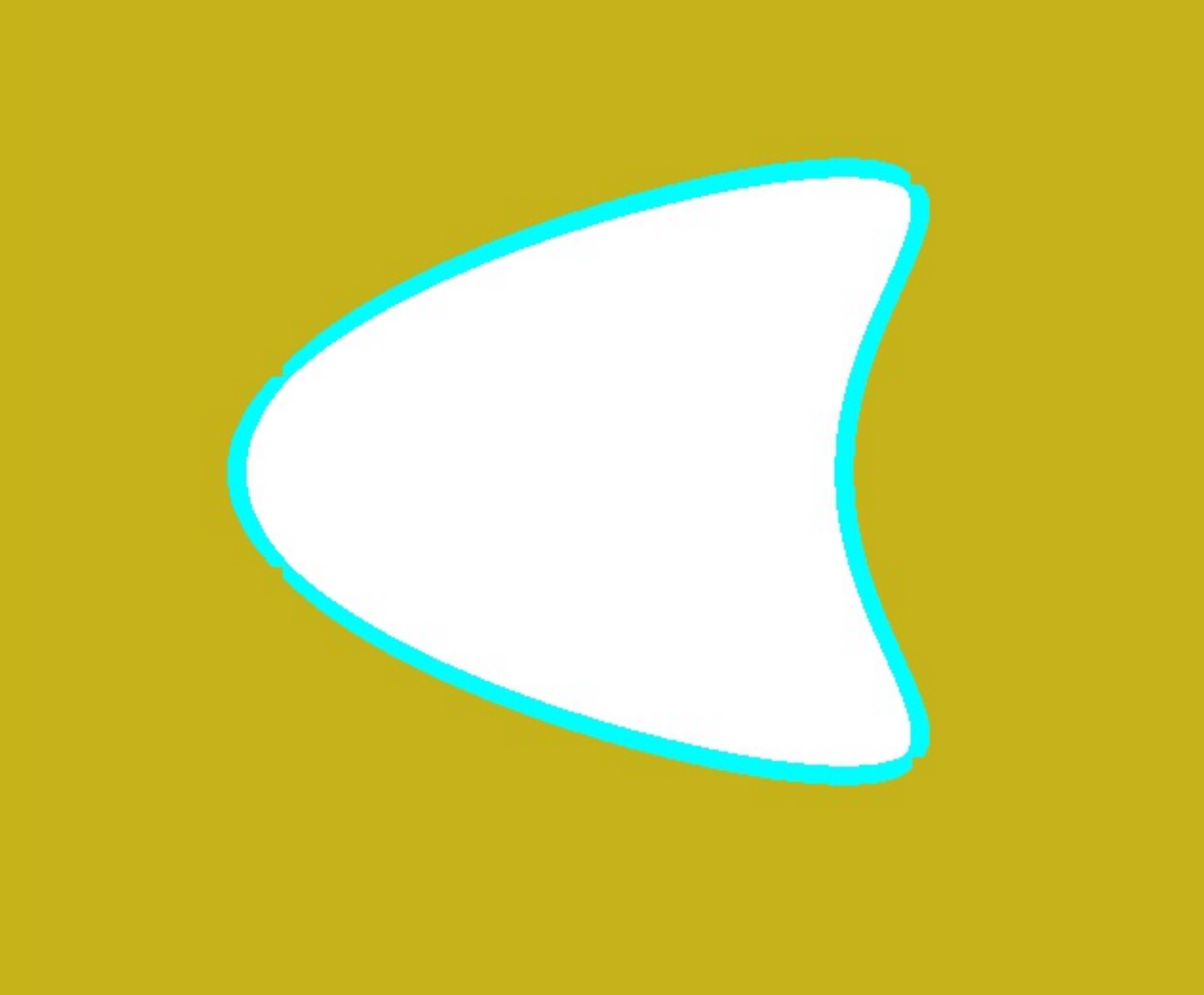}
}
\end{minipage}
\hskip .5em
\begin{minipage}{0.1\linewidth}
\includegraphics[ trim = 0mm 0mm 0mm 0mm,  clip,  width=\textwidth
]{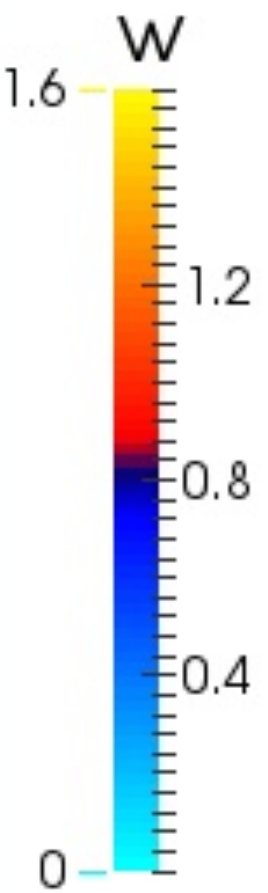}
\end{minipage}
\caption{Simulation results of \S \ref{sec:sim-2d}. (First three rows) Snapshots of the computed solutions $U$ and $W$ of (\ref{eqn:wf_eps_problem}) in 2D at times $0.01, 0.2, 0.4$ and $0.7$ (reading from left to right) for different values of $\eps=\delO=\delG$. The fourth row shows the computed solutions $U$ and $W=-V$ post-processed from solving the elliptic variational inequality (\ref{eqn:EVI}) at times $0.01, 0.2, 0.4$ and $0.7$ reading from left to right.}\label{fig:2d-res-uv}
\end{figure}

\begin{figure}[htbp!]
\centering
\subfigure[][{$\eps=10^{-1}$}]{
\includegraphics[ trim = 5mm 140mm 20mm 10mm,  clip,  width=.24\textwidth
]{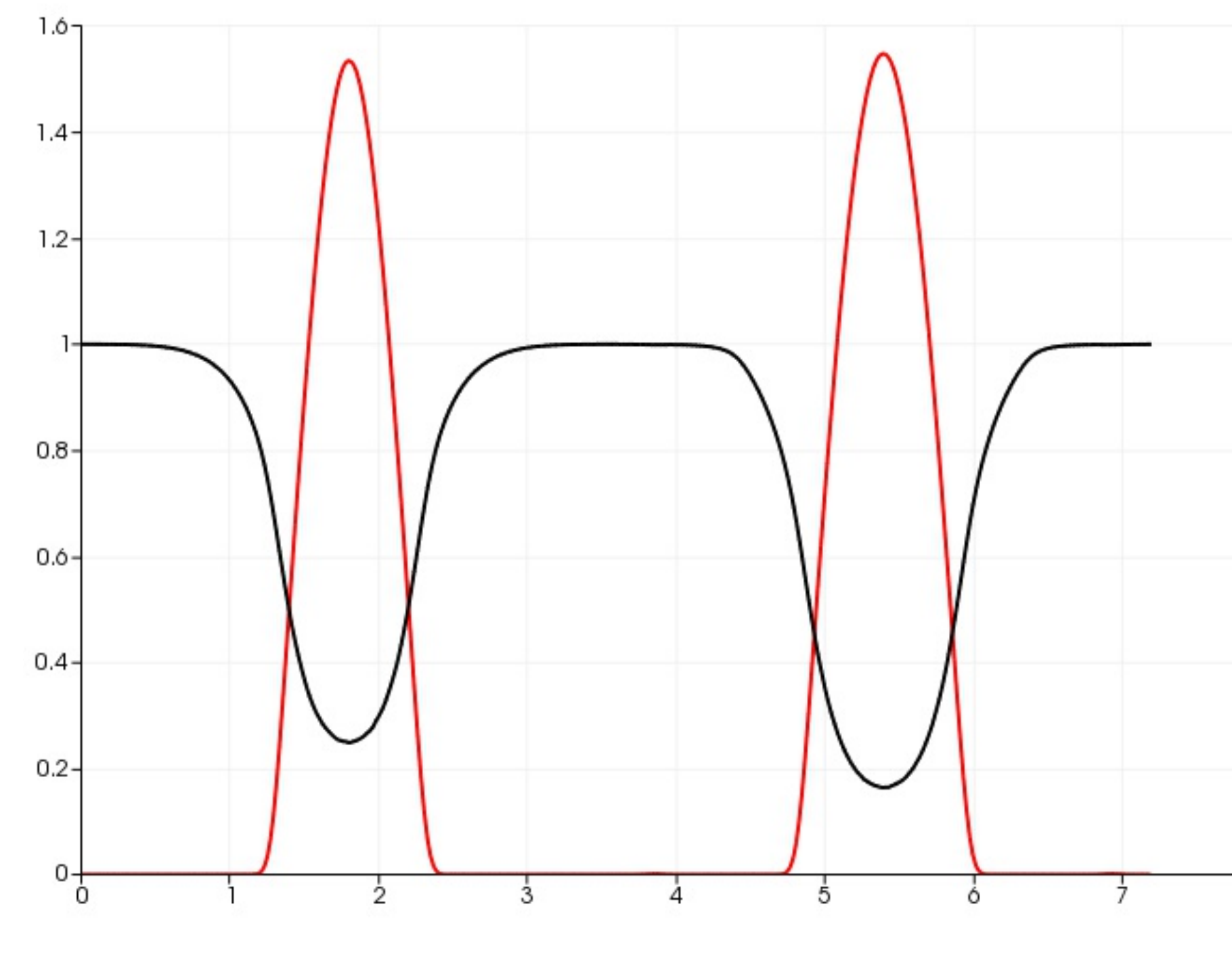}
\includegraphics[ trim = 5mm 140mm 20mm 10mm,  clip,  width=.24\textwidth
]{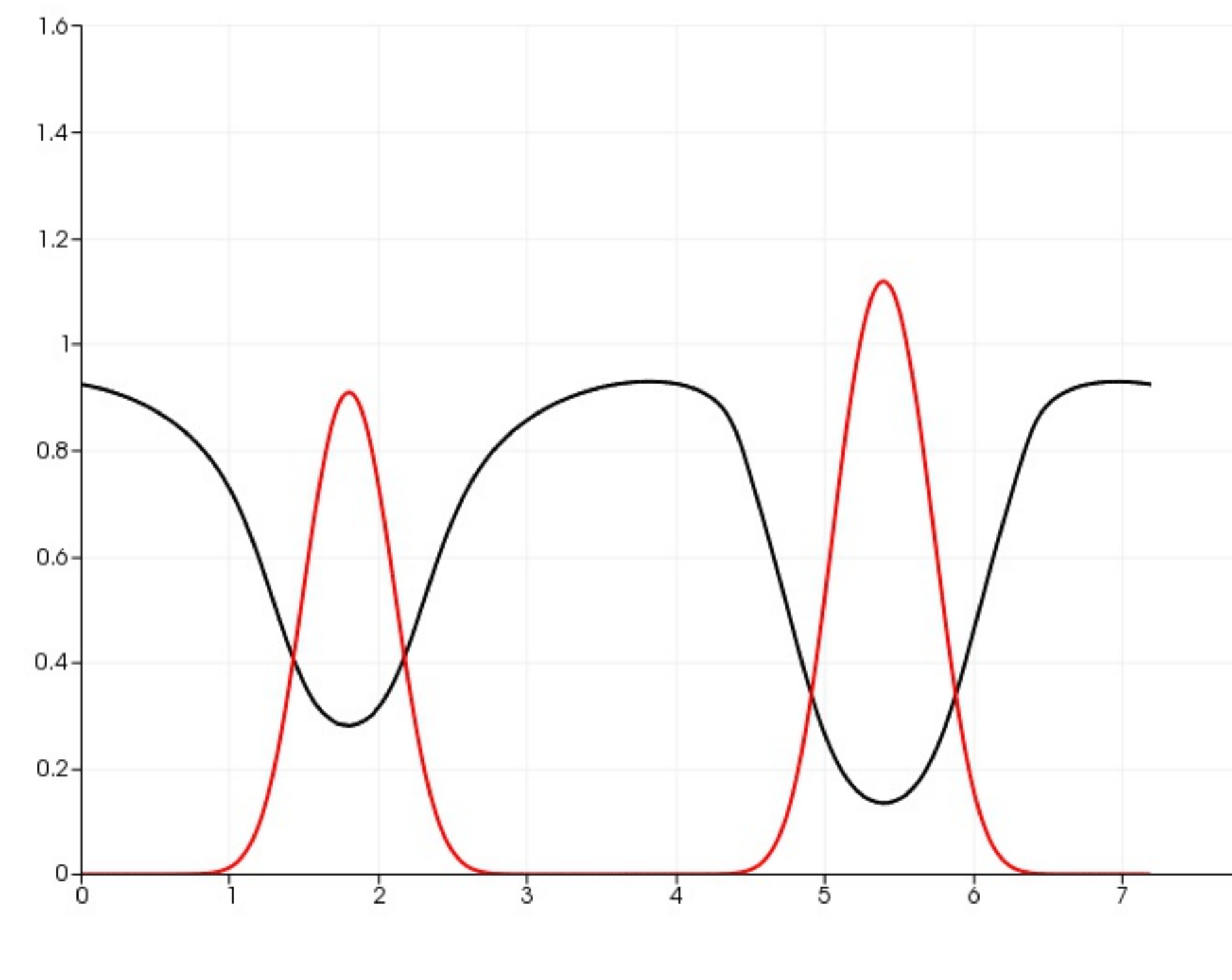}
\includegraphics[ trim = 5mm 140mm 20mm 10mm,  clip,  width=.24\textwidth
]{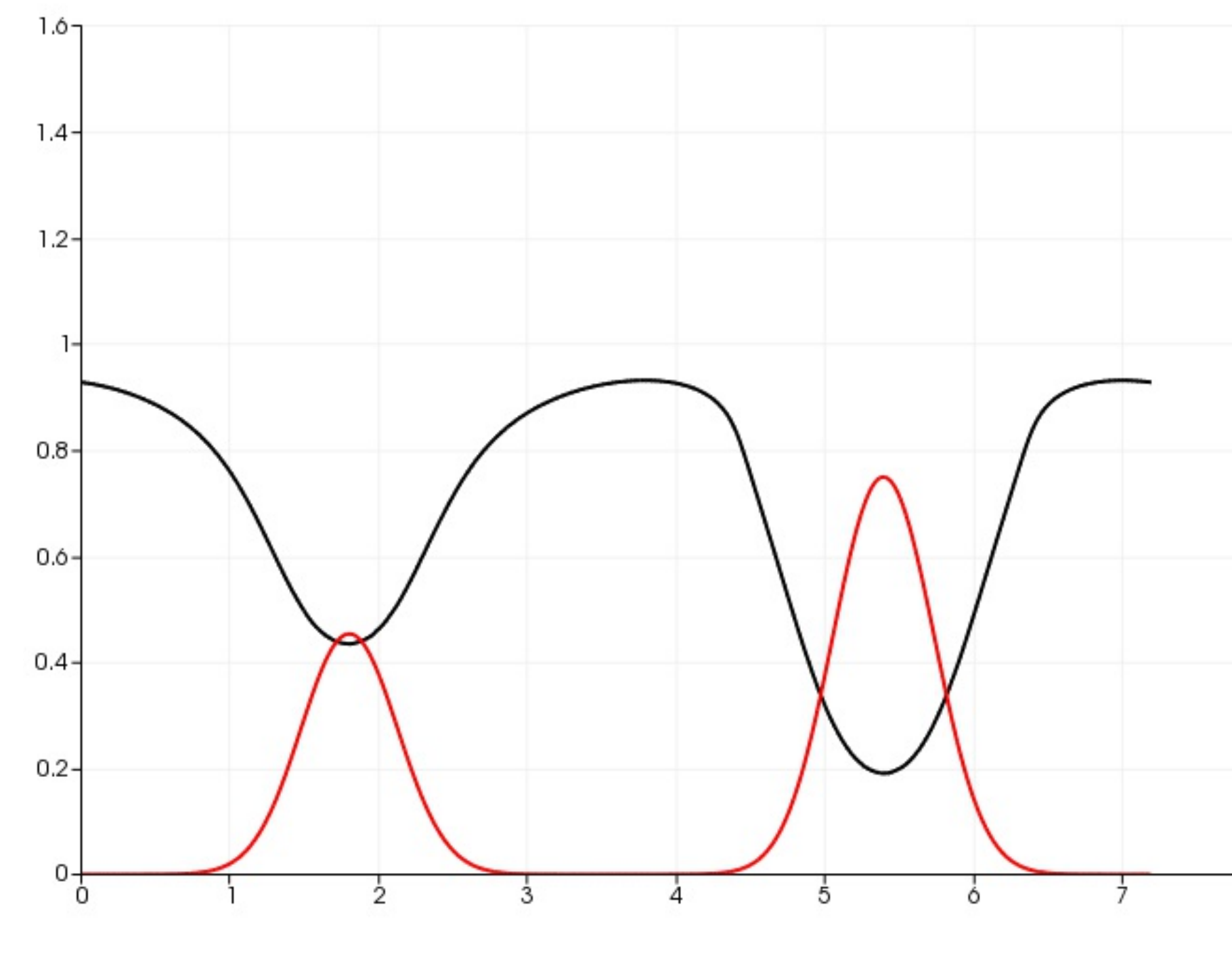}
\includegraphics[ trim = 5mm 140mm 20mm 10mm,  clip,  width=.24\textwidth
]{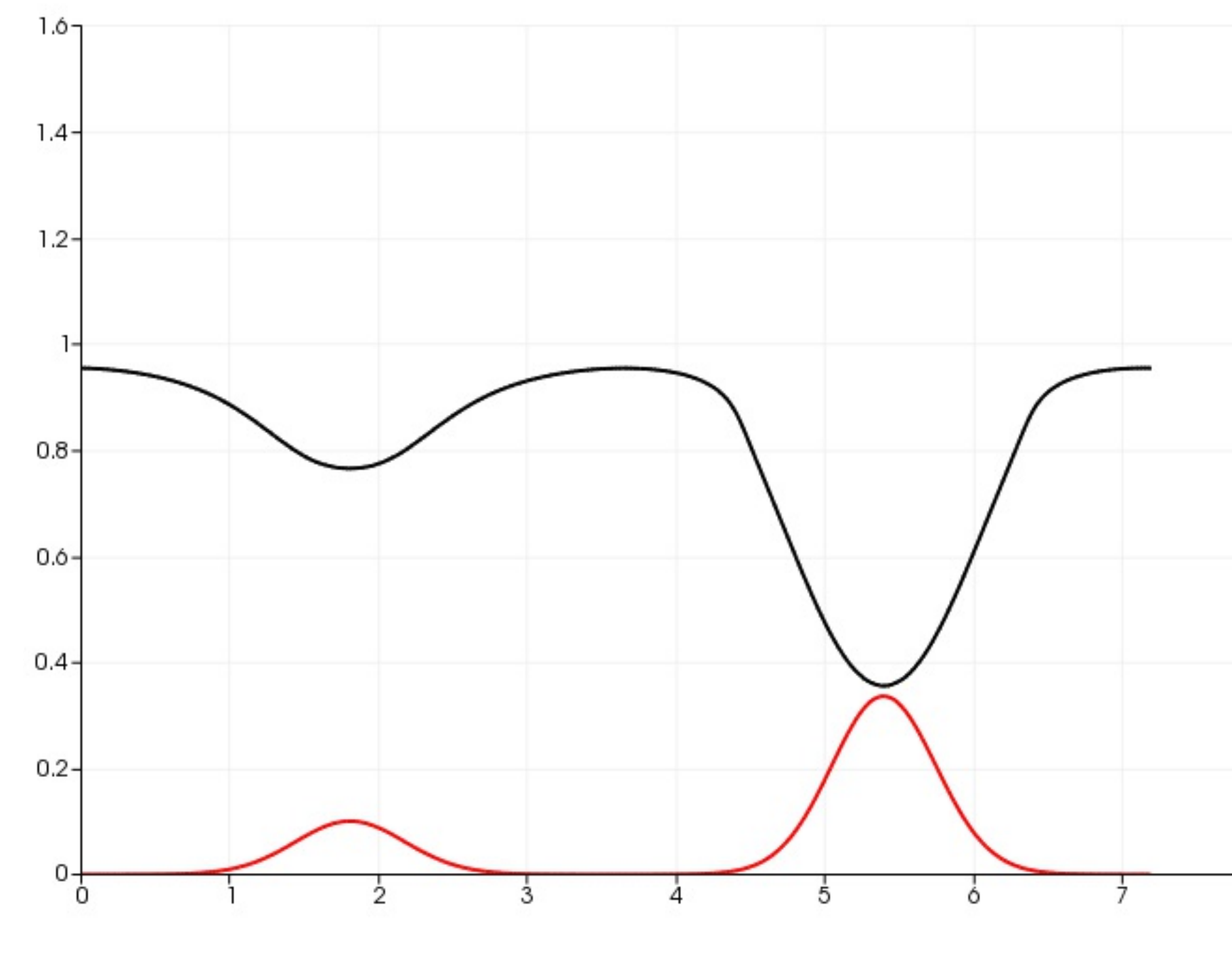}
}
\subfigure[][{$\eps=10^{-2}$}]{
\includegraphics[ trim = 5mm 140mm 20mm 10mm,  clip,  width=.24\textwidth
]{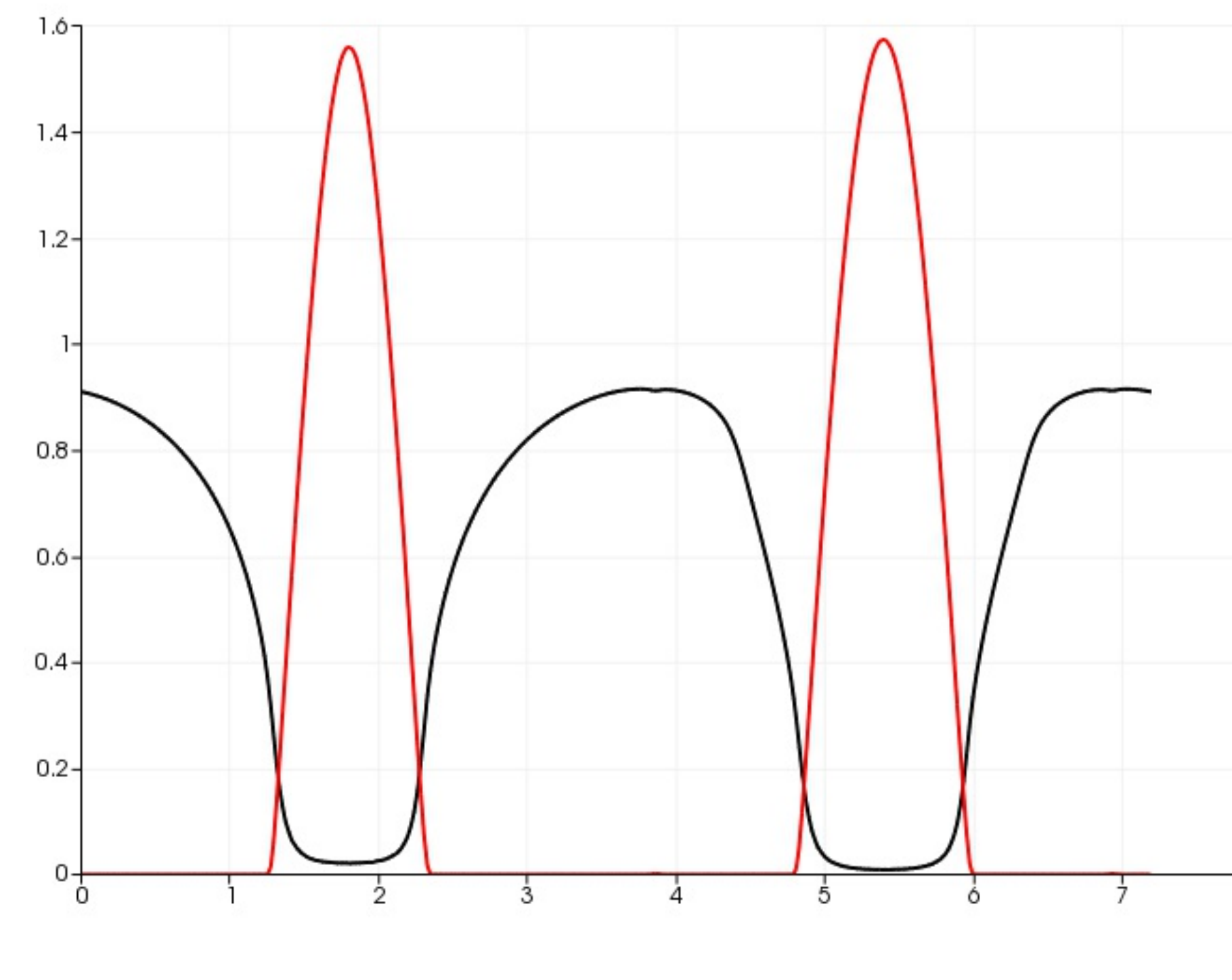}
\includegraphics[ trim = 5mm 140mm 20mm 10mm,  clip,  width=.24\textwidth
]{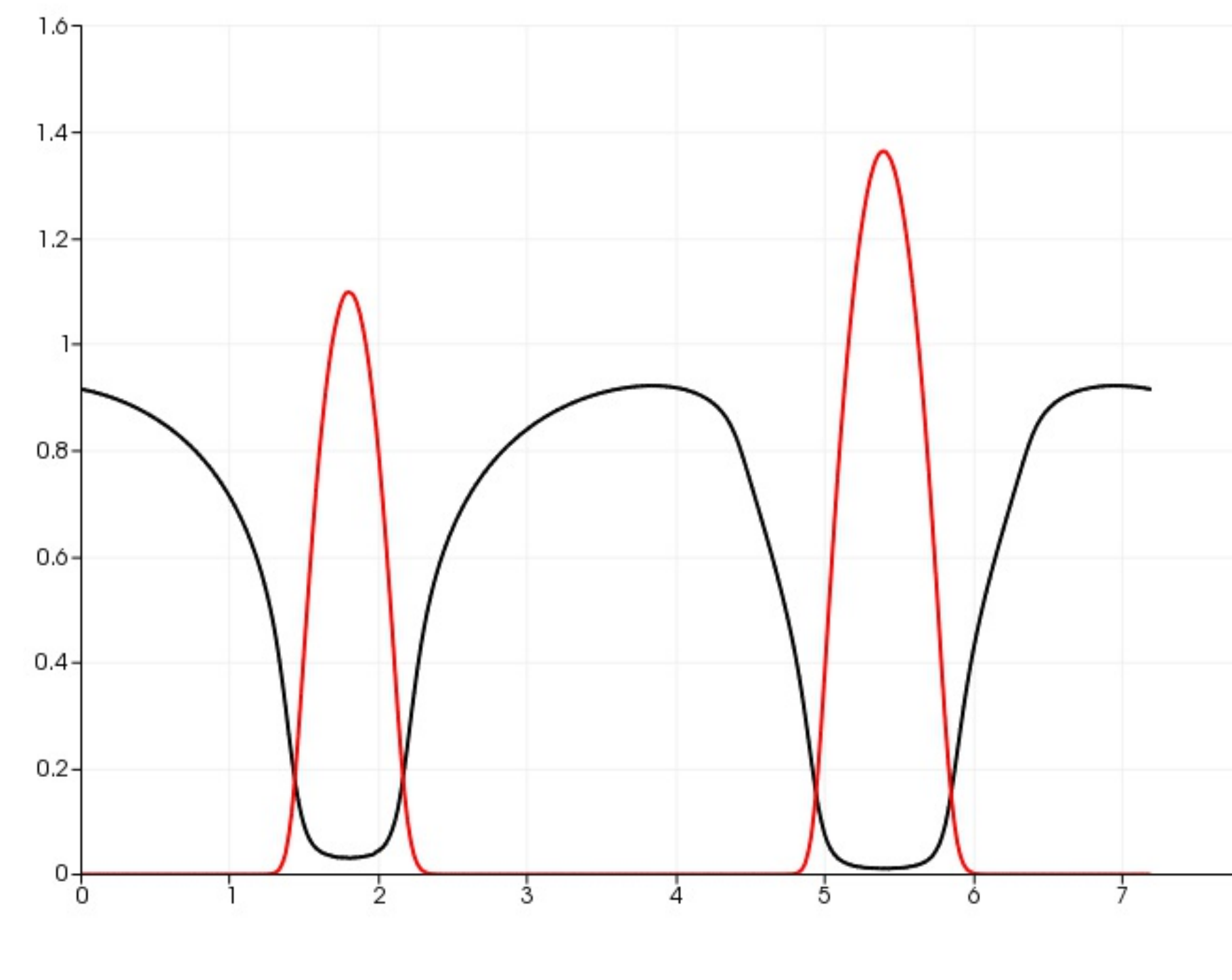}
\includegraphics[ trim = 5mm 140mm 20mm 10mm,  clip,  width=.24\textwidth
]{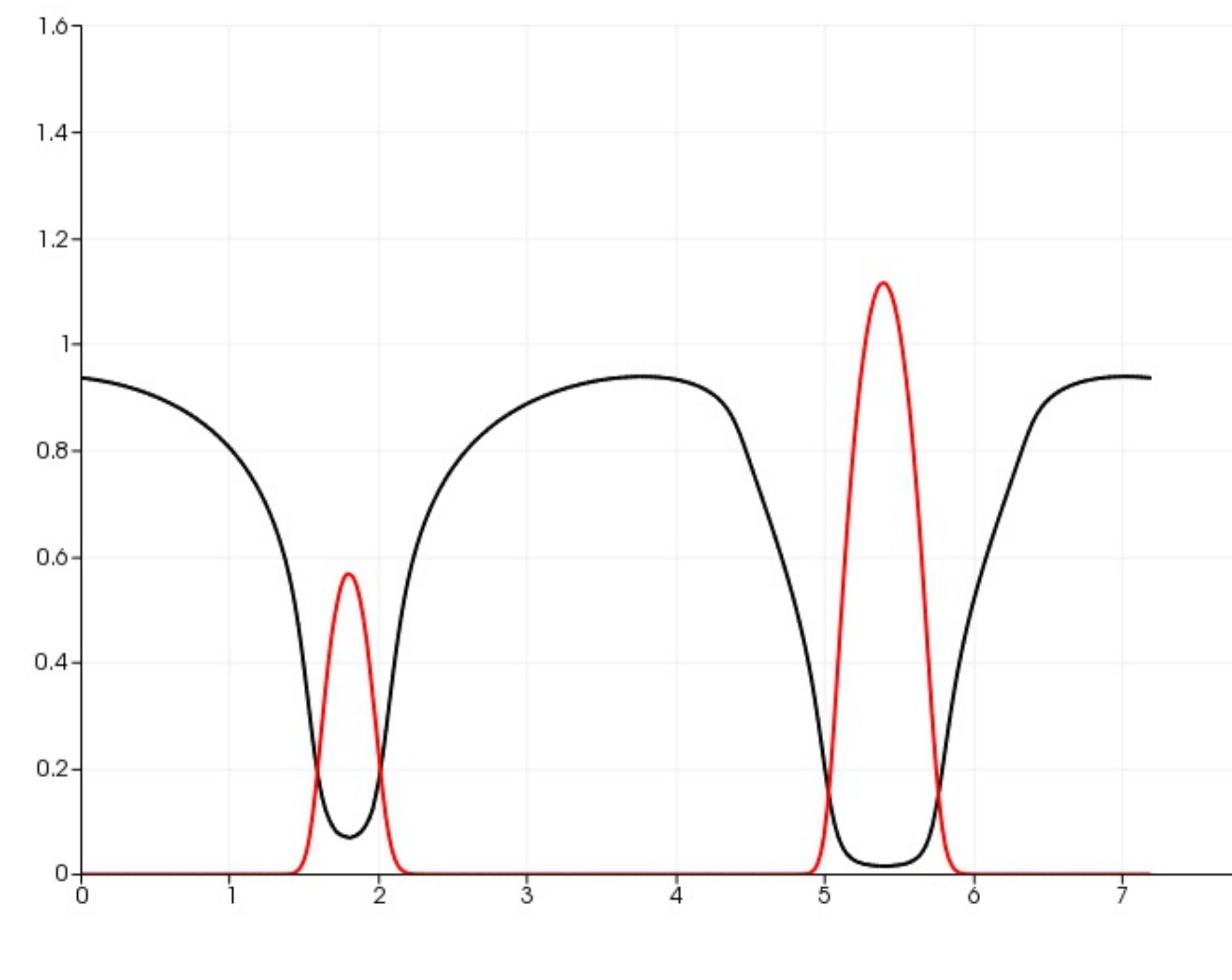}
\includegraphics[ trim = 5mm 140mm 20mm 10mm,  clip,  width=.24\textwidth
]{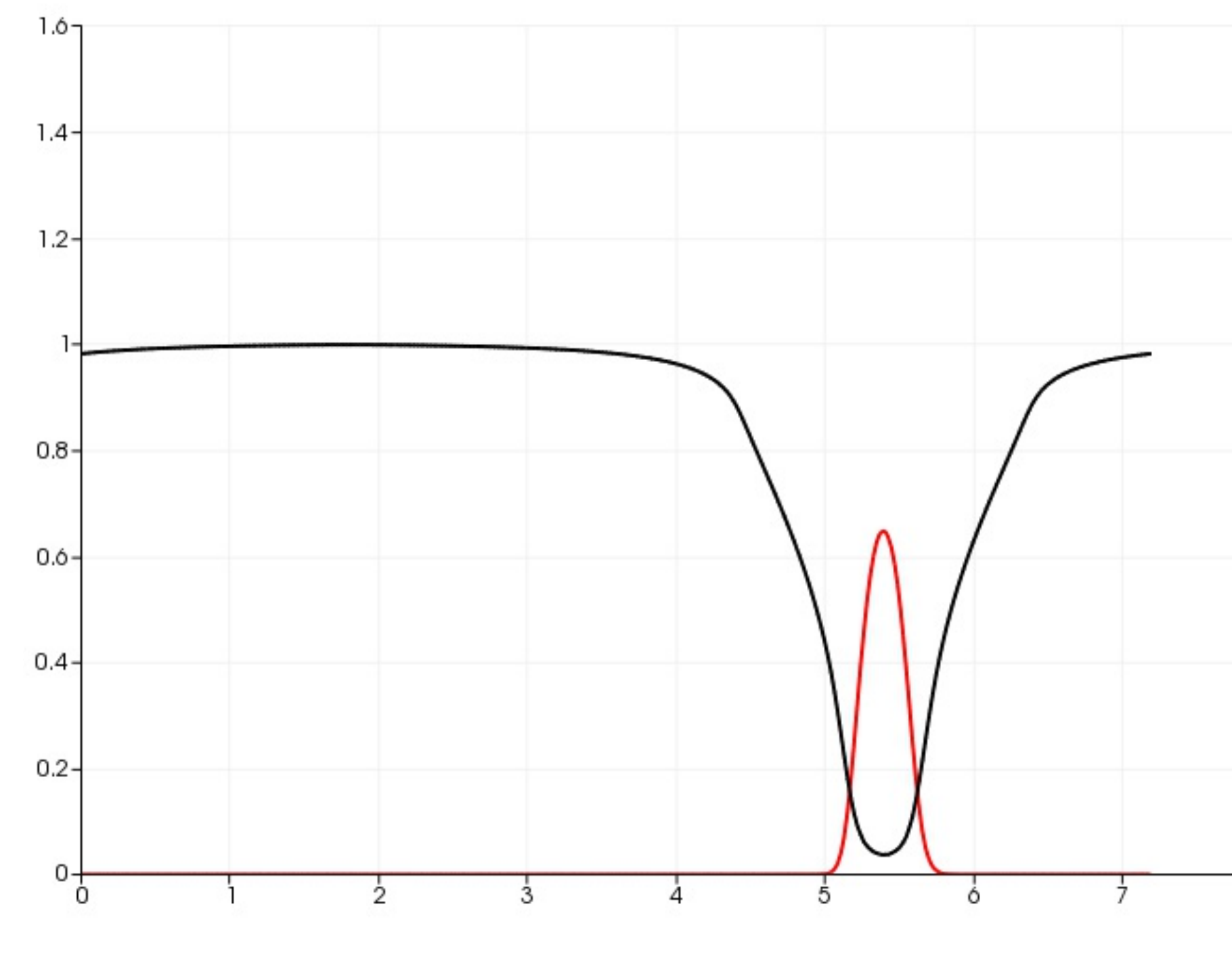}
}
\subfigure[][{$\eps=10^{-3}$}]{
\includegraphics[ trim = 5mm 140mm 20mm 10mm,  clip,  width=.24\textwidth
]{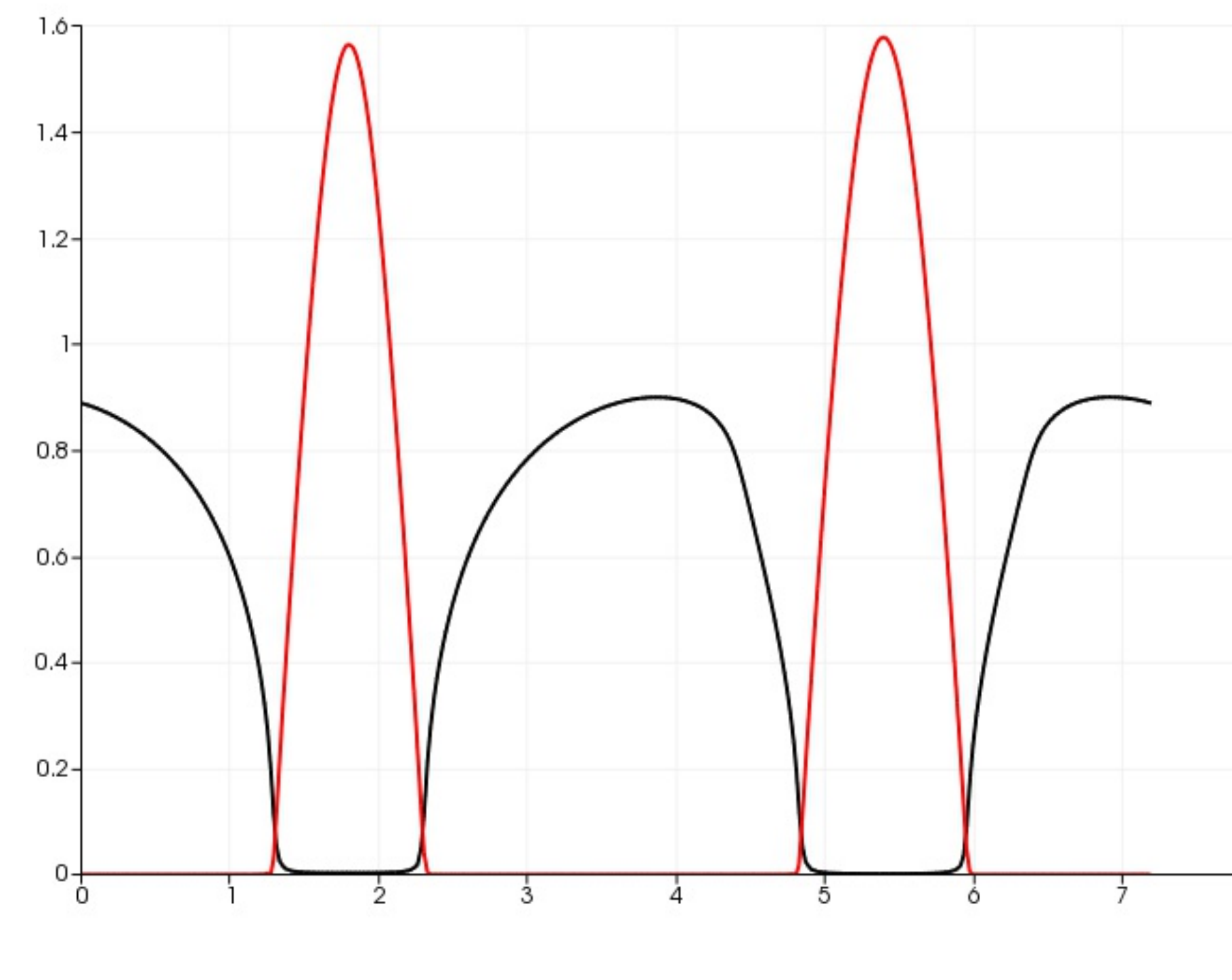}
\includegraphics[ trim = 5mm 140mm 20mm 10mm,  clip,  width=.24\textwidth
]{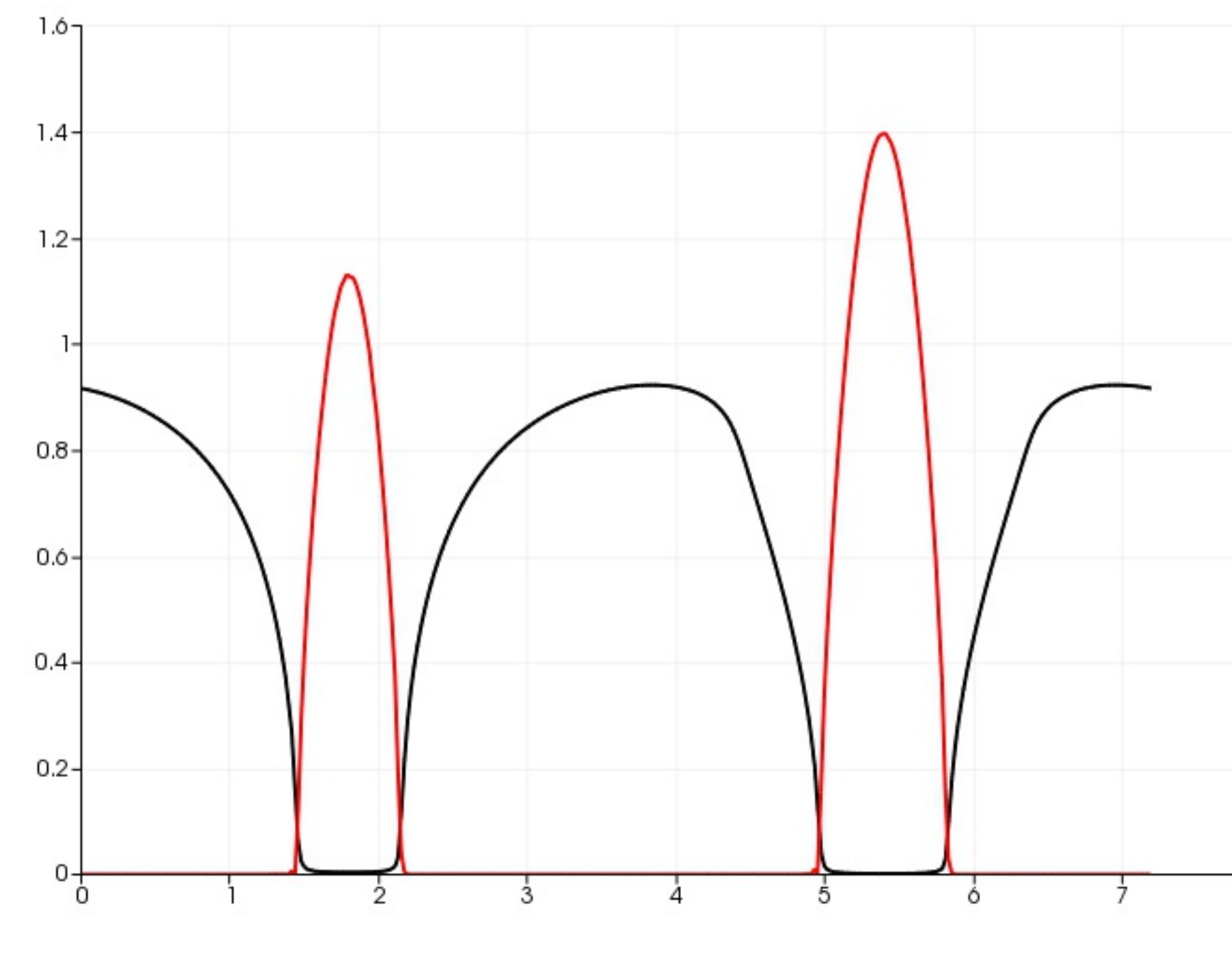}
\includegraphics[ trim = 5mm 140mm 20mm 10mm,  clip,  width=.24\textwidth
]{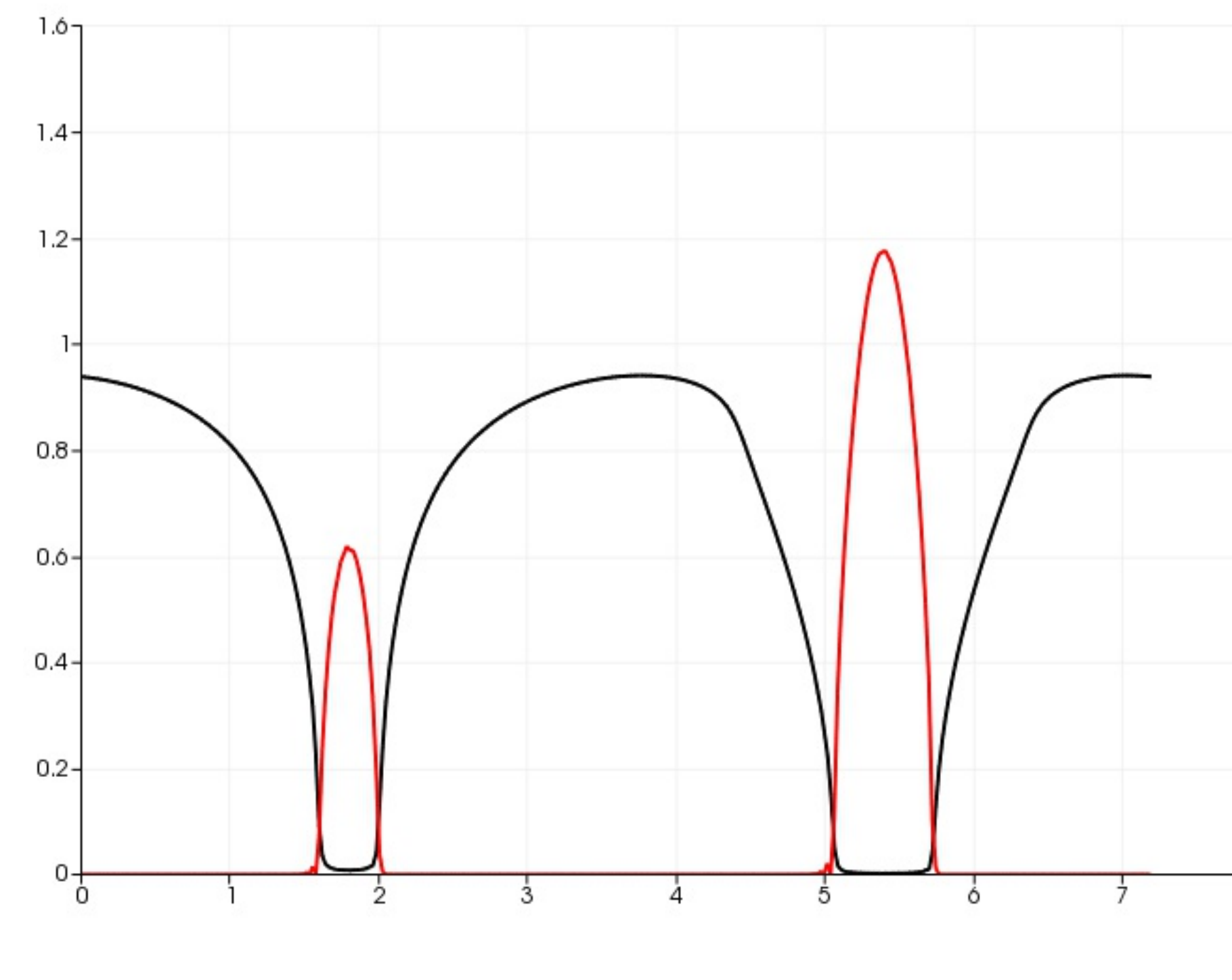}
\includegraphics[ trim = 5mm 140mm 20mm 10mm,  clip,  width=.24\textwidth
]{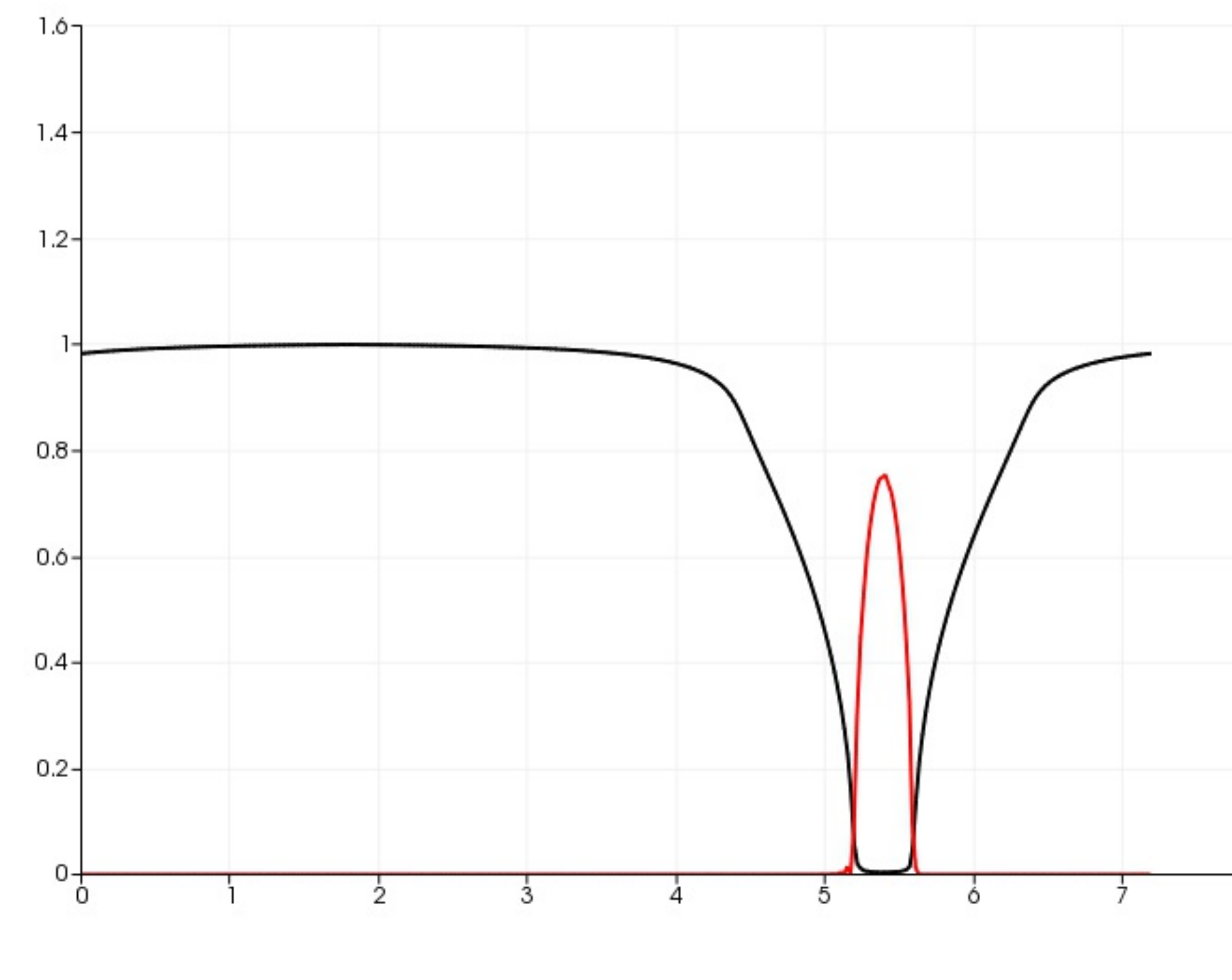}
}
\subfigure[][{$u^n=(z^{t_n}-z^{t_n-0.01})/0.01, \quad w^n=w^0+\nabla z^n\cdot\normal$}]{
\includegraphics[ trim = 5mm 120mm 20mm 10mm,  clip,  width=.24\textwidth
]{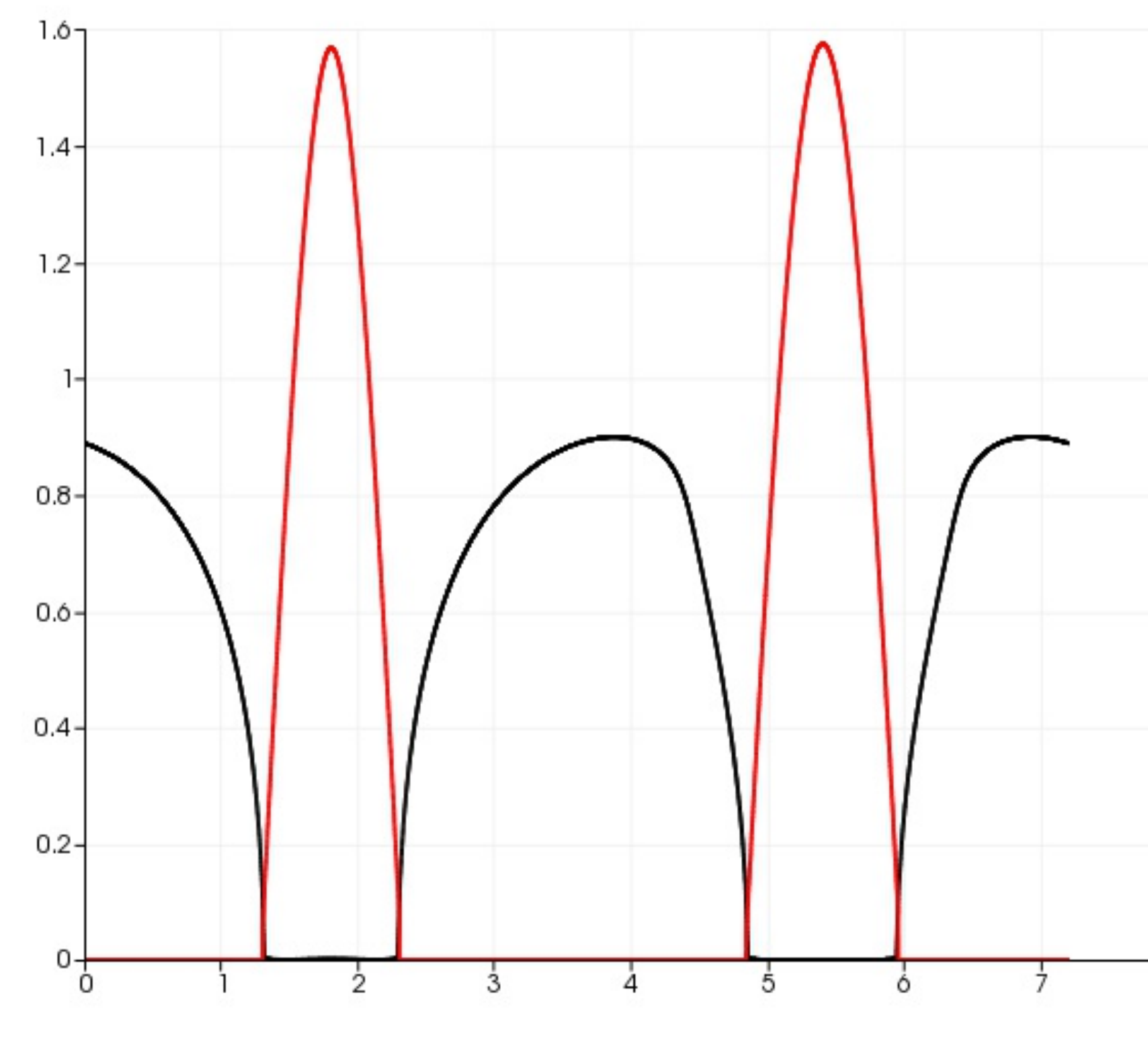}
\includegraphics[ trim = 5mm 120mm 20mm 10mm,  clip,  width=.24\textwidth
]{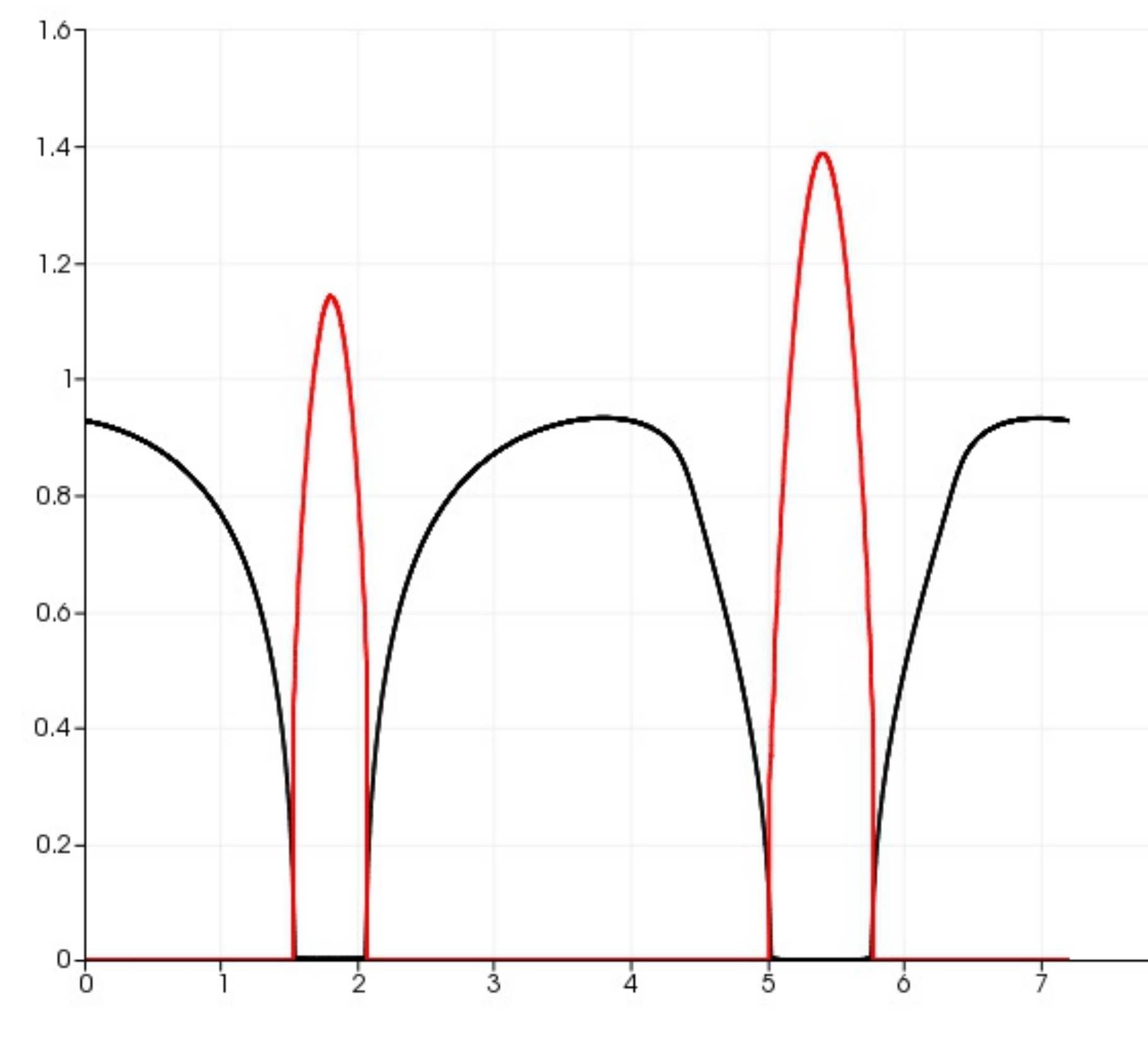}
\includegraphics[ trim = 5mm 120mm 20mm 10mm,  clip,  width=.24\textwidth
]{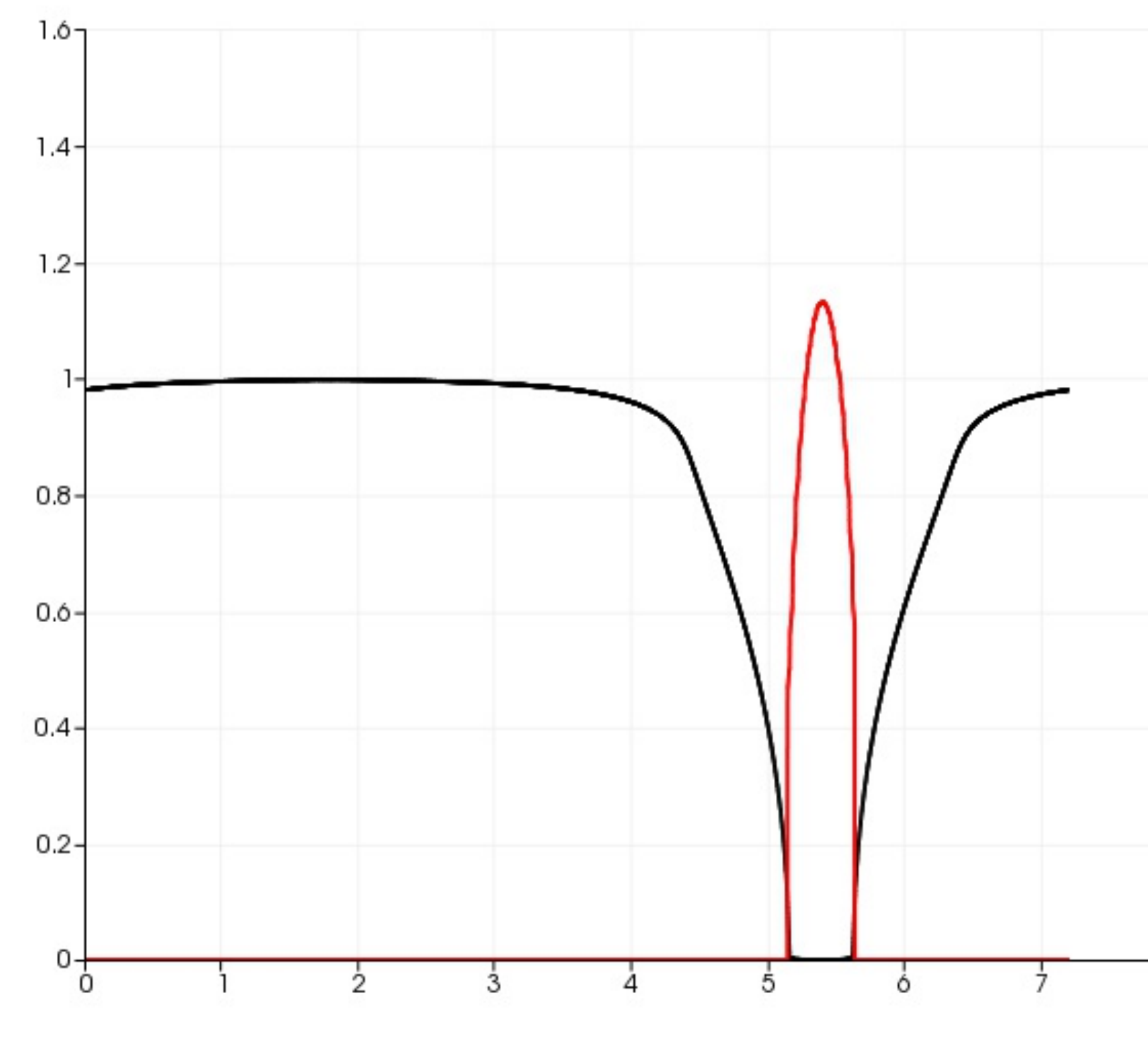}
\includegraphics[ trim = 5mm 120mm 20mm 10mm,  clip,  width=.24\textwidth
]{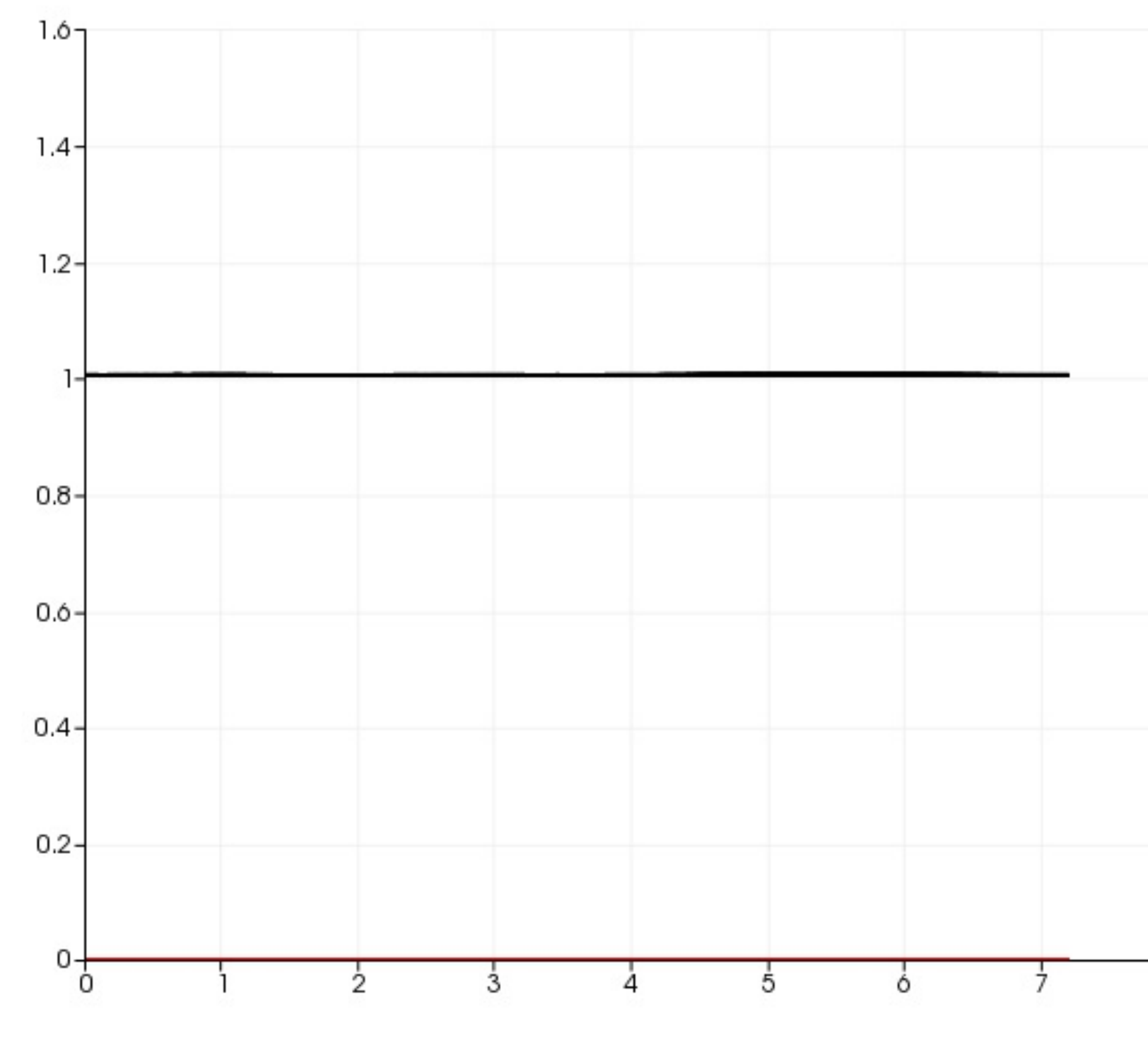}
}
\caption{Simulation results of \S \ref{sec:sim-2d}. (First three rows)  Plots of the trace of $U$ (black) and $W$ (red) of (\ref{eqn:wf_eps_problem}) over $\G_h$ at times $0.01, 0.2, 0.4$ and $0.7$ (reading from left to right) for different values of $\eps=\delO=\delG$. The fourth row shows plots of the trace of $U$ (black)  and $W=-V$ (red) post-processed from solving the elliptic variational inequality (\ref{eqn:EVI}) at times $0.01, 0.2, 0.4$ and $0.7$ reading from left to right.}\label{fig:2d-plot-uv}
\end{figure}

{In order to support our assertion that the changes observed in Figures \ref{fig:2d-res-uv} and \ref{fig:2d-plot-uv} are due to the changes in $\eps$ and not due to insufficient numerical resolution,  in Appendix \ref{sec:appendix} we investigate numerically the effect of 
the discretisation parameters, specifically the mesh-szie and the timestep, on the numerical solution. The results of   Appendix \ref{sec:appendix} illustrate that the large qualitative changes observed on reducing $\eps$ are due to the changing parameter rather than issues with numerical resolution.
}

\subsection{3D simulations}\label{sec:sim-3d}
We conclude this section with some 3D simulations. We set $\partial_0\O=\{\vec x\in\Reals^3\vert \lv \vec x\rv=2\}$, i.e., the surface of the sphere of radius two centred at the origin  and define the surface of the cell $\G$ by the level set function $\G=\{\vec x\in\Reals^3\vert (x_1+0.2-x_2^2)^2+4x_3^2+x_2^2-1=0\}$.  We generated a triangulation of the bulk domain (and the corresponding induced surface triangulation) using \Program{CGAL} \citep{cgal:ry-smg-13b}. We used a bulk mesh with 11167 DOFs and the induced surface triangulation had 2449 DOFs for the simulation of (\ref{eqn:wf_eps_problem}) whilst for the simulation of (\ref{eqn:EVI}) we used a finer mesh with 60583 bulk DOFs and  15169 surface DOFs. Figure \ref{fig:3d_triang} shows the computational domain used for all the simulation of (\ref{eqn:wf_eps_problem}). 

\begin{figure}
\includegraphics[trim = 10mm 150mm 10mm 10mm,  clip, width=0.32\textwidth]{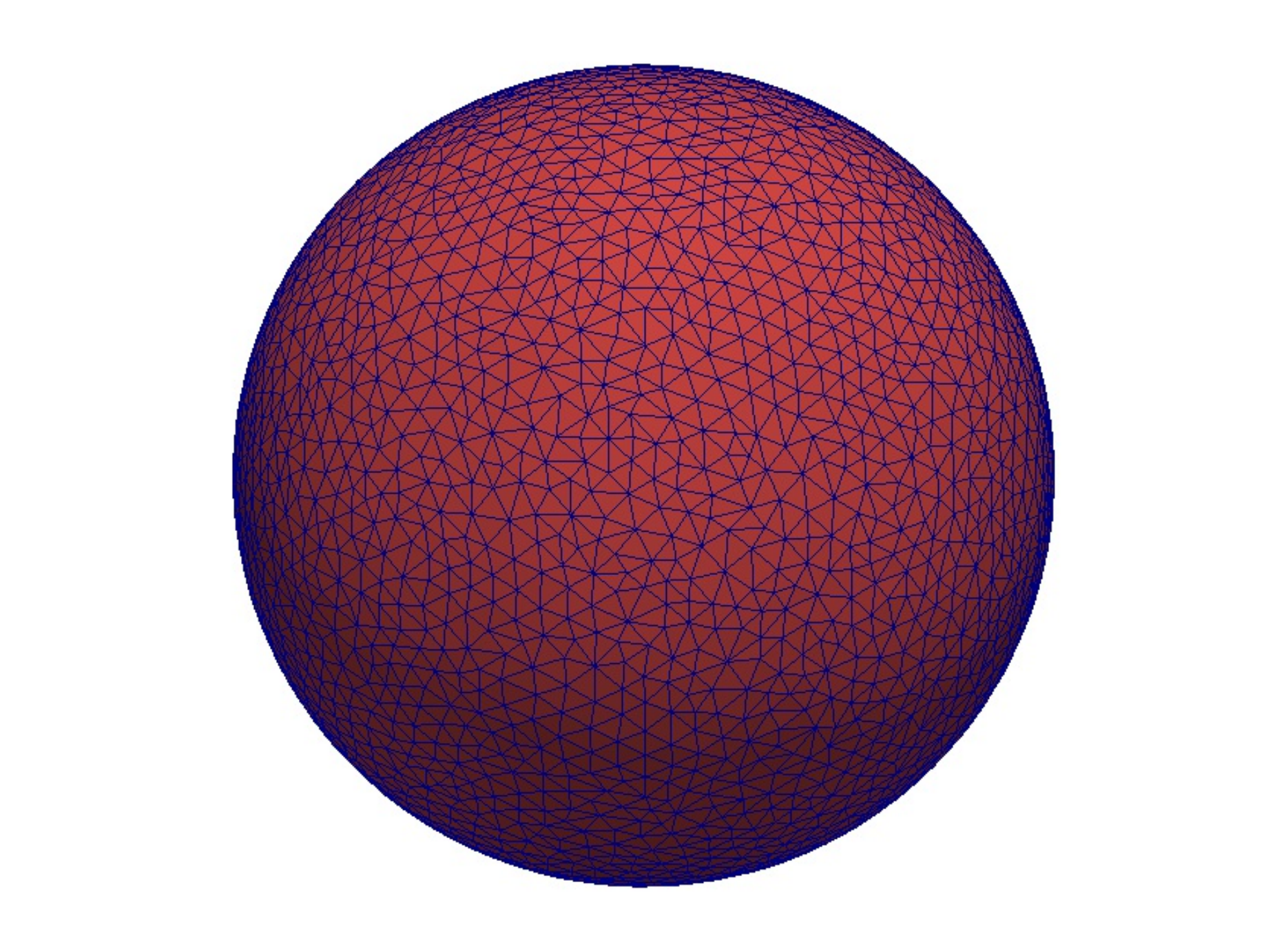}
\includegraphics[trim = 10mm 150mm 10mm 10mm,  clip, width=0.32\textwidth]{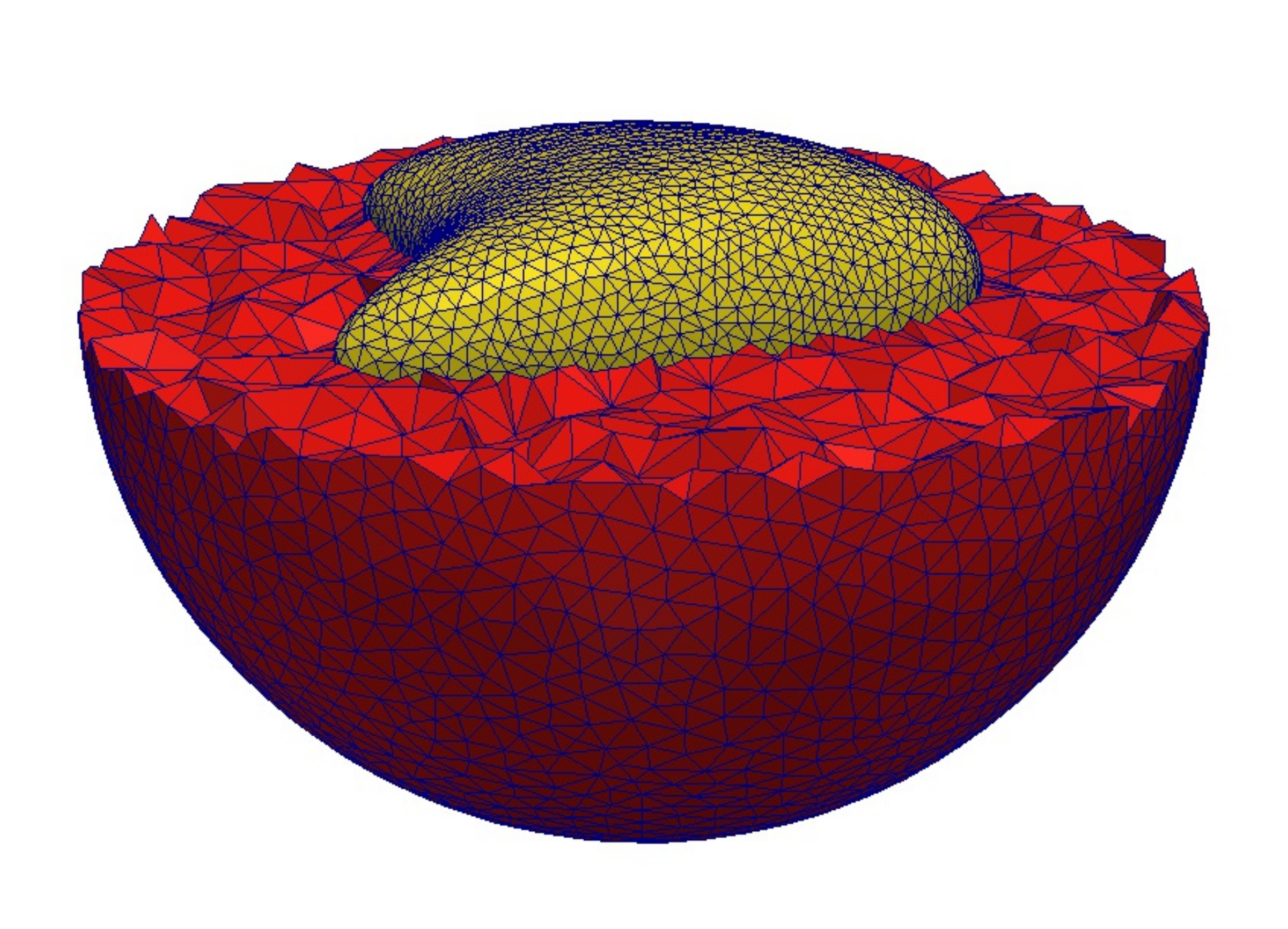}
\includegraphics[trim = 10mm 150mm 10mm 10mm,  clip, width=0.32\textwidth]{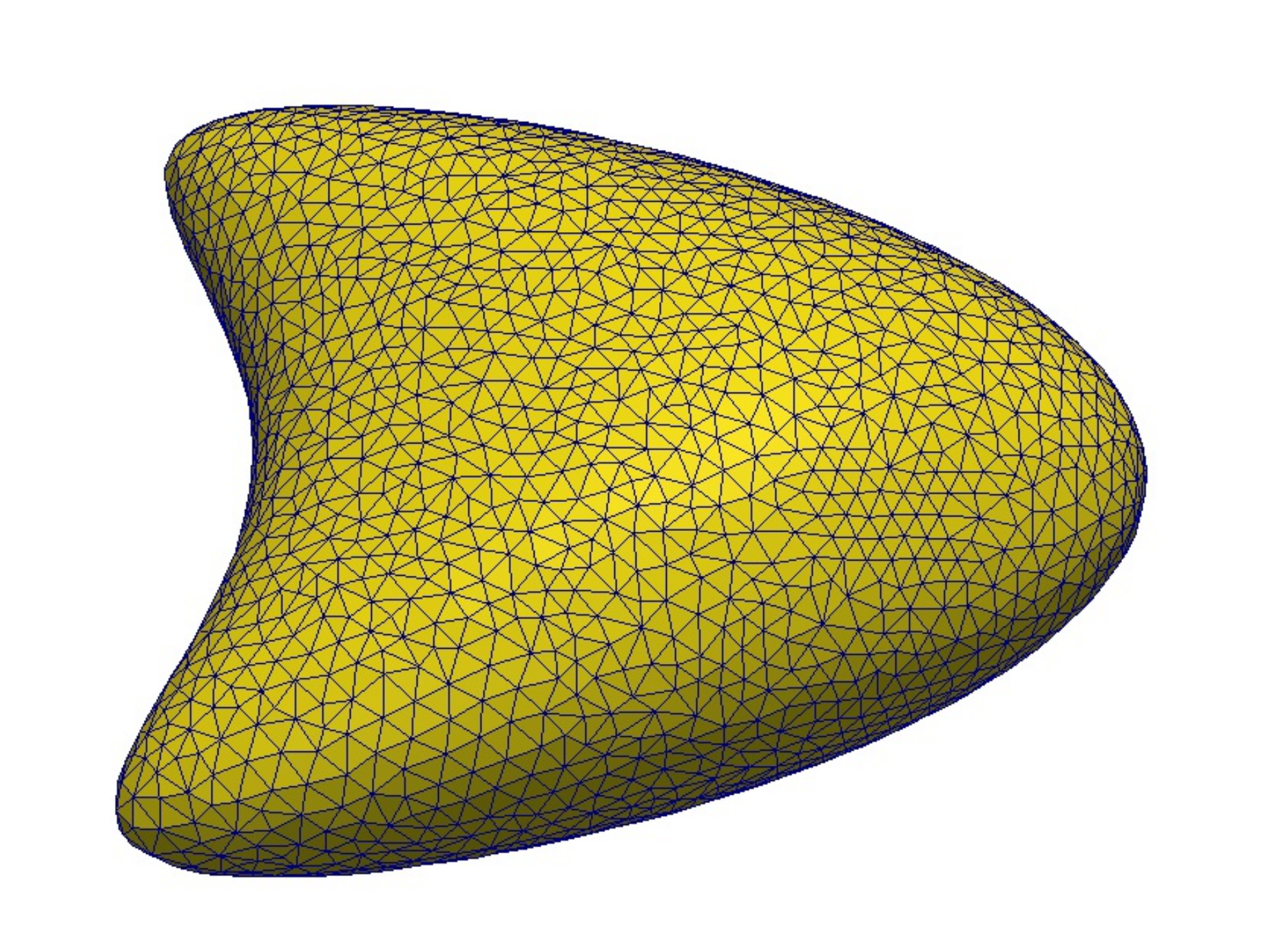}
\caption{The coarser computational domain used for the simulations in $3d$ of \S  \ref{sec:sim-3d}, generated using \Program{CGAL} \citep{cgal:ry-smg-13b}. The left figure shows the outer boundary of the bulk triangulation, the middle figure shows a the bulk triangulation with elements with their barycenters in the top half ($x_3>0$) removed together with the surface triangulation of the interior surface $\G_h$ and the right figure shows the triangulation of the surface $\G_h$.}
\label{fig:3d_triang}
\end{figure}

We report on the results of two simulations. We consider the approximation of (\ref{eqn:wf_eps_problem}) with $\eps=\delO=\delG=1\times10^{-2}$ and for the problem data we set $T=0.6$, $u_D=u^0=1$ and $w^0=\max(\cos(\pi x_2)+\sin(\pi x_1),0)$,  $\vec x\in\G$ and similarly to \S \ref{sec:sim-2d} we compare these results with those obtained from post-processing the solution to the elliptic variational inequality (\ref{eqn:EVI}) with $v^0=-w^0$. For the simulation of (\ref{eqn:wf_eps_problem}) we used a fixed uniform time step of $1\times10^{-6}$.  Snapshots of the solution $Z$ to (\ref{eqn:EVI}) at a series of distinct times is shown in Figure \ref{fig:3d-res-z}. As previously, to post-process $U^{t_m}:=(Z^{t_m}-Z^{t_m-\tau})/\tau$ we solve (\ref{eqn:EVI}) at  $t_m$ and $t_m-\tau$ fixing $\tau=0.01$. Figure \ref{fig:3d-res-uv} shows snapshots of the simulated $U$ and $W$. Analogous behaviour to the 2D case of \S \ref{sec:sim-2d} is observed.  We note that the solution of $Z$  shown in Figure \ref{fig:3d-res-z} appears quite smooth and the rough nature of the post-processed $U$ and $W$ may be an artefact of the post-processing together with the slice through the bulk triangulation taken for visualisation purposes. As noted in \S \ref{sec:var-ineq}, the elliptic variational inequality is a reformulation of the Hele-Shaw free boundary problem 
on the surface $\G$ with the differential operator now the half-Laplacian rather than the usual Laplacian  (Laplace-Beltrami).
We therefore conclude the numerical results section with Figure \ref{fig:3d-res-FB} which shows the evolution of the approximated free boundary on the surface $\G_h$. We approximate the position of the free boundary by plotting the level curve of the set where the trace of $Z=5\times10^{-3}$ at a series of times. 
\begin{figure}[htbp!]
\begin{minipage}{0.06\linewidth}
\includegraphics[ trim = 0mm 0mm 0mm 0mm,  clip,  width=\textwidth
]{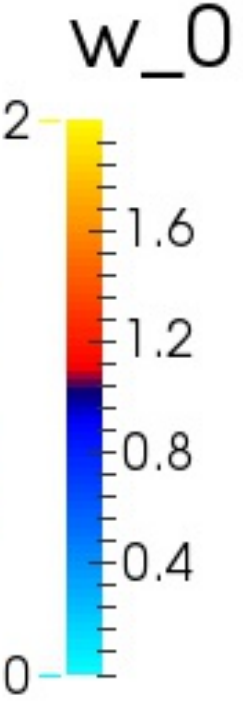}
\end{minipage}
\hskip .2em
\centering
\begin{minipage}{0.92\linewidth}
\includegraphics[ trim = 0mm 20mm 0mm 10mm,  clip,  width=0.24\textwidth
]{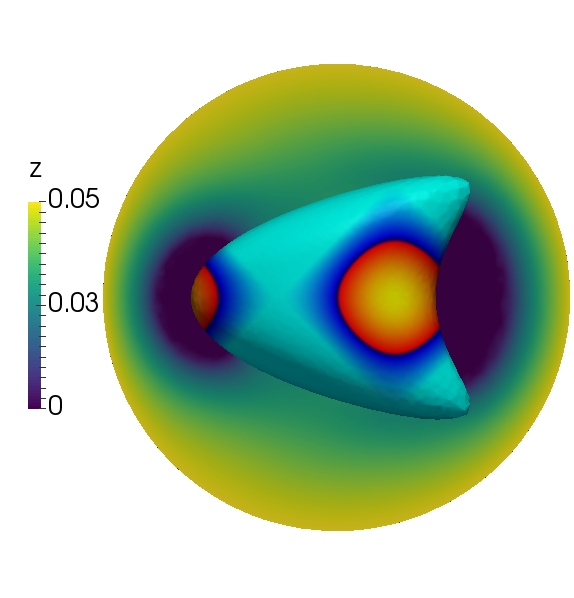}
\includegraphics[ trim = 0mm 20mm 0mm 10mm,  clip,  width=0.24\textwidth
]{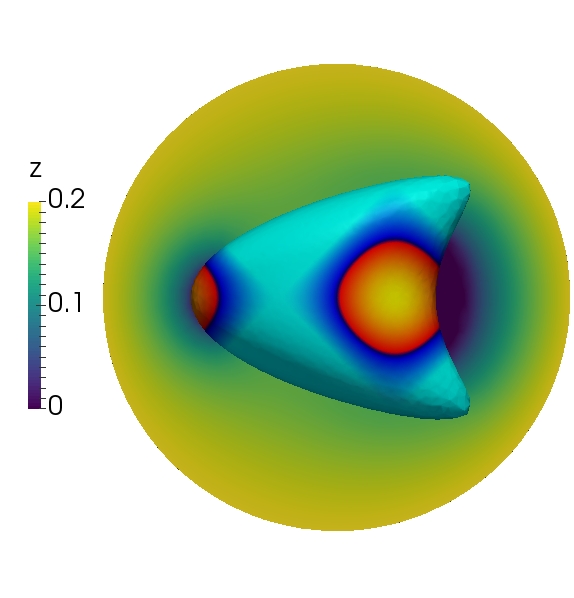}
\includegraphics[ trim = 0mm 20mm 0mm 10mm,  clip,  width=0.24\textwidth
]{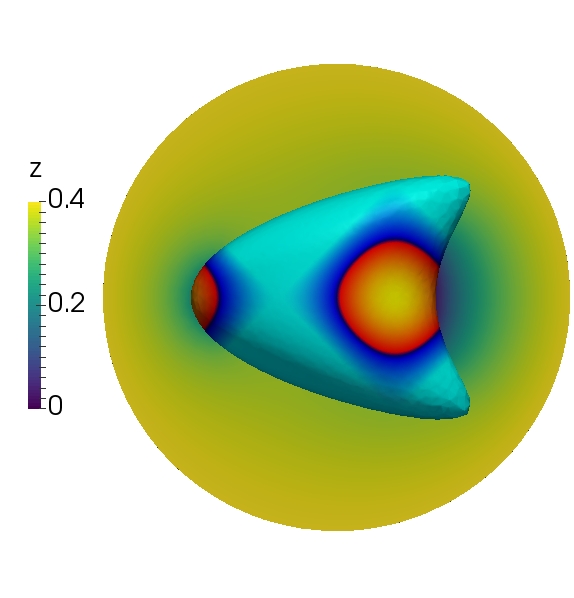}
\includegraphics[ trim = 0mm 20mm 0mm 10mm,  clip,  width=0.24\textwidth
]{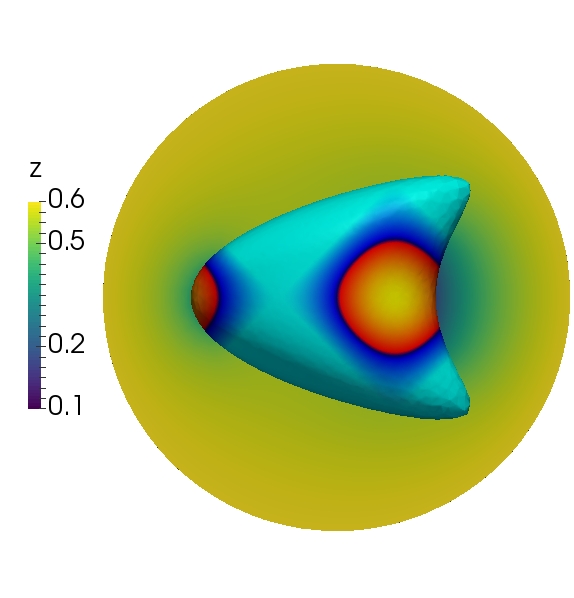}
\end{minipage}
\caption{Simulation results of \S \ref{sec:sim-3d}. Snapshots of the computed solution $Z$ together with the initial data  $W^0$ of the elliptic variational inequality (\ref{eqn:EVI}) at times $0.05, 0.2, 0.4$ and $0.6$ reading from left to right. The colour scale for $W^0$ is fixed in every figure. For visualisation, we have hidden the top half of the bulk domain (points with $x_3>0$).  }\label{fig:3d-res-z}
\end{figure}
\begin{figure}[htbp!]
\begin{minipage}{0.1\linewidth}
\includegraphics[ trim = 0mm 0mm 0mm 0mm,  clip,  width=\textwidth
]{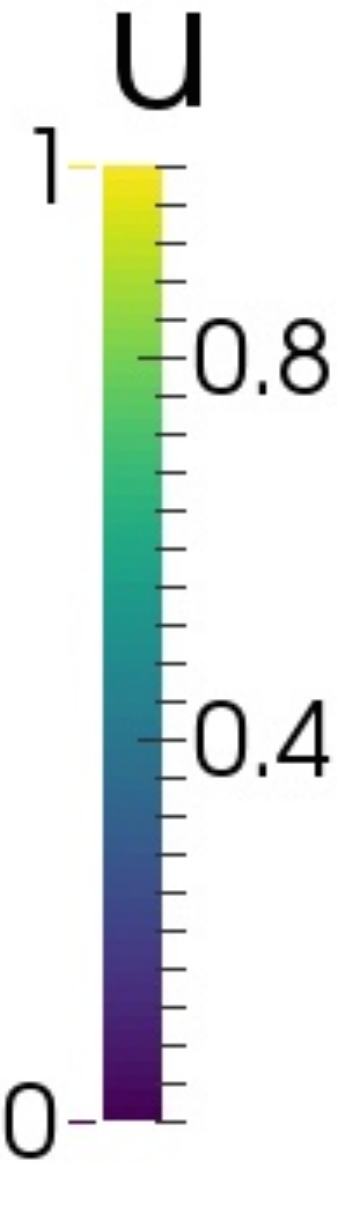}
\end{minipage}
\hskip .5em
\begin{minipage}{0.75\linewidth}
\centering
\subfigure[][{$\eps=0.01$}]{
\includegraphics[ trim = 0mm 0mm 0mm 0mm,  clip,  width=0.24\textwidth
]{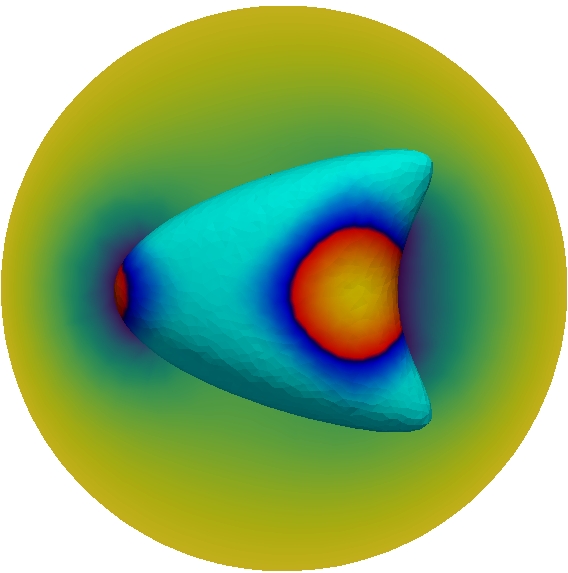}
\includegraphics[ trim = 0mm 0mm 0mm 0mm,  clip,  width=0.24\textwidth
]{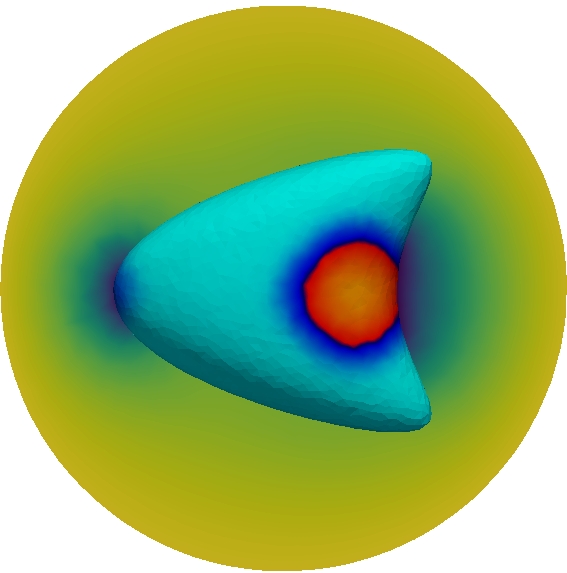}
\includegraphics[ trim = 0mm 0mm 0mm 0mm,  clip,  width=0.24\textwidth
]{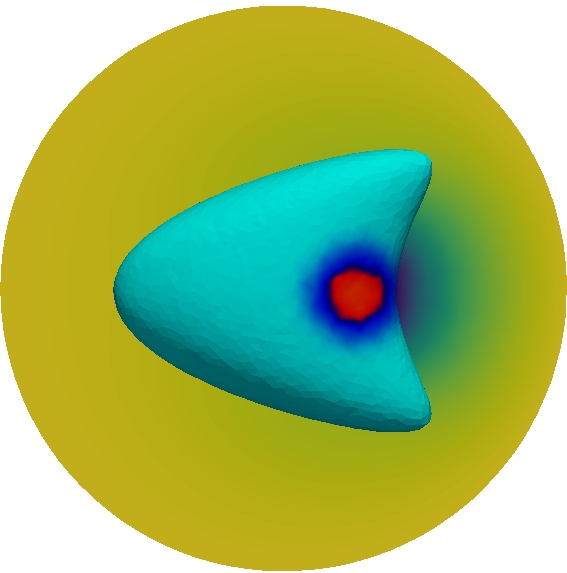}
\includegraphics[ trim = 0mm 0mm 0mm 0mm,  clip,  width=0.24\textwidth
]{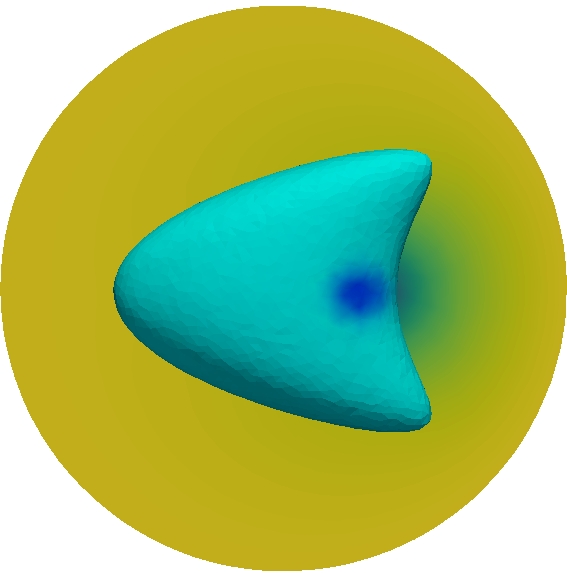}
}
\subfigure[][{$u^m=(z^{t_m}-z^{t_m-0.01})/0.01, \quad w^m=w^0+\nabla z^m\cdot\normal$}]{
\includegraphics[ trim = 0mm 0mm 0mm 0mm,  clip,  width=0.24\textwidth
]{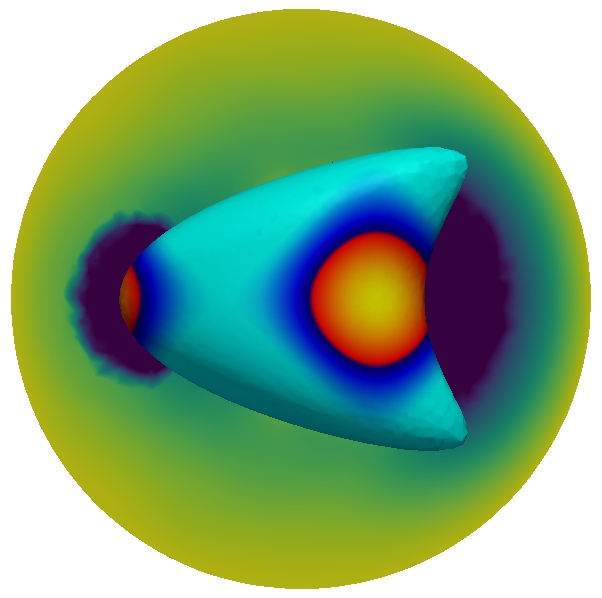}
\includegraphics[ trim = 0mm 0mm 0mm 0mm,  clip,  width=0.24\textwidth
]{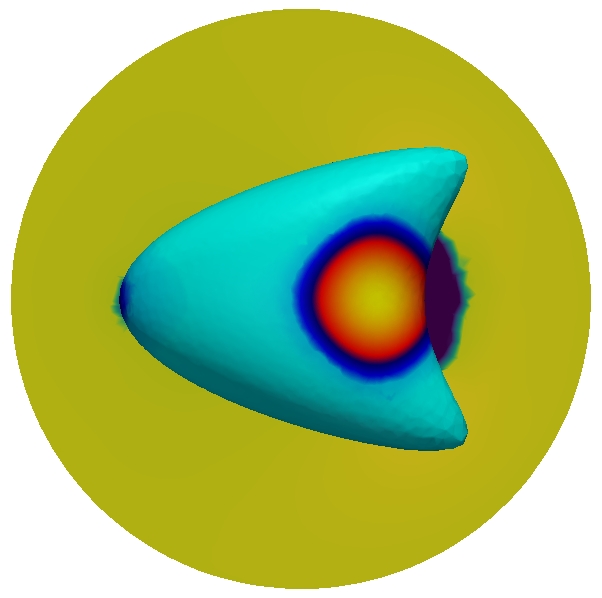}
\includegraphics[ trim = 0mm 0mm 0mm 0mm,  clip,  width=0.24\textwidth
]{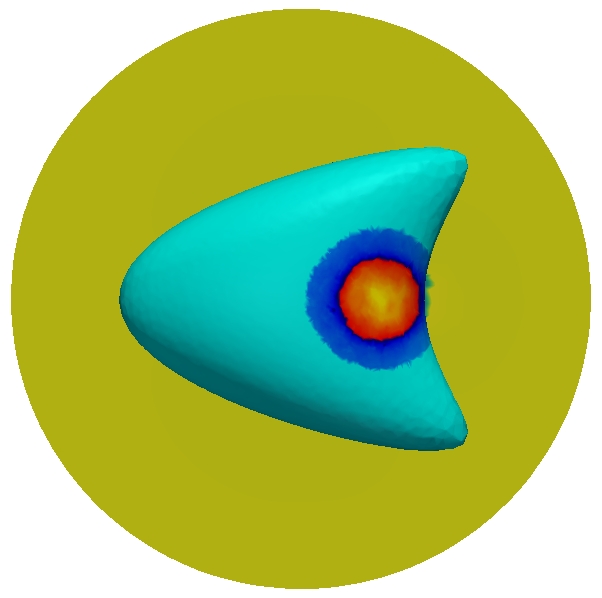}
\includegraphics[ trim = 0mm 0mm 0mm 0mm,  clip,  width=0.24\textwidth
]{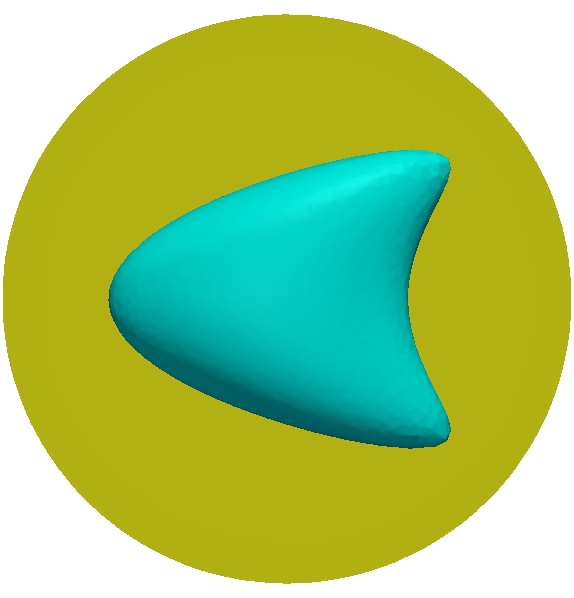}
}
\end{minipage}
\hskip .5em
\begin{minipage}{0.1\linewidth}
\includegraphics[ trim = 0mm 0mm 0mm 0mm,  clip,  width=\textwidth
]{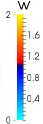}
\end{minipage}
\caption{Simulation results of \S \ref{sec:sim-3d}. Top row, snapshots of the computed solutions $U$ and $W$ of (\ref{eqn:wf_eps_problem}) in 3D at times $0.05, 0.2, 0.4$ and $0.6$ (reading from left to right) for $\eps=\delO=\delG=0.01$ on a coarser mesh. Bottom row, the computed solutions $U$ and $W=-V$ post-processed from solving the elliptic variational inequality (\ref{eqn:EVI}) at times $0.05, 0.2, 0.4$ and $0.6$ reading from left to right on a finer mesh. For visualisation, we have hidden the top half of the bulk domain (points with $x_3>0$).  }\label{fig:3d-res-uv}
\end{figure}
\begin{figure}[htbp!]
\centering
\includegraphics[ trim = 10mm 120mm 10mm 0mm,  clip,  width=0.5\textwidth
]{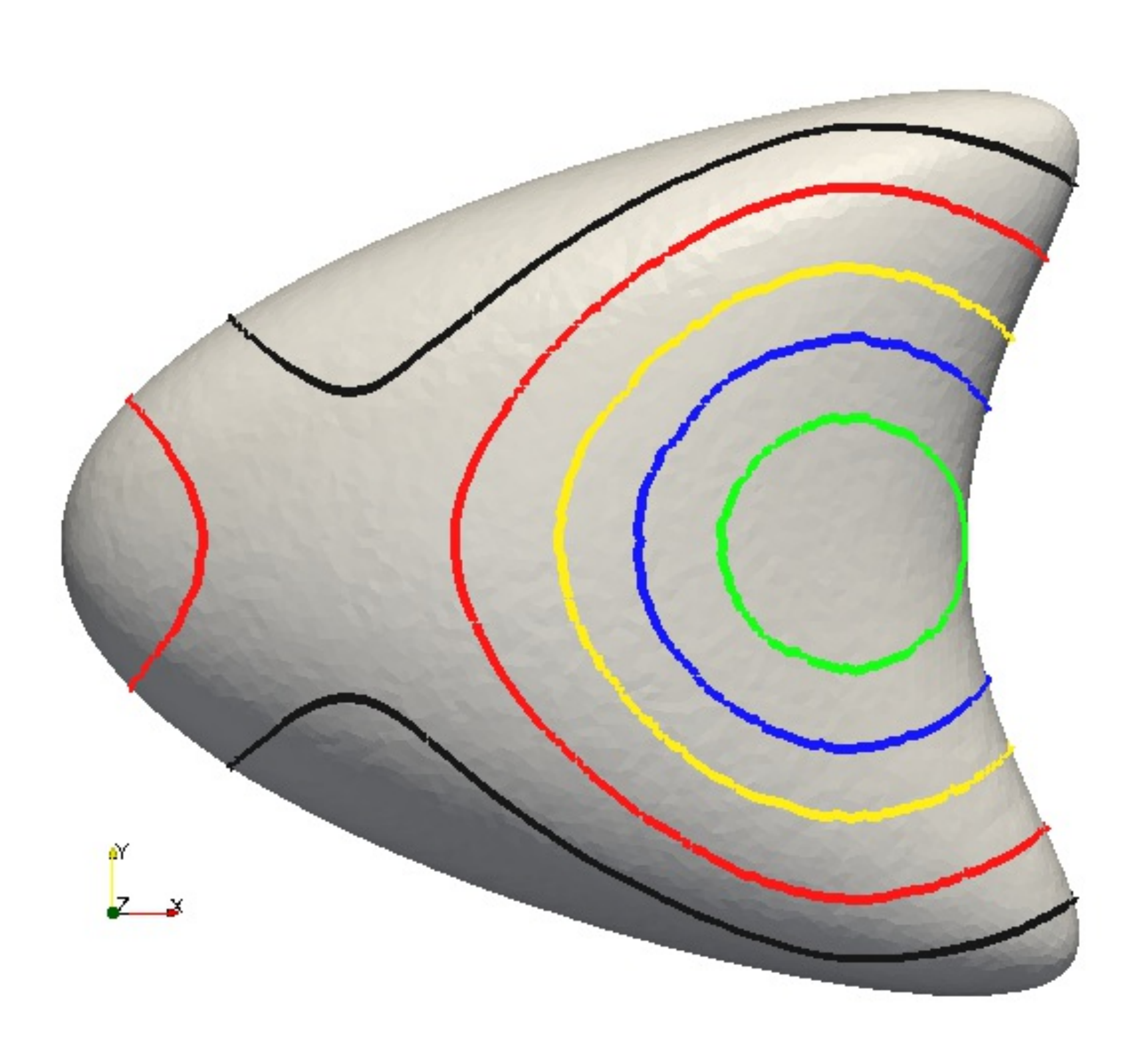}
\caption{Simulation results of \S \ref{sec:sim-3d}. Snapshots of the level curve on which the trace of $Z=5\times 10^{-3}$
that approximates the free boundary in the elliptic variational inequality (\ref{eqn:EVI}) and thus the surface Hele-Shaw problem (\ref{eqn:ell_VI_DtN}) at times $0.05$ (black), $0.15$ (red), $0.25$ (yellow), $0.35$ (blue) and $0.45$ (green).}\label{fig:3d-res-FB}
\end{figure}

\section{Conclusion}\label{sec:conc}
In this work we developed a well-posedness theory for a system of coupled bulk-surface PDEs with nonlinear coupling. The system under consideration arises naturally as a simplification of models for receptor-ligand dynamics in cell biology and hence developing a rigorous mathematical framework for the treatment of such systems is an important task due to their widespread use in modelling and computational studies, e.g.,  \citep{garcia2013mathematical, levine2005membrane, MadChuVen14,bao2014well}. Whilst the model we consider (\ref{eq:eps-problem}) is a simplified model problem,  the nonlinear coupling between the bulk and surface species is preserved and this is expected to be the main difficulty in the mathematical understanding of more biologically complex models of  receptor-ligand interactions. Thus our techniques should be applicable to many of the models derived and simulated in the literature. 

On non-dimensionalisation of the model using experimentally estimated parameter values, we identified three biologically meaningful asymptotic (small-parameter) limits of the model. We present a rigorous derivation of the limiting problems which correspond to free boundary problems on the surface of the cell and we demonstrated the well-posedness of the free boundary problems. Moreover, we discussed connections between the different free boundary problems and classical free boundary problems, namely the one-phase Stefan problem and the Hele-Shaw  problem. This perspective gives rise to the possibility of using these ideas when constructing receptor-ligand models with other mechanisms.

Finally, we reported on numerical simulations of the original problem (\ref{eq:eps-problem}) and a suitable reformulation of the elliptic limiting problem obtained when one considers fast reaction, slow surface diffusion and fast bulk diffusion. The simulation results illustrated the convergence towards the limiting problem thereby supporting our theoretical findings. We note that the reformulated problem is considerably cheaper to solve computationally. Hence in a biological setting where one is in a parameter regime in which the limiting problem provides a good approximation to the original problem it may be preferable to solve the limiting free boundary problem rather than the original coupled system of parabolic equations.

\section*{Acknowledgements}

This work was started whilst the authors were participants in the Isaac Newton Institute programme: ``Free Boundary Problems and Related Topics'' and finalised whilst the authors were participants in the Isaac Newton Institute programme: ``Coupling Geometric PDEs with Physics for Cell Morphology, Motility and Pattern Formation'' supported by EPSRC Grant Number EP/K032208/1. The work of CV received support from the Leverhulme Trust Research Project Grant (RPG-2014-149). {The authors would like to acknowledge the anonymous reviewers whose comments improved the manuscript.}


\bibliography{./ERV_RL.bib}

\begin{thebibliography}{57}
\providecommand{\natexlab}[1]{#1}
\providecommand{\url}[1]{\texttt{#1}}
\expandafter\ifx\csname urlstyle\endcsname\relax
  \providecommand{\doi}[1]{doi: #1}\else
  \providecommand{\doi}{doi: \begingroup \urlstyle{rm}\Url}\fi

\bibitem[Aitchison et~al.(1984)Aitchison, Lacey, and
  Shillor]{aitchison1984model}
J.~Aitchison, A.~Lacey, and M.~Shillor.
\newblock A model for an electropaint process.
\newblock \emph{IMA Journal of Applied Mathematics}, 33\penalty0 (1):\penalty0
  17--31, 1984.

\bibitem[Aitchison et~al.(1983)Aitchison, Elliott, and
  Ockendon]{aitchison1983percolation}
J.~M. Aitchison, C.~M. Elliott, and J.~R. Ockendon.
\newblock Percolation in gently sloping beaches.
\newblock \emph{IMA Journal of Applied Mathematics}, 30\penalty0 (3):\penalty0
  269--287, 1983.

\bibitem[Alphonse et~al.(2016)Alphonse, Elliott, and Terra]{AlpEllTer-pp}
A.~Alphonse, C.~M. Elliott, and J.~Terra.
\newblock A coupled bulk surface system in evolving domains modelling ligand
  receptor dynamics.
\newblock Work in preparation, 2016.

\bibitem[Athanasopoulos and Caffarelli(2010)]{athanasopoulos2010continuity}
I.~Athanasopoulos and L.~A. Caffarelli.
\newblock Continuity of the temperature in boundary heat control problems.
\newblock \emph{Advances in Mathematics}, 224\penalty0 (1):\penalty0 293--315,
  2010.

\bibitem[Bao et~al.(2014)Bao, Fellner, and Latos]{bao2014well}
T.~Q. Bao, K.~Fellner, and E.~Latos.
\newblock Well-posedness and exponential equilibration of a volume-surface
  reaction-diffusion system with nonlinear boundary coupling.
\newblock \emph{arXiv preprint arXiv:1404.2809}, 2014.

\bibitem[Bongrand(1999)]{bongrand1999ligand}
P.~Bongrand.
\newblock Ligand-receptor interactions.
\newblock \emph{Reports on Progress in Physics}, 62\penalty0 (6):\penalty0 921,
  1999.

\bibitem[Bothe(2001)]{bothe2001instantaneous}
D.~Bothe.
\newblock The instantaneous limit of a reaction-diffusion system.
\newblock \emph{Lecture Notes in Pure and Applied Mathematics}, pages 215--224,
  2001.

\bibitem[Bothe and Pierre(2012)]{bothe2012instantaneous}
D.~Bothe and M.~Pierre.
\newblock The instantaneous limit for reaction-diffusion systems with a fast
  irreversible reaction.
\newblock \emph{Discrete and Continuous Dynamical Systems - Series S},
  5:\penalty0 49--59, 2012.

\bibitem[Caffarelli and Silvestre(2007)]{caffarelli2007extension}
L.~Caffarelli and L.~Silvestre.
\newblock An extension problem related to the fractional {L}aplacian.
\newblock \emph{Communications in Partial Differential equations}, 32\penalty0
  (8):\penalty0 1245--1260, 2007.

\bibitem[Caffarelli and Friedman(1985)]{caffarelli1985nonlinear}
L.~A. Caffarelli and A.~Friedman.
\newblock A nonlinear evolution problem associated with an electropaint
  process.
\newblock \emph{SIAM Journal on Mathematical Analysis}, 16\penalty0
  (5):\penalty0 955--969, 1985.

\bibitem[Calatroni and Colli(2013)]{calatroni2013global}
L.~Calatroni and P.~Colli.
\newblock Global solution to the allen{\textendash}cahn equation with singular
  potentials and dynamic boundary conditions.
\newblock \emph{Nonlinear Analysis: Theory, Methods {\&} Applications},
  79:\penalty0 12--27, mar 2013.
\newblock \doi{10.1016/j.na.2012.11.010}.
\newblock URL \url{http://dx.doi.org/10.1016/j.na.2012.11.010}.

\bibitem[Colli and Kenmochi(1987)]{colli1987nonlinear}
P.~Colli and N.~Kenmochi.
\newblock Nonlinear semigroup approach to a class of evolution equations
  arising from percolation in sandbanks.
\newblock \emph{Annali di Matematica Pura ed Applicata}, 149\penalty0
  (1):\penalty0 113--133, 1987.

\bibitem[Conti et~al.(2005)Conti, Terracini, and Verzini]{conti2005asymptotic}
M.~Conti, S.~Terracini, and G.~Verzini.
\newblock Asymptotic estimates for the spatial segregation of competitive
  systems.
\newblock \emph{Advances in Mathematics}, 195\penalty0 (2):\penalty0 524--560,
  aug 2005.
\newblock \doi{10.1016/j.aim.2004.08.006}.
\newblock URL \url{http://dx.doi.org/10.1016/j.aim.2004.08.006}.

\bibitem[Crooks et~al.(2004)Crooks, Dancer, Hilhorst, Mimura, and
  Ninomiya]{crooks2004spatial}
E.~Crooks, E.~Dancer, D.~Hilhorst, M.~Mimura, and H.~Ninomiya.
\newblock Spatial segregation limit of a competition-diffusion system with
  dirichlet boundary conditions.
\newblock \emph{Nonlinear Analysis: Real World Applications}, 5\penalty0
  (4):\penalty0 645--665, sep 2004.
\newblock \doi{10.1016/j.nonrwa.2004.01.004}.
\newblock URL \url{http://dx.doi.org/10.1016/j.nonrwa.2004.01.004}.

\bibitem[Crowley(1979)]{crowley1979weak}
A.~Crowley.
\newblock On the weak solution of moving boundary problems.
\newblock \emph{IMA Journal of Applied Mathematics}, 24\penalty0 (1):\penalty0
  43--57, 1979.

\bibitem[Dancer et~al.(1999)Dancer, Hilhorst, Mimura, and
  Peletier]{dancer1999spatial}
E.~Dancer, D.~Hilhorst, M.~Mimura, and L.~Peletier.
\newblock Spatial segregation limit of a competition--diffusion system.
\newblock \emph{European Journal of Applied Mathematics}, 10\penalty0
  (02):\penalty0 97--115, 1999.

\bibitem[Duvaut(1973)]{duvaut1973resolution}
G.~Duvaut.
\newblock R{\'e}solution d'un probleme de {S}tefan (fusion d'un bloc de glacea
  z{\'e}ro degr{\'e}).
\newblock \emph{CR Acad. Sci. Paris S{\'e}r. AB}, 276:\penalty0 A1461--A1463,
  1973.

\bibitem[Elliott(1980)]{ell80}
C.~M. Elliott.
\newblock On a variational inequality formulation of an electrochemical
  machining moving boundary problem and its approximation by the finite element
  method.
\newblock \emph{{IMA} Journal of Applied Mathematics}, 25:\penalty0 121--131,
  1980.

\bibitem[Elliott and Friedman(1985)]{elliott1985analysis}
C.~M. Elliott and A.~Friedman.
\newblock Analysis of a model of percolation in a gently sloping sand-bank.
\newblock \emph{SIAM Journal on Mathematical Analysis}, 16\penalty0
  (5):\penalty0 941--954, 1985.

\bibitem[Elliott and Janovsk{\`y}(1981)]{elliott1981variational}
C.~M. Elliott and V.~Janovsk{\`y}.
\newblock A variational inequality approach to {H}ele-{S}haw flow with a moving
  boundary.
\newblock \emph{Proceedings of the Royal Society of Edinburgh: Section A
  Mathematics}, 88\penalty0 (1-2):\penalty0 93--107, 1981.

\bibitem[Elliott and Ockendon(1982)]{EllOck82}
C.~M. Elliott and J.~R. Ockendon.
\newblock \emph{Weak and variational methods for moving boundary problems},
  volume~59 of \emph{Research Notes in Mathematics Series}.
\newblock Pitman, London, 1982.

\bibitem[Elliott and Ranner(2013)]{elliott2013finite}
C.~M. Elliott and T.~Ranner.
\newblock Finite element analysis for a coupled bulk--surface partial
  differential equation.
\newblock \emph{IMA Journal of Numerical Analysis}, 33\penalty0 (2):\penalty0
  377--402, 2013.

\bibitem[Evans(1980)]{evans1980ctc}
L.~Evans.
\newblock {A convergence theorem for a chemical diffusion-reaction system}.
\newblock \emph{Houston Journal of Mathematics}, 6\penalty0 (2):\penalty0
  259--267, 1980.

\bibitem[Garc{\'{\i}}a-Pe{\~{n}}arrubia
  et~al.(2013)Garc{\'{\i}}a-Pe{\~{n}}arrubia, G{\'{a}}lvez, and
  G{\'{a}}lvez]{garcia2013mathematical}
P.~Garc{\'{\i}}a-Pe{\~{n}}arrubia, J.~J. G{\'{a}}lvez, and J.~G{\'{a}}lvez.
\newblock Mathematical modelling and computational study of two-dimensional and
  three-dimensional dynamics of receptor{\textendash}ligand interactions in
  signalling response mechanisms.
\newblock \emph{Journal of Mathematical Biology}, 69\penalty0 (3):\penalty0
  553--582, jul 2013.
\newblock \doi{10.1007/s00285-013-0712-4}.
\newblock URL \url{http://dx.doi.org/10.1007/s00285-013-0712-4}.

\bibitem[Grisvard(2011)]{grisvard2011elliptic}
P.~Grisvard.
\newblock \emph{Elliptic problems in nonsmooth domains}, volume~69 of
  \emph{Classics in Applied Mathematics}.
\newblock SIAM, 2011.

\bibitem[Henderson et~al.(2004)Henderson, Ahrens, and
  Law]{henderson2004paraview}
A.~Henderson, J.~Ahrens, and C.~Law.
\newblock \emph{The ParaView Guide}.
\newblock Kitware Clifton Park, NY, 2004.

\bibitem[Hilhorst et~al.(1996)Hilhorst, Van Der~Hout, and
  Peletier]{hilhorst1996fast}
D.~Hilhorst, R.~Van Der~Hout, and L.~Peletier.
\newblock The fast reaction limit for a reaction-diffusion system.
\newblock \emph{Journal of Mathematical Analysis and Applications},
  199\penalty0 (2):\penalty0 349--373, 1996.

\bibitem[Hilhorst et~al.(2001)Hilhorst, Iida, Mimura, and
  Ninomiya]{hilhorst2001competition}
D.~Hilhorst, M.~Iida, M.~Mimura, and H.~Ninomiya.
\newblock A competition-diffusion system approximation to the classical
  two-phase {S}tefan problem.
\newblock \emph{Japan Journal of Industrial and Applied Mathematics},
  18\penalty0 (2):\penalty0 161--180, 2001.

\bibitem[Hilhorst et~al.(2003)Hilhorst, Mimura, and
  Sch{\"a}tzle]{hilhorst2003vanishing}
D.~Hilhorst, M.~Mimura, and R.~Sch{\"a}tzle.
\newblock Vanishing latent heat limit in a {S}tefan-like problem arising in
  biology.
\newblock \emph{Nonlinear Analysis: Real World Applications}, 4\penalty0
  (2):\penalty0 261--285, 2003.

\bibitem[Holmes et~al.(1994)Holmes, Lewis, Banks, and Veit]{holmes1994partial}
E.~E. Holmes, M.~A. Lewis, J.~E. Banks, and R.~R. Veit.
\newblock Partial differential equations in ecology: Spatial interactions and
  population dynamics.
\newblock \emph{Ecology}, 75\penalty0 (1):\penalty0 17--29, jan 1994.
\newblock \doi{10.2307/1939378}.
\newblock URL \url{http://dx.doi.org/10.2307/1939378}.

\bibitem[Hynes(1992)]{hynes1992integrins}
R.~O. Hynes.
\newblock Integrins: versatility, modulation, and signaling in cell adhesion.
\newblock \emph{Cell}, 69\penalty0 (1):\penalty0 11--25, 1992.

\bibitem[Jilkine et~al.(2007)Jilkine, Mar{\'e}e, and
  Edelstein-Keshet]{jilkine2007mathematical}
A.~Jilkine, A.~F. Mar{\'e}e, and L.~Edelstein-Keshet.
\newblock Mathematical model for spatial segregation of the rho-family gtpases
  based on inhibitory crosstalk.
\newblock \emph{Bulletin of Mathematical Biology}, 69\penalty0 (6):\penalty0
  1943--1978, 2007.

\bibitem[Lakkis et~al.(2013)Lakkis, Madzvamuse, and
  Venkataraman]{lakkis2013implicit}
O.~Lakkis, A.~Madzvamuse, and C.~Venkataraman.
\newblock Implicit--explicit timestepping with finite element approximation of
  reaction--diffusion systems on evolving domains.
\newblock \emph{SIAM Journal on Numerical Analysis}, 51\penalty0 (4):\penalty0
  2309--2330, 2013.

\bibitem[Levine and Rappel(2005)]{levine2005membrane}
H.~Levine and W.-J. Rappel.
\newblock Membrane-bound {T}uring patterns.
\newblock \emph{Physical Review E}, 72\penalty0 (6):\penalty0 061912, 2005.

\bibitem[Linderman and Lauffenburger(1986)]{LINDERMAN1986295}
J.~Linderman and D.~Lauffenburger.
\newblock Analysis of intracellular receptor/ligand sorting. calculation of
  mean surface and bulk diffusion times within a sphere.
\newblock \emph{Biophysical Journal}, 50\penalty0 (2):\penalty0 295 -- 305,
  1986.
\newblock ISSN 0006-3495.
\newblock \doi{http://dx.doi.org/10.1016/S0006-3495(86)83463-4}.
\newblock URL
  \url{http://www.sciencedirect.com/science/article/pii/S0006349586834634}.

\bibitem[Locksley et~al.(2001)Locksley, Killeen, and Lenardo]{locksley2001tnf}
R.~M. Locksley, N.~Killeen, and M.~J. Lenardo.
\newblock The tnf and tnf receptor superfamilies-integrating mammalian biology.
\newblock \emph{Cell}, 104\penalty0 (4):\penalty0 487--501, 2001.

\bibitem[Madzvamuse et~al.(2015)Madzvamuse, Chung, and
  Venkataraman]{MadChuVen14}
A.~Madzvamuse, A.~H.~W. Chung, and C.~Venkataraman.
\newblock Stability analysis and simulations of coupled bulk-surface
  reaction{\textendash}diffusion systems.
\newblock \emph{Proceedings of the Royal Society of London A: Mathematical,
  Physical and Engineering Sciences}, 471\penalty0 (2175), 2015.
\newblock ISSN 1364-5021.
\newblock \doi{10.1098/rspa.2014.0546}.

\bibitem[Marciniak-Czochra and Ptashnyk(2008)]{marciniak2008derivation}
A.~Marciniak-Czochra and M.~Ptashnyk.
\newblock Derivation of a macroscopic receptor-based model using homogenization
  techniques.
\newblock \emph{{SIAM} Journal on Mathematical Analysis}, 40\penalty0
  (1):\penalty0 215--237, jan 2008.
\newblock \doi{10.1137/050645269}.
\newblock URL \url{http://dx.doi.org/10.1137/050645269}.

\bibitem[McLennan et~al.(2012)McLennan, Dyson, Prather, Morrison, Baker, Maini,
  and Kulesa]{mclennan2012multiscale}
R.~McLennan, L.~Dyson, K.~W. Prather, J.~A. Morrison, R.~E. Baker, P.~K. Maini,
  and P.~M. Kulesa.
\newblock Multiscale mechanisms of cell migration during development: theory
  and experiment.
\newblock \emph{Development}, 139\penalty0 (16):\penalty0 2935--2944, 2012.

\bibitem[McLennan et~al.(2015{\natexlab{a}})McLennan, Schumacher, Morrison,
  Teddy, Ridenour, Box, Semerad, Li, McDowell, Kay, Maini, Baker, and
  Kulesa]{mclennan2015vegf}
R.~McLennan, L.~J. Schumacher, J.~A. Morrison, J.~M. Teddy, D.~A. Ridenour,
  A.~C. Box, C.~L. Semerad, H.~Li, W.~McDowell, D.~Kay, P.~K. Maini, R.~E.
  Baker, and P.~M. Kulesa.
\newblock {VEGF} signals induce trailblazer cell identity that drives neural
  crest migration.
\newblock \emph{Developmental Biology}, 407\penalty0 (1):\penalty0 12--25, nov
  2015{\natexlab{a}}.
\newblock \doi{10.1016/j.ydbio.2015.08.011}.
\newblock URL \url{http://dx.doi.org/10.1016/j.ydbio.2015.08.011}.

\bibitem[McLennan et~al.(2015{\natexlab{b}})McLennan, Schumacher, Morrison,
  Teddy, Ridenour, Box, Semerad, Li, McDowell, Kay, et~al.]{mclennan2015neural}
R.~McLennan, L.~J. Schumacher, J.~A. Morrison, J.~M. Teddy, D.~A. Ridenour,
  A.~C. Box, C.~L. Semerad, H.~Li, W.~McDowell, D.~Kay, et~al.
\newblock Neural crest migration is driven by a few trailblazer cells with a
  unique molecular signature narrowly confined to the invasive front.
\newblock \emph{Development}, 142\penalty0 (11):\penalty0 2014--2025,
  2015{\natexlab{b}}.

\bibitem[Morgan and Sharma(2015)]{morgan2015global}
J.~Morgan and V.~Sharma.
\newblock Global existence of solutions to reaction diffusion systems with mass
  transport type boundary conditions.
\newblock \emph{arXiv preprint arXiv:1504.00323}, 2015.

\bibitem[Mori et~al.(2008)Mori, Jilkine, and Edelstein-Keshet]{mori2008wave}
Y.~Mori, A.~Jilkine, and L.~Edelstein-Keshet.
\newblock Wave-pinning and cell polarity from a bistable reaction-diffusion
  system.
\newblock \emph{Biophysical Journal}, 94\penalty0 (9):\penalty0 3684--3697,
  2008.

\bibitem[Nochetto et~al.(2015)Nochetto, Ot{\'a}rola, and
  Salgado]{nochetto2014convergence}
R.~H. Nochetto, E.~Ot{\'a}rola, and A.~J. Salgado.
\newblock Convergence rates for the classical, thin and fractional elliptic
  obstacle problems.
\newblock \emph{Philosophical Transactions of the Royal Society A},
  373:\penalty0 20140449, 2015.

\bibitem[Persson and Strang(2004)]{persson2004simple}
P.-O. Persson and G.~Strang.
\newblock A simple mesh generator in matlab.
\newblock \emph{SIAM review}, 46\penalty0 (2):\penalty0 329--345, 2004.

\bibitem[Perthame et~al.(2014)Perthame, Quir{\'o}s, and
  V{\'a}zquez]{perthame2014hele}
B.~Perthame, F.~Quir{\'o}s, and J.~L. V{\'a}zquez.
\newblock The {H}ele--{S}haw asymptotics for mechanical models of tumor growth.
\newblock \emph{Archive for Rational Mechanics and Analysis}, 212\penalty0
  (1):\penalty0 93--127, 2014.

\bibitem[Pierre(2010)]{pierre2010global}
M.~Pierre.
\newblock Global existence in reaction-diffusion systems with control of mass:
  a survey.
\newblock \emph{Milan Journal of Mathematics}, 78\penalty0 (2):\penalty0
  417--455, aug 2010.
\newblock \doi{10.1007/s00032-010-0133-4}.
\newblock URL \url{http://dx.doi.org/10.1007/s00032-010-0133-4}.

\bibitem[R{\"a}tz and R{\"o}ger(2012)]{ratz2012turing}
A.~R{\"a}tz and M.~R{\"o}ger.
\newblock Turing instabilities in a mathematical model for signaling networks.
\newblock \emph{Journal of Mathematical Biology}, 65\penalty0 (6-7):\penalty0
  1215--1244, 2012.

\bibitem[R{\"a}tz and R{\"o}ger(2014)]{ratz2013symmetry}
A.~R{\"a}tz and M.~R{\"o}ger.
\newblock Symmetry breaking in a bulk-surface reaction-diffusion model for
  signaling networks.
\newblock \emph{Nonlinearity}, 27:\penalty0 1805--1827, 2014.

\bibitem[Rineau and Yvinec(2013)]{cgal:ry-smg-13b}
L.~Rineau and M.~Yvinec.
\newblock {3D} surface mesh generation.
\newblock In \emph{{CGAL} User and Reference Manual}. {CGAL Editorial Board},
  {4.3} edition, 2013.

\bibitem[Rodrigues(1987)]{rodrigues1987variational}
J.-F. Rodrigues.
\newblock The variational inequality approach to the one-phase {S}tefan
  problem.
\newblock \emph{Acta Applicandae Mathematica}, 8\penalty0 (1):\penalty0 1--35,
  1987.

\bibitem[{Schimperna} et~al.(2013){Schimperna}, {Segatti}, and
  {Zelik}]{2013arXiv1302.5026S}
G.~{Schimperna}, A.~{Segatti}, and S.~{Zelik}.
\newblock {On a singular heat equation with dynamic boundary conditions}.
\newblock \emph{ArXiv e-prints}, Feb. 2013.

\bibitem[Schmidt and Siebert(2005)]{schmidt2005design}
A.~Schmidt and K.~Siebert.
\newblock \emph{{Design of adaptive finite element software: The finite element
  toolbox ALBERTA}}.
\newblock Springer Verlag, 2005.

\bibitem[Simon(1986)]{simon1986compact}
J.~Simon.
\newblock Compact sets in the space {$L^p(0, T; B)$}.
\newblock \emph{Annali di Matematica Pura ed Applicata}, 146\penalty0
  (1):\penalty0 65--96, 1986.

\bibitem[V{\'a}zquez and Vitillaro(2008)]{vazquez2008heat}
J.~L. V{\'a}zquez and E.~Vitillaro.
\newblock Heat equation with dynamical boundary conditions of reactive type.
\newblock \emph{Communications in Partial Differential Equations}, 33\penalty0
  (4):\penalty0 561--612, 2008.

\bibitem[V{\'a}zquez and Vitillaro(2009)]{vazquez2009laplace}
J.~L. V{\'a}zquez and E.~Vitillaro.
\newblock On the {L}aplace equation with dynamical boundary conditions of
  reactive--diffusive type.
\newblock \emph{Journal of Mathematical Analysis and Applications},
  354\penalty0 (2):\penalty0 674--688, 2009.

\bibitem[V{\'a}zquez and Vitillaro(2011)]{vazquez2011heat}
J.~L. V{\'a}zquez and E.~Vitillaro.
\newblock Heat equation with dynamical boundary conditions of
  reactive--diffusive type.
\newblock \emph{Journal of Differential Equations}, 250\penalty0 (4):\penalty0
  2143--2161, 2011.

\end{thebibliography}

{{
\appendix
\section{Numerical investigation of the influence of the mesh-size and timestep}\label{sec:appendix}
In order to verify that the results of \S \ref{sec:sim-2d} are due to changes in the parameter $\eps$ rather than the discretisation parameters,
we now carry out the numerical experiment of \S   \ref{sec:sim-2d} on a series of different meshes with different timesteps. Specifically, we consider a coarse triangulation of the domain considered in \S \ref{sec:sim-2d} and two finer triangulations generated by refining the coarse triangulation. 
The triangulations had $376, 1369$ and $5206$ bulk degrees of freedom respectively and the corresponding surface triangulations
had $106, 212$ and $424$ degrees of freedom. Figure \ref{fig:3_triangs} shows the three meshes.

\begin{figure}
\includegraphics[trim = 0mm 0mm 0mm 0mm,  clip, width=0.32\textwidth]{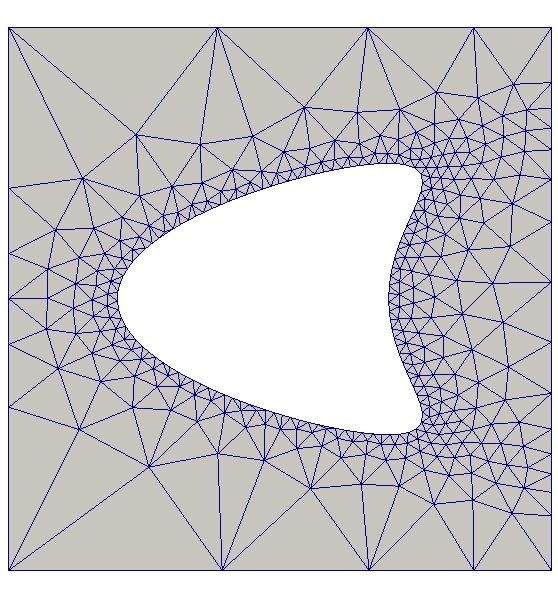}
\includegraphics[trim = 0mm 0mm 0mm 0mm,  clip, width=0.32\textwidth]{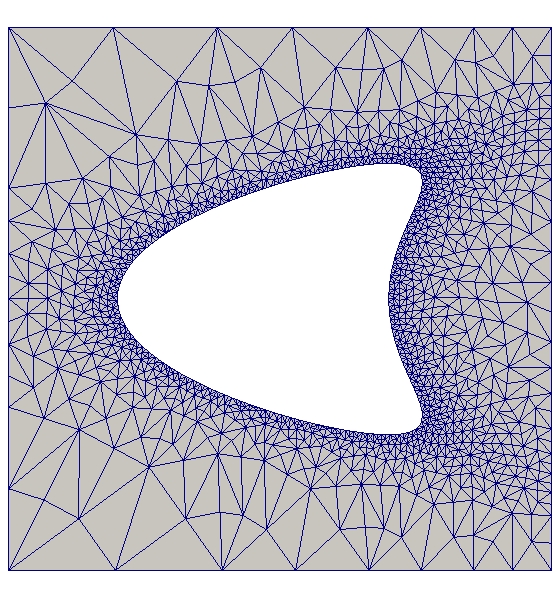}
\includegraphics[trim = 0mm 0mm 0mm 0mm,  clip, width=0.32\textwidth]{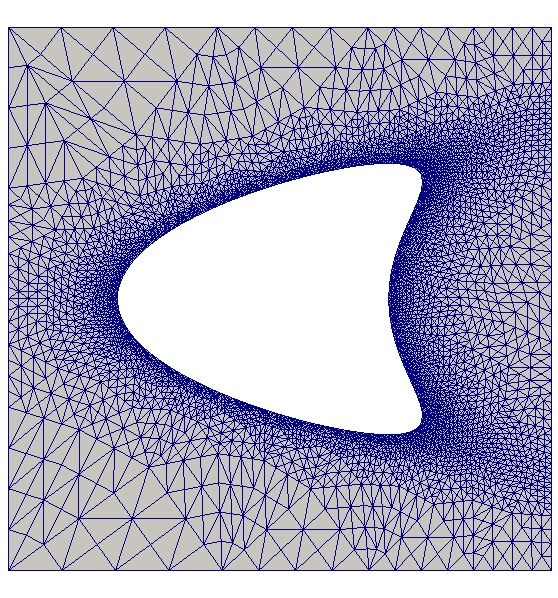}
\caption{Meshes used for the simulations of $\S$ \ref{sec:appendix}. A coarse mesh (left) and two finer meshes generated by globally bisecting  the elements of the coarse mesh twice (middle) and four times (right).}
\label{fig:3_triangs}
\end{figure}

For the simulations we solved \eqref{eqn:wf_eps_problem} with the same initial conditions and final time of \S \ref{sec:sim-2d} with $\eps=\delk=\delO=\delG=0.1$ and $0.01$. For the smaller value of $\eps=0.001$ considered in \S \ref{sec:sim-2d} the numerical scheme was unstable for significantly larger timesteps than that employed in  \S \ref{sec:sim-2d}. We set the timestep to be $2\times 10^{-6}$, $1\times 10^{-6}$ and  
$5\times 10^{-7}$ for the coarse, medium and fine mesh simulations respectively. 

\begin{figure}
\begin{minipage}{0.1\linewidth}
\includegraphics[ trim = 0mm 0mm 0mm 0mm,  clip,  width=\textwidth
]{./Figures/Paper_Figures/3d/uscale.pdf}
\end{minipage}
\begin{minipage}{0.75\linewidth}
\subfigure[][{$\eps=0.1$}]{
\includegraphics[trim = 0mm 0mm 0mm 0mm,  clip, width=\textwidth]{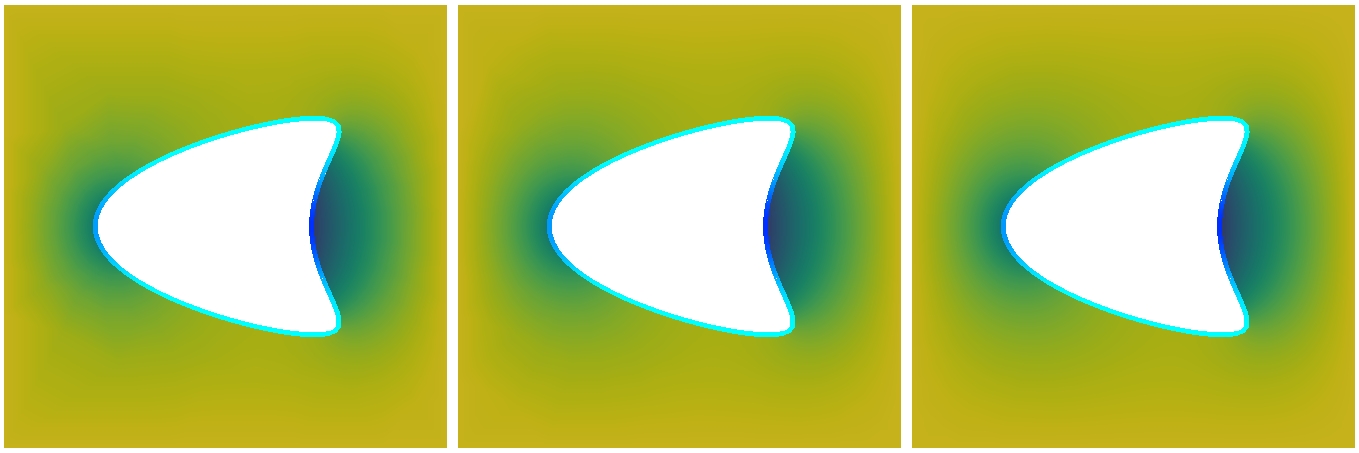}
}
\subfigure[][{$\eps=0.01$}]{
\includegraphics[trim = 0mm 0mm 0mm 0mm,  clip, width=\textwidth]{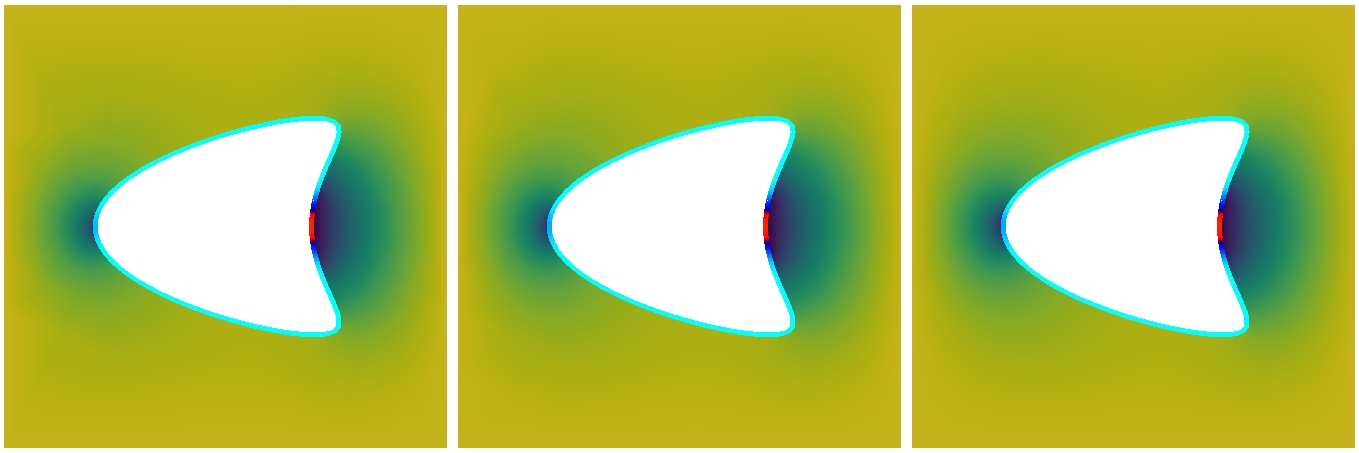}
}
\end{minipage}
\begin{minipage}{0.1\linewidth}
\includegraphics[ trim = 0mm 0mm 0mm 0mm,  clip,  width=\textwidth
]{./Figures/Paper_Figures/3d/wscale.pdf}
\end{minipage}
\caption{Snapshots of the numerical solution of $U$ and $W$ for the experiments of Appendix \ref{sec:appendix} at $t=0.5$ on the coarse mesh with large timestep (left), the twice globally refined mesh with medium timestep (middle) and the fine mesh (four times globally refined) with small timestep (right). }
\label{fig:3_sims}
\end{figure}

\begin{figure}
\centering
\subfigure[][$\Lp{2}(\O)$ norm of the difference between fine mesh $U$ and coarse mesh $U$ (blue) and fine mesh $U$ and medium mesh $U$ (purple), for $\eps=0.1$.]{
\includegraphics[trim = 0mm 0mm 0mm 0mm,  clip, width=0.45\textwidth]{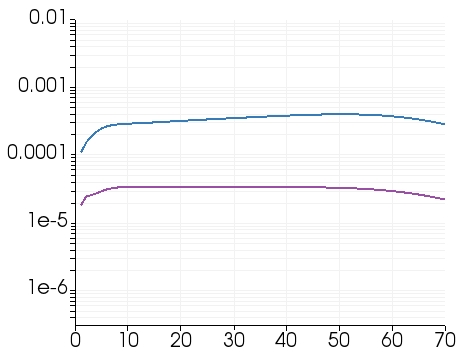}
}
\hskip 1em
\subfigure[][$\Lp{2}(\O)$ norm of the difference between fine mesh $U$ and coarse mesh $U$ (green) and fine mesh $U$ and medium mesh $U$ (red), for $\eps=0.01$.]{
\includegraphics[trim = 0mm 0mm 0mm 0mm,  clip, width=0.45\textwidth]{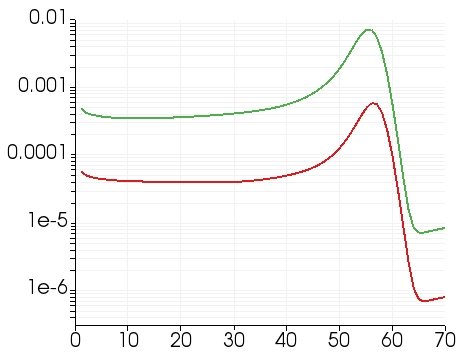}
}
\subfigure[][$\Lp{2}(\G)$ norm of the difference between fine mesh $W$ and coarse mesh $W$ (blue) and fine mesh $W$ and medium mesh $W$ (purple), for $\eps=0.1$.]{
\includegraphics[trim = 0mm 0mm 0mm 0mm,  clip, width=0.45\textwidth]{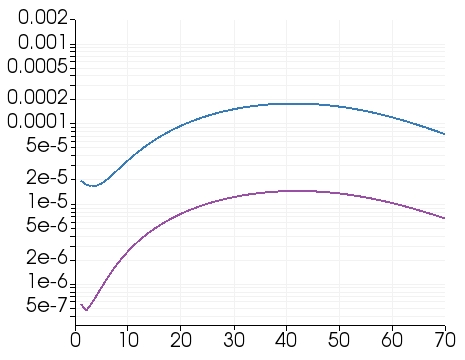}
}
\hskip 1em
\subfigure[][$\Lp{2}(\G)$ norm of the difference between fine mesh $W$ and coarse mesh $W$ (green) and fine mesh $W$ and medium mesh $W$ (red), for $\eps=0.01$.]{
\includegraphics[trim = 0mm 0mm 0mm 0mm,  clip, width=0.45\textwidth]{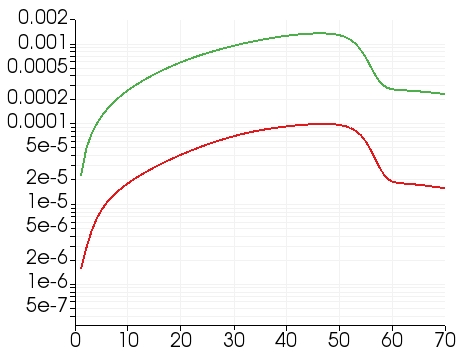}
}
\caption{Plots of the difference between the fine and coarse mesh solution and the fine and medium mesh solution.}
\label{fig:3_plots}
\end{figure}

Figure \ref{fig:3_sims} shows a snapshot of the numerical solutions at $t=0.5$ for the cases $\eps=0.1$ and $\eps=0.01$ for the three  
different numerical experiments. We observe that whilst for a fixed value of $\eps$, the qualitative features of the simulation are similar for all the different discretisation parameters under consideration, there are clear differences between the simulation results for the two different values of $\eps$.

In order to provide quantitative evidence for the convergence of the numerical solutions as the discretisation parameters are reduced, in Figure~\ref{fig:3_plots}, we plot the $\Lp{2}$ difference between the solution on the finest mesh and the solutions on the coarser meshes against time. We observe that the numerical solutions appear to converge as the discretisation parameters are refined for a fixed value of $\eps$.

}
\end{document}